\documentclass{amsart}
\usepackage{amsthm, amssymb, amsfonts, amscd}
\usepackage{graphicx}
\usepackage{stmaryrd}
\usepackage{circuitikz}
\usepackage{tikz}
\usepackage{tikz-cd}
\usepackage{bbold}
\usetikzlibrary{arrows}
\usepackage{verbatim}
\usepackage[all,arc]{xy}
\usepackage{mathrsfs, mathabx}
\usepackage{hyperref}
\usepackage{soul}
\usepackage{esint}
\usepackage[all]{xy}
\usepackage[margin=1.0in]{geometry}
\usepackage[
   open,
   openlevel=2,
   atend
 ]{bookmark}
\usepackage{graphics}
\usepackage{mathtools}
\usepackage{array}
\usepackage{multirow}

\bookmarksetup{color=blue}

\theoremstyle{definition} 
\newtheorem{thm}{Theorem}[section]
\newtheorem{cor}[thm]{Corollary}
\newtheorem{prop}[thm]{Proposition}
\newtheorem{lem}[thm]{Lemma}
\newtheorem{conj}[thm]{Conjecture}
\newtheorem{quest}[thm]{Question}
\newtheorem{defn}[thm]{Definition}
\newtheorem{exmp}[thm]{Example}
\newtheorem{rmk}[thm]{Remark}
\newtheorem{obs}[thm]{Observation}

\newcommand{\mr}{\mathrm}
\newcommand{\mf}{\mathfrak}
\newcommand{\mc}{\mathcal}

\newcommand{\GL}{\mr{GL}}
\newcommand{\Aut}{\mr{Aut}}

\newcommand{\bb}{\mathbb}
\newcommand{\bN}{\bb{N}}
\newcommand{\Z}{\bb{Z}}
\newcommand{\bQ}{\bb{Q}}
\newcommand{\bR}{\bb{R}}
\newcommand{\bP}{\bb{P}}

\newcommand{\bF}{\bb{F}}
\newcommand{\bE}{\bb{E}}
\newcommand{\fp}{\bF_p}
\newcommand{\fq}{\bF_q}
\newcommand{\zp}{\Z_p}

\newcommand{\ol}{\overline}
\newcommand{\Sur}{\mr{Sur}}
\newcommand{\M}{\mr{M}}
\newcommand{\cF}{\mc F}

\newcommand{\ld}{\lambda}
\newcommand{\al}{\alpha}

\newcommand{\sg}{\sigma}
\newcommand{\Sg}{\Sigma}
\newcommand{\dt}{\delta}
\newcommand{\cok}{\mr{cok}}
\newcommand{\Hom}{\mr{Hom}}

\newcommand{\lt}{\left}
\newcommand{\rt}{\right}
\newcommand{\rank}{\mr{rank}}

\numberwithin{equation}{section}

\usepackage[foot]{amsaddr}
\allowdisplaybreaks

\begin{document}

\title[Random $p$-adic matrices with fixed zero entries]{Random $p$-adic matrices with fixed zero entries and the Cohen--Lenstra distribution}
%    Information for the authors
\author{Dong Yeap Kang}
\address[D. Y. Kang]{Extremal Combinatorics and Probability Group, Institute for Basic Science, Daejeon 34126, Republic of Korea}
\email[D. Y. Kang]{dykang.math@ibs.re.kr}

\author{Jungin Lee}
\address[J. Lee]{Department of Mathematics, Ajou University, Suwon 16499, Republic of Korea}
\email[J. Lee]{jileemath@ajou.ac.kr}

\author{Myungjun Yu}
\address[M. Yu]{Department of Mathematics, Yonsei University, Seoul 03722, Republic of Korea}
\email[M. Yu]{mjyu@yonsei.ac.kr}

\date{}

\begin{abstract}
In this paper, we study the distribution of the cokernels of random $p$-adic matrices with fixed zero entries.
Let $X_n$ be a random $n \times n$ matrix over $\zp$ in which some entries are fixed to be zero and the other entries are i.i.d. copies of a random variable $\xi \in \zp$. We consider the minimal number of random entries of $X_n$ required for the cokernel of $X_n$ to converge to the Cohen--Lenstra distribution. 
When $\xi$ is given by the Haar measure, we prove a lower bound of the number of random entries and prove its converse-type result using random regular bipartite multigraphs. When $\xi$ is a general random variable, we determine the minimal number of random entries.
Let $M_n$ be a random $n \times n$ matrix over $\zp$ with $k$-step stairs of zeros and the other entries given by independent random $\epsilon$-balanced variables valued in $\zp$. We prove that the cokernel of $M_n$ converges to the Cohen--Lenstra distribution under a mild assumption. This extends Wood's universality theorem on random $p$-adic matrices.
\end{abstract}
\maketitle

%-----------------------------------------
%-----------------------------------------
\section{Introduction} \label{Sec1}

The Cohen--Lenstra conjecture, formulated by Cohen and Lenstra \cite{CL84}, provides a striking probabilistic model that predicts the distribution of the ideal class groups of imaginary quadratic number fields. The conjecture is based on the idea that, for a fixed prime $p$, the occurrence of any finite abelian $p$-group $G$ as the $p$-part of the ideal class group of a random imaginary quadratic field should be inversely proportional to the size of its automorphism group, $\Aut(G)$. To state it more precisely, let $K$ be an imaginary quadratic field and let $\mathrm{Cl}(K)$ be the ideal class group of $K$. 

\begin{conj}
Let $p$ be an odd prime and let $G$ be a finite abelian $p$-group. Then for every finite abelian $p$-group $G$, we have
$$
\lim_{X \to \infty} \frac{| \{K : \mathrm{Cl}(K)[p^\infty] \cong G  ~\text{and}~ \mathrm{Disc}(K) > -X \} |}{| \{ K : \mathrm{Disc}(K) > - X \} |} = \frac{1}{|\Aut(G)|}\prod_{i=1}^{\infty} (1-p^{-i}).
$$    
\end{conj}

In the above conjecture, we may include the $p=2$ case by replacing $\mathrm{Cl}(K)$ with $2\mathrm{Cl}(K)$ \cite{Ger87}. Smith proved the Cohen--Lenstra conjecture when $p=2$ \cite{Smi26}, but it remains unknown for other primes. 

Friedman and Washington \cite{FW89} considered the function field analogue of the Cohen--Lenstra conjecture. They observed that the ideal class group of an imaginary quadratic extension of $\fq(t)$ can be represented as the cokernel of a square matrix over the ring of $p$-adic integers $\zp$. They actually proved the distribution of the cokernel of a random matrix over $\zp$ converges to that of Cohen--Lenstra, thereby gave an evidence for why the Cohen--Lenstra conjecture should hold for function fields. 

\begin{thm}[Friedman--Washington \cite{FW89}]
\label{thm: Friedman-Washington}
Let $X_n$ be a Haar-random $n \times n$ matrix over $\zp$. Then for every finite abelian $p$-group $G$, we have
$$
\lim_{n \to \infty} \bP(\cok(X_n) \cong G) = \frac{1}{|\Aut(G)|}\prod_{i=1}^{\infty} (1-p^{-i}).
$$
\end{thm}

In the function field case, Ellenberg, Venkatesh, and Westerland \cite{EVW16} proved the Cohen--Lenstra conjecture for the $\ell$-parts of the class groups of quadratic function fields over $\fq(t)$ when $q \not\equiv 1 \pmod{\ell}$. Note that if $q \equiv 1 \pmod{\ell}$ (that is, $\fq$ contains an $\ell$-th root of unity), then it does not converge to the Cohen--Lenstra distribution. When $q$ ranges over prime powers such that $q \equiv 1 \pmod{\ell^n}$ but $q \not\equiv 1 \pmod{\ell^{n+1}}$ for a given positive integer $n$, Lipnowski, Sawin and Tsimerman \cite[Theorem 1.1]{LST20} determined the large $q$ limit of the distribution of the $\ell$-parts of the class groups of quadratic function fields over $\fq(t)$.

The above theorem of Friedman--Washington has been extensively generalized by Wood \cite{Woo19} as follows. We refer to Definition \ref{def7a} for the notion of \emph{$\epsilon$-balanced} random variables.

\begin{thm}[Wood \cite{Woo19}] 
\label{thm: Wood's universality theorem}
Let $X_n$ be a random $n \times n$ matrix over $\zp$ whose entries are given by independent $\epsilon$-balanced random variables in $\zp$. Then for every finite abelian $p$-group $G$, we have
$$
\lim_{n \to \infty} \bP(\cok(X_n) \cong G) = \frac{1}{|\Aut(G)|}\prod_{i=1}^{\infty} (1-p^{-i}).
$$    
\end{thm}

In the above theorem, the limiting distribution of $\cok(X_n)$ always converges to the same distribution, which is independent of the distribution of each $X_n$. Such phenomenon is called \emph{universality}. 

Theorem \ref{thm: Friedman-Washington} and \ref{thm: Wood's universality theorem} can be generalized to the distribution of the cokernels of various types of random $p$-adic matrices. The distribution of the cokernel of a random uniform symmetric matrix over $\zp$ was computed by Clancy, Kaplan, Leake, Payne and Wood \cite{CKLPW15}, and it was extended to more general random symmetric matrices over $\zp$ whose entries are $\epsilon$-balanced by Wood \cite{Woo17}. Similarly, Bhargava, Kane, Lenstra, Poonen and Rains \cite{BKLPR15} computed the distribution of the cokernel of a random uniform skew-symmetric matrix over $\zp$ and it was extended by Nguyen and Wood \cite{NW22b} to random skew-symmetric matrices with $\epsilon$-balanced entries. The second author \cite{Lee23b} determined the distribution the cokernel of a random Hermitian matrix over the ring of integers of a quadratic extension of $\bQ_p$ with $\epsilon$-balanced entries. 

The above results concern local statistics for random matrices. There are also global universality results for random matrices over $\Z$ whose entries are i.i.d. copies of $\epsilon$-balanced random integer. Nguyen and Wood \cite{NW22a} proved the universality of the distribution of the cokernels of random $n \times n$ matrices over $\Z$. The same authors \cite{NW22b} also proved the universality of the distribution of random symmetric and skew-symmetric matrices over $\Z$. 

\begin{table}[h]
\centering
\begin{tabular}{|c|c|c|c|c|}
\hline
Distribution & Non-symmetric & Symmetric & Skew-symmetric & Hermitian \\ \hline
Uniform, local & Friedman--Washington \cite{FW89} & Clancy et al. \cite{CKLPW15} & Bhargava et al. \cite{BKLPR15} & 
\multirow{2}{*}{Lee \cite{Lee23b}} \\ \cline{1-4}
$\epsilon$-balanced, local & Wood \cite{Woo19} & Wood \cite{Woo17} & Nguyen--Wood \cite{NW22b} & \\ \hline
$\epsilon$-balanced, global & Nguyen--Wood \cite{NW22a} & 
\multicolumn{2}{c|}{Nguyen--Wood \cite{NW22b}} & \\ \hline
\end{tabular}
\bigskip
\caption{Distribution of the cokernels of various types of random integral matrices}
\label{known results}
\end{table}

An $n \times n$ matrix $M$ is symmetric (resp. skew-symmetric) if $M_{i,j} = M_{j,i}$ (resp. $M_{i,j} = -M_{j,i}$) for each $1 \le i,j \le n$. As a vast generalization of random symmetric and skew-symmetric matrices, we can consider random matrices with \emph{linear relations} imposed among the entries of the matrices. One of the simplest kinds of linear relations is to fix some entries to be zero. In this paper, we study the distribution of the cokernels of random $p$-adic matrices with some entries fixed to be $0$ (and the other entries are independent). It turns out that even in this simple case, we have many interesting new questions and theorems. See Section \ref{Sub11} more details. 

Random $p$-adic (or integral) matrices have been found to be helpful for understanding random combinatorial objects. Most notably, the local and global universality of the cokernels of random symmetric matrices can be applied to the distribution of random graphs. Let $0<q<1$ and $\Gamma \in G(n, q)$ be an Erdős--Rényi random graph on $n$ vertices with each edge has a probability $q$ of existing. Wood \cite{Woo17} determined the limiting distribution of the Sylow $p$-subgroups of the sandpile group $S_{\Gamma}$ of $\Gamma$. Nguyen and Wood \cite{NW22b} proved that
$$
\lim_{n \to \infty} \bP(S_{\Gamma} \text{ is cyclic} ) = \prod_{i=1}^{\infty} \zeta(2i+1)^{-1} \approx 0.7935
$$
when $q=1/2$, which resolves a conjecture of Lorenzini \cite{Lor08}. As explained in Section \ref{Sub12}, we hope to extend these results to larger classes of random graphs using random symmetric matrices over $\zp$ (or $\Z$) with fixed zero entries.

In other direction, Kahle and Newman \cite{KN22} conjectured that when $C_n$ is a random $2$-dimensional hypertree according to the determinantal measure, then the distribution of the Sylow $p$-subgroup of $H^1(C_n)$ follows the Cohen--Lenstra distribution. The first homology group $H^1(C_n)$ can be realized as the cokernel of $I_n^T[C_n]$ where $I_n$ is a random matrix given in \cite[Section 1.1]{Mes23}. Mészáros \cite[Theorem 1.1]{Mes23} constructed a sparse random matrix model $A_n = B_n[X_n]$ which is similar to $I_n^T[C_n]$ and proved that the distribution of the Sylow $p$-subgroup of $\cok(A_n)$ converges to the Cohen--Lenstra distribution for every prime $p \ge 5$.

There are more research topics in the theory of random $p$-adic matrices. For example, Theorem \ref{thm: Friedman-Washington} and \ref{thm: Wood's universality theorem} can be generalized by concerning the joint distribution of multiple cokernels \cite{CY23, Lee23a, Lee24, NVP24, VP23}, working over a general countable Dedekind domain with finite quotients \cite{Yan23}, or relaxing the $\epsilon$-balanced condition on each entry to certain regularity condition on block matrices \cite{Gor24}.

%-----------------------------------------
\subsection{Main results} \label{Sub11}
To explain our main results, let us first set up the notation.

Let $\sg_{n,1}, \ldots, \sg_{n,n}$ be subsets of $[n] := \{ 1, 2, \ldots, n \}$ and $\Sg_n := (\sg_{n,1}, \ldots, \sg_{n,n})$. Let $X_n \in \M_n(\zp)$ be a random $n \times n$ matrix such that $(X_n)_{i,j}=0$ for $i \not\in \sg_{n,j}$ and the $(i, j)$-th entries with $i \in \sg_{n,j}$ are Haar-random and independent. (Such $X_n$ is called a \emph{Haar-random matrix supported on} $\Sg_n$.) We say $\cok(X_n)$ \emph{converges to CL} if the distribution of $\cok(X_n)$ converges to the Cohen--Lenstra distribution as $n \to \infty$.
We write
$$
|\Sg_n|:= \sum_{i=1}^n |\sigma_{n,i}|.
$$
In Theorem \ref{thm41a}, we prove that if $\cok(X_n)$ converges to CL, then
\begin{equation} \label{eq11a}
\underset{n \to \infty}{\lim} \left(\frac{|\Sigma_n|}{n} - \log_p n\right) = \infty.
\end{equation}
The converse of the above statement does not hold. For example, let $\sg_{n,1} = \varnothing$ and let $\sg_{n,2}= \cdots = \sg_{n,n} = [n]$. Obviously,
$$
\underset{n \to \infty}{\lim} \left(\frac{|\Sigma_n|}{n} - \log_p n\right) = \infty.
$$
However, the first column of $X_n$ is identically zero, so $\cok(X_n)$ does not converge to CL.  
Although the converse of Theorem \ref{thm41a} itself does not hold, we expect the following converse-type result, which is the best possible by the equation \eqref{eq11a}.

\begin{conj}[Conjecture \ref{conj42a}]
\label{conj42a-int}
For every sequence $(a_n)_{n \ge 1}$ such that $n \le a_n \le n^2$ and $\underset{n \to \infty}{\lim} (\frac{a_n}{n} - \log_p n) = \infty$, there is a sequence $(\Sg_n)_{n \ge 1}$ such that $\cok(X_n)$ converges to CL and $ | \Sg_n |=a_n$ for all $n$. 
\end{conj}

By the work of Wood \cite{Woo19}, $\cok(X_n)$ converges to CL if %and only if
$$
\bE (|\Sur(\cok(X_n), G)|) =1
$$
for every finite abelian $p$-group $G$. The following theorem gives an evidence of the above conjecture, whose proof uses random regular bipartite multigraph. 

\begin{thm}[Theorem \ref{mainthm1}]
\label{mainthm1-int}
Let $(t_n)_{n \ge 1}$ be a sequence of positive integers such that $t_n \le n$ for each $n$ and $\underset{n \to \infty}{\lim} (t_n - \log_p n) = \infty$. Then there are $\sg_{n,1}, \ldots, \sg_{n,n} \subseteq [n]$ such that $1 \le | \sg_{n,i} | \le t_n$, $\bigcup_{i=1}^{n} \sg_{n,i} = [n]$ and
\begin{equation}
\underset{n \to \infty}{\lim} \bE (|\Sur(\cok(X_n), \Z / p \Z)|)
= 1.
\end{equation}
\end{thm}

Regarding Conjecture \ref{conj42a}, we provide an example of $X_n$ with ``small'' number of random entries such that $\cok(X_n)$ converges to CL. See Section \ref{Sub44} for a concrete example of a sequence $(\Sg_n)_{n \ge 1}$ such that $|\Sg_n| \le 4n \log_p n$ and $\cok(X_n)$ converges to CL.
While this paper was close to completion, we became aware of a recent preprint by Mészáros \cite{Mes24b} which provides an example such that $|\Sg_n| = (2+o(1))n \log_p n$ ($o(1) \to 0$ as $n \to \infty$) and $\cok(X_n)$ converges to CL (see Remark \ref{rmk44c}).

If we drop the assumption that random entries are equidistributed with respect to Haar measure, then we prove the following theorem analogous to Theorem \ref{thm41a} and Conjecture \ref{conj42a}. 
A remarkable point of the following theorem is that $\cok(Y_n)$ may converge to CL even if the number of random entries is very small (i.e., $|\Sg_n| = (1+o(1))n$).

\begin{thm}[Theorem \ref{thm43a}]
\label{thm43a-int}
Let $\xi$ be a random variable taking values in $\zp$ and $Y_n \in \M_n(\zp)$ be a random matrix supported on $\Sg_n$ whose random entries are i.i.d. copies of $\xi$. 
\begin{enumerate}
    \item If $\cok(Y_n)$ converges to CL, then $\underset{n \to \infty}{\lim} (|\Sg_n| - n) = \infty$.

    \item Assume that $p$ is odd. For every sequence of integers $(a_n)_{n \ge 1}$ such that $0 \le a_n \le n^2$ and $\underset{n \to \infty}{\lim} (a_n - n) = \infty$, there is a random variable $\xi \in \zp$ and a sequence $(\Sg_n)_{n \ge 1}$ such that $\cok(Y_n)$ converges to CL and $| \Sg_n  |=a_n$ for all $n$. 
\end{enumerate}
\end{thm}

We also extend the universality result of Wood (i.e. Theorem \ref{thm: Wood's universality theorem}) to $\epsilon$-balanced random matrices over $\zp$ ``having $k$-step stairs of $0$'' (see Section \ref{universality first section} for the terminology) as follows. We emphasize, however, that it is not true that the universality holds for \emph{any} $\epsilon$-balanced random matrices with fixed zero entries. We refer to Theorem \ref{thm6a} for an example that the universality does not hold.

\begin{thm}[Theorem \ref{thm: universality main theorem}] \label{thm:univ main-int}
Let $M_n$ be an $\epsilon$-balanced random $n \times n$ matrix over $\Z_p$ having $k$-step stairs of $0$ with respect to $\alpha_n^{(i)}$ and $\beta_n^{(i)}$. Suppose that for every $1 \le i \le k$, 
$$
\underset{n \to \infty}{\lim} (n - \alpha_n^{(i)} - \beta_n^{(i)}) = \infty.
$$
Then $\cok(M_n)$ converges to CL, i.e. for every finite abelian $p$-group $G$, we have
$$
\underset{n \to \infty}{\lim}\bP(\cok(M_n) \cong G) = \frac{1}{|\Aut(G)|}\prod_{i = 1}^\infty (1- p^{-i}).
$$
\end{thm}
Note that the above theorem can be further generalized to $\epsilon$-balanced $n \times (n+t)$ matrices for a non-negative integer $t$ (Theorem \ref{thm: universality main theorem with t}).

\begin{rmk}
By Theorem \ref{thm: Wood's universality theorem}, we know that for an $\epsilon$-balanced $n\times n$ matrix over $\zp$, the distribution of the cokernel of such a matrix converges to CL as $n \to \infty$, which is referred to as a universality theorem. Surprisingly, even if we fix nearly a half of the entries to be $0$, such a universality result can still hold. Now we illustrate this. 
Let $k$ be a positive integer. For $1\le i \le k$, let
\begin{align*}
\alpha_n^{(i)} & = n - (i+1)\lfloor \frac{n}{k+2} \rfloor, \\
\beta_n^{(i)} &  = i\lfloor \frac{n}{k+2} \rfloor.
\end{align*}
Then it is clear that for every $1\le i \le k$
$$
\lim_{n \to \infty} (n - \alpha_n^{(i)} - \beta_n^{(i)}) = \infty
$$
Now consider an $\epsilon$-balanced random $n \times n$ matrix $M_n$ over $\zp$ having $k$-step stairs of $0$ with respect to $\alpha_n^{(i)}$ and $\beta_n^{(i)}$. Then by Theorem \ref{thm: universality main theorem}, we have (the distribution of) $\cok(M_n)$ converges to CL. The number of entries given by $\epsilon$-balanced variables (those not fixed to be $0$) is
$$
n^2 - \sum_{i = 1}^{k} \lfloor \frac{n}{k+2} \rfloor \left(n-(i+1)\lfloor \frac{n}{k+2} \rfloor\right) \sim n^2\left(1 -\frac{k}{k+2} + \frac{(k+3)k}{2(k+2)^2}\right) \quad (\text{as $n \to \infty$}),
$$
and the right hand side converges to $n^2/2$ as $k \to \infty$. 
\end{rmk}

\begin{quest}
Let $M_n$ be a random $n \times n$ matrix over $\zp$ with some entries fixed to be $0$ and those not fixed to be $0$ are given by (independent) random $\epsilon$-balanced variables in $\zp$. 
Let $Z_n$ denote the set of pairs $(i,j)$ such that $(M_n)_{i,j}$ is fixed to be $0$. Can we find $Z_n$ such that the distribution of $\cok(M_n)$ converges to CL and 
$$
\lim_{n \to \infty}\frac{n^2-|Z_n|}{n^2} < \frac{1}{2}
$$    
for \emph{any} choice of $\epsilon$-balanced variables for the random entries? (the above remark tells us that it is possible when $1/2$ on the right hand side is replaced by any number strictly larger than $1/2$.)
\end{quest}

%-----------------------------------------
\subsection{Future work} \label{Sub12}

In future work, we aim to study random $p$-adic (or integral) matrices with fixed zero entries in various settings. For example, let $\mu_{sym}$ be the limiting distribution of the cokernel of a Haar-random $n \times n$ symmetric matrix over $\zp$ and let $X_n$ be a random $n \times n$ symmetric matrix over $\zp$ such that some entries are fixed to be $0$ and the other upper-triangular entries are Haar-random and independent. We may ask what is the minimal number of random entries of $X_n$ required for $\cok(X_n)$ to converge to $\mu_{sym}$, as an analogue of Conjecture \ref{conj42a}. We can also try to prove analogues of Theorem \ref{thm43a} and \ref{thm: universality main theorem} for random symmetric and skew-symmetric matrices.

The study of random symmetric matrices over $\zp$ (or $\Z$) with fixed entries will be useful for extending the previously known applications of random matrices to the random graphs (\cite{Woo17}, \cite{NW22b}). Indeed, let $\Gamma$ be a random graph on $n$ vertices such that some edges can never exist and the other edges has a probability $q \in (0,1)$ of existing. Then the sandpile group $S_{\Gamma}$ is given by the cokernel of a random symmetric matrix with some entries are fixed to be $0$ and the other entries are independent and $\epsilon$-balanced.

%-----------------------------------------
\subsection{Outline of the paper} \label{Sub13}

The paper is organized as follows. In Section \ref{Sec2}, we give some preliminary results. Basic properties of the moments of the cokernels of random $p$-adic matrices are given in Section \ref{Sec3}, where we also apply them to Haar-random matrices whose zero entries are stair-shaped. 
In Section \ref{Sec4}, we present the main theorems of the paper (except Theorem \ref{thm:univ main-int}). First we provide a lower bound for the number of random entries needed to satisfy the condition that $\cok(X_n)$ converges to CL in Section \ref{Sub41}. This leads us to 
Conjecture \ref{conj42a-int} and Theorem \ref{mainthm1-int} in Section \ref{Sub42}. A proof of Theorem \ref{mainthm1-int} using random bipartite multigraphs is given in Section \ref{Sec5}. In Section \ref{Sub43}, we prove Theorem \ref{thm43a-int}.

The latter half of the paper is devoted to the proof of the universality result for random matrices having $k$-step stairs of zeros. We prove Theorem \ref{thm:univ main-int} from Section \ref{universality first section} to \ref{universality third section}, and prove its generalization (Theorem \ref{thm: universality main theorem with t}) in Section \ref{Sec10}. In Section \ref{Sec6}, we provide an example of a random matrix with fixed zero entries such that the universality result fails.

%-----------------------------------------
%-----------------------------------------
\section{Preliminaries} \label{Sec2}

\subsection{Notation and terminology} \label{Sub21}

The following notation will be used throughout the paper.

\begin{itemize}
    \item Let $p$ be a fixed prime and $\zp$ be the ring of $p$-adic integers. For a positive integer $n$, let $[n] := \lt \{ 1, 2, \ldots, n \rt \}$. 

    \item For a commutative ring $R$, let $\M_{m \times n}(R)$ be the set of $m \times n$ matrices over $R$. For a matrix $A \in \M_{m \times n}(R)$, $i \in [m]$ and $j \in [n]$, let $A_{i, j}$ be the $(i, j)$-th entry of $A$. For $A \in \M_{m \times n}(R)$, $\tau \subseteq [m]$ and $\tau' \subseteq [n]$, let $A_{\tau, \tau'}$ be the submatrix of $A$ which is obtained by choosing $i$-th rows for $i \in \tau$ and $j$-th columns for $j \in \tau'$. 

    \item Let $\sg_{n,1}, \ldots, \sg_{n,n}$ be subsets of $[n]$ and $\Sg_n := (\sg_{n,1}, \ldots, \sg_{n,n})$. Let $X_n \in \M_n(\zp)$ be a random $n \times n$ matrix such that $(X_n)_{i,j}=0$ for $i \not\in \sg_{n,j}$ and the $(i, j)$-th entries with $i \in \sg_{n,j}$ are Haar-random and independent. In this case, we say $X_n$ is a \emph{Haar-random matrix supported on} $\Sg_n$. 

    \item We say $\cok(X_n)$ \emph{converges to CL} if the distribution of $\cok(X_n)$ converges to the Cohen--Lenstra distribution as $n \to \infty$. 
\end{itemize}

\begin{figure}[ht]
\begin{equation*}
\begin{pmatrix}
* & * & 0 & 0\\ 
* & * & 0 & 0\\ 
* & 0 & 0 & 0\\ 
0 & 0 & * & *
\end{pmatrix}
\end{equation*}
\caption{A matrix $X_4 \in \M_4(\zp)$ for $\Sg_{4} = (\lt \{ 1,2,3 \rt \}, \lt \{ 1,2 \rt \}, \lt \{ 4 \rt \}, \lt \{ 4 \rt \})$}
\label{fig3}
\end{figure}

\begin{rmk} \label{rmk2n1}
Let $X_n \in \M_n(\zp)$ be a Haar-random matrix supported on $\Sg_n = (\sg_{n,1}, \ldots, \sg_{n,n})$. If $X_n$ has a row or column which is identically zero, then $\cok(X_n)$ does not converge to CL. Therefore, we may and will assume that $\sg_{n,i}$ is nonempty for each $i$ and $\bigcup_{i=1}^{n} \sg_{n,i} = [n]$.
\end{rmk}

In this section, we consider a special case where the zero entries are given by a block of size $a_n \times b_n$. More general cases will be discussed in the upcoming sections.

\begin{lem} \label{lem2a}
(\cite[Lemma 2.3]{Lee23a}) For any integers $n \geq r > 0$ and a Haar-random matrix $C \in \M_{n \times r}(\zp)$, 
$$
\bP \left ( \text{there exists } Y \in \GL_{n}(\zp) \text{ such that } YC = \begin{pmatrix}
I_r \\
O
\end{pmatrix} \right ) = \prod_{j=0}^{r-1} \left ( 1 - \frac{1}{p^{n-j}} \right ).
$$
\end{lem}

\begin{prop} \label{prop2b}
Let $(a_n)_{n \ge 1}$, $(b_n)_{n \ge 1}$ be sequences of positive integers satisfying $a_n, b_n \le n$, $\sg_{n, i} = \lt \{ a_n+1, a_n+2, \ldots, n \rt \}$ for $1 \le i \le b_n$, $\sg_{n, i} = [n]$ for $i > b_n$ and $X_n \in \M_n(\zp)$ be a Haar-random matrix supported on $\Sg_n$. Then we have
$$
\underset{n \to \infty}{\lim} \bP (\cok(X_n) \cong H) = \frac{1}{ | \Aut(H)  |} \prod_{i=1}^{\infty} (1-p^{-i})
$$
for every finite abelian $p$-group $H$ if and only if 
$$
\underset{n \to \infty}{\lim} (n-a_n-b_n) = \infty.
$$
\end{prop}

\begin{proof}
($\Leftarrow$) Assume that $n$ is sufficiently large so that $n > a_n+b_n$. Let $c_n = n-a_n$, $d_n = n-b_n$ and 
$$
X_{n} = \begin{pmatrix}
O & A_n\\ 
B_n & C_n
\end{pmatrix} \in \M_{(a_n+c_n) \times (b_n+d_n)}(\zp).
$$
For $\displaystyle 
  Y = \begin{pmatrix}
    I_{a_n} & O \\
    O & Y_1
  \end{pmatrix} \in \GL_{n}(\zp) \;\; (Y_1 \in \GL_{c_n}(\zp))$, we have
$$
\cok(X_n) \cong \cok(YX_n) = \cok \begin{pmatrix}
O & A_n\\ 
Y_1B_n & Y_1C_n
\end{pmatrix}
$$
and the matrices $A_n$, $Y_1C_n$ are independent and Haar-random for given $Y_1$ and $B_n$. Let 
$$
\widetilde{\M}_{n}(\zp) := \lt \{ \begin{pmatrix}
O & * \\ 
O & * \\ 
I_{b_n} & * 
\end{pmatrix} \in \M_{(a_n+e_n+b_n) \times (b_n+d_n)}(\zp) \rt \} \subset \M_n(\zp)
$$
($e_n=n-a_n-b_n > 0$) and $\widetilde{X}_n$ be the Haar-random matrix in $\widetilde{\M}_{n}(\zp)$. By Lemma \ref{lem2a}, we have
$$
 | \bP(\cok(X_n) \cong H) - \bP(\cok(\widetilde{X}_n) \cong H)  | \leq 1 - \prod_{j=0}^{b_n-1} \left ( 1 - \frac{1}{p^{c_n-j}} \right ).
$$

If $Z_n$ is the Haar-random matrix in $\M_{d_n}(\zp)$, we conclude that
$$
\underset{n \to \infty}{\lim} \bP(\cok(X_n) \cong H)
= \underset{n \to \infty}{\lim} \bP(\cok(\widetilde{X}_n) \cong H)
= \underset{n \to \infty}{\lim} \bP(\cok(Z_n) \cong H)
= \frac{1}{ | \Aut(H)  |} \prod_{i=1}^{\infty} (1-p^{-i}),
$$
where the first inequality is due to the fact that $\underset{n \to \infty}{\lim} (c_n-b_n) = \infty$. \\

($\Rightarrow$) Assume that $n-a_n-b_n$ does not go to infinity as $n \to \infty$. If $n-a_n-b_n < 0$, then $\det(X_n) = 0$ so $\cok(X_n)$ is infinite. Therefore we may assume that there is an integer $d \ge 0$ such that $n-a_n-b_n = d$ for infinitely many $n$. Let $(s_k)_{k \ge 1}$ be a strictly increasing sequence of positive integers such that $s_k - a_{s_k} - b_{s_k} = d$ for every $k$. Write 
$$
x_k = a_{s_k}, \, y_k = s_k - x_k, \, z_k = b_{s_k}, \,  w_k = s_k - z_k
$$ 
for simplicity. If the matrix
$$
X_{s_k} = \begin{pmatrix}
O & A_k\\ 
B_k & C_k
\end{pmatrix} \in \M_{(x_k+y_k) \times (z_k+w_k)}(\zp)
$$
has a trivial cokernel, then the $\fp$-rank of $\overline{B_k} \in \M_{y_k \times z_k}(\fp)$ should be $z_k$. Thus
\begin{align*}
\bP(\cok(X_{s_k}) = 0) & = \bP(\rank(\overline{B_k}) = z_k) \bP(\cok(X_{s_k}) = 0 \mid \rank(\overline{B_k}) = z_k) \\
& = \prod_{i=1}^{z_k}(1-p^{-y_k+i-1}) \bP(\cok(X_{s_k}) = 0 \mid \rank(\overline{B_k}) = z_k) \\
& \le (1-p^{-d-1}) \bP(\cok(X_{s_k}) = 0 \mid B_k = \begin{pmatrix}
O \\ 
I_{z_k}
\end{pmatrix}) \\
& = (1-p^{-d-1}) \bP(\cok(D_k) = 0) \\
& = (1-p^{-d-1}) \prod_{i=1}^{w_k} (1-p^{-i}), 
\end{align*}
where $D_k$ is a Haar-random matrix in $\M_{w_k}(\zp)$. Since $w_k = d+x_k \ge d+1$, we have
$$
\lim_{k \to \infty} \bP(\cok(X_{s_k}) = 0) 
\leq (1-p^{-d-1}) \prod_{i=1}^{d+1} (1-p^{-i})
< \prod_{i=1}^{\infty} (1-p^{-i}).
$$
(The last inequality holds because $\prod_{i=d+2}^{\infty} (1-p^{-i}) > 1 - \sum_{i=d+2}^{\infty} p^{-i} \ge 1 - p^{-d-1}$.)
\end{proof}

\begin{rmk} \label{rmk2c}
The ``if'' part of the above proposition is a special case of Theorem \ref{thm: universality main theorem}. Indeed, if we take $k=1$ and $M$ to be Haar-random in Theorem \ref{thm: universality main theorem}, then we recover the ``if'' part of Proposition \ref{prop2b}. However, the ``only if'' part of Proposition \ref{prop2b} may not hold if $X_n$ is a general $\epsilon$-balanced matrix. See Remark \ref{rem: counter example for the converse of universality theorem} for a discussion for this.
\end{rmk}

%-----------------------------------------
%-----------------------------------------
\section{Moments} \label{Sec3}

Let $X_n$ be a Haar-random matrix supported on $\Sg_n$. By the work of Wood \cite{Woo19}, $\cok(X_n)$ converges to CL if 
%and only if
$$
E_n(G) := \bE(\# \Sur(\cok(X_n), G))
= \sum_{F \in \Sur(R^n, G)} \bP(FX_n = 0)
= \sum_{F \in \Sur(R^n, G)} \frac{1}{ | FV_{\sg_{n,1}}  | \cdots  | FV_{\sg_{n,n}}  |}
$$
converges to $1$ as $n \to \infty$ for every finite abelian $p$-group $G$. For $G_1, \ldots, G_n \le G$, let
$$
S_{G_1, \ldots, G_n} := \lt \{ F \in \Sur(R^n, G) \mid FV_{\sg_{n,i}} = G_i \text{ for } 1 \le i \le n \rt \}
$$
and
$$
d_{G_1, \ldots, G_n} := \sum_{F \in S_{G_1, \ldots, G_n}} \frac{1}{ | FV_{\sg_{n,1}}  | \cdots  | FV_{\sg_{n,n}}  |} = \frac{ | S_{G_1, \ldots, G_n}  |}{ | G_1  | \cdots  | G_n  |}.
$$
Write $S_{n, 0} := S_{G, \ldots, G}$ and $d_{n, 0} := d_{G, \ldots, G}$ for simplicity. Then we have
\begin{equation} \label{eq3a}
E_n(G) = d_{n,0} + \sum_{\substack{(G_1, \ldots, G_n) \\ \neq (G, \ldots ,G)}}  d_{G_1, \ldots, G_n}.
\end{equation}

% Proposition 3.1
\begin{prop} \label{prop3b}
$\underset{n \to \infty}{\lim} E_n(G) = 1$ if and only if 
\begin{equation} \label{eq3b}
\underset{n \to \infty}{\lim} \sum_{\substack{(G_1, \ldots, G_n) \\ \neq (G, \ldots ,G)}}  d_{G_1, \ldots, G_n} = 0.
\end{equation}
\end{prop}

\begin{proof}
For every $(G_1, \ldots, G_n) \neq (G, \ldots, G)$, we have $d_{G_1, \ldots, G_n} \ge p \frac{ | S_{G_1, \ldots, G_n}  |}{ | G  |^n}$ so
$$
E_n(G)
= d_{n,0} + \sum_{\substack{(G_1, \ldots, G_n) \\ \neq (G, \ldots ,G)}}  d_{G_1, \ldots, G_n}
\ge \sum_{F \in S_{n, 0}} \frac{1}{ | G  |^n} + \sum_{F \not\in S_{n, 0}} \frac{p}{ | G  |^n}
= p \frac{ | \Sur(R^n, G)  |}{ | G  |^n} - (p-1)d_{n, 0}.
$$
If $\underset{n \to \infty}{\lim} E_n(G) = 1$, then the above inequality implies that $\underset{n \to \infty}{\lim} d_{n,0} = 1$. (Note that $d_{n,0} \le 1$ for every $n$.) By the equation (\ref{eq3a}), the condition (\ref{eq3b}) is satisfied. 

Conversely, assume that the condition (\ref{eq3b}) is satisfied. Since $d_{G_1, \ldots, G_n} \ge \frac{ | S_{G_1, \ldots, G_n}  |}{ | G  |^n}$, we have
$$
\underset{n \to \infty}{\lim} \sum_{\substack{(G_1, \ldots, G_n) \\ \neq (G, \ldots ,G)}}  \frac{ | S_{G_1, \ldots, G_n}  |}{ | G  |^n} = \underset{n \to \infty}{\lim} \frac{ | \Sur(R^n, G)  | -  | S_{n,0}  |}{ | G  |^n}
= \underset{n \to \infty}{\lim} (1-d_{n,0})=0
$$
so $\underset{n \to \infty}{\lim} d_{n,0} = 1$. Now the equation (\ref{eq3a}) implies that $\underset{n \to \infty}{\lim} E_n(G) = 1$.
\end{proof}

% Proposition 3.2
\begin{prop} \label{prop3c}
Let $\Sg_n = (\sg_{n,1}, \ldots, \sg_{n,n})$, $\Sg_n' = (\sg_{n,1}', \ldots, \sg_{n,n}')$ and assume that $\sg_{n,i} \subseteq \sg_{n,i}'$ for each $n \ge 1$ and $i \in [n]$. Let $X_n$ (resp. $X_n'$) be a Haar-random matrix in $\M_n(\zp)$ supported on $\Sg_n$ (resp. $\Sg_n'$). If $\underset{n \to \infty}{\lim} E_n(G)=1$, then $\underset{n \to \infty}{\lim} E_n(G)'=1$ where $E_n(G)' := \bE(\# \Sur(\cok(X_n'), G))$.
\end{prop}

\begin{proof}
Since $FV_{\sg_{n,i}} \subseteq FV_{\sg_{n,i}'}$ for each $n$ and $i$, we have $E_n(G) \ge E_n(G)'$. If $\underset{n \to \infty}{\lim} E_n(G)=1$, then we have $\underset{n \to \infty}{\limsup} E_n(G)' \le 1$. Also the inequality $E_n(G)' \ge \frac{| \Sur(R^n, G) |}{|G|^n}$ implies that $\underset{n \to \infty}{\liminf} E_n(G)' \ge 1$.
\end{proof}

% Proposition 3.2 (old version)
\begin{comment}
\begin{prop} \label{prop3c}
Let $\Sg_n = (\sg_{n,1}, \ldots, \sg_{n,n})$, $\Sg_n' = (\sg_{n,1}', \ldots, \sg_{n,n}')$ and assume that $\sg_{n,i} \subseteq \sg_{n,i}'$ for each $n \ge 1$ and $i \in [n]$. Let $X_n$ (resp. $X_n'$) be a Haar-random matrix in $\M_n(\zp)$ supported on $\Sg_n$ (resp. $\Sg_n'$). If $\cok(X_n)$ converges to CL, then $\cok(X_n')$ also converges to CL.
\end{prop}

\begin{proof}
Since $FV_{\sg_{n,i}} \subseteq FV_{\sg_{n,i}'}$ for each $n$ and $i$, we have $E_n(G) \ge E_n(G)'$. If $\cok(X_n)$ converges to CL, then $\underset{n \to \infty}{\lim} E_n(G)=1$ so $\limsup_{n \to \infty} E_n(G)' \le 1$. Since $E_n(G)' \ge \frac{ | \Sur(R^n, G)  |}{ | G  |^n}$, we have $\liminf_{n \to \infty} E_n(G)' \ge 1$. Thus $\underset{n \to \infty}{\lim} E_n(G)'=1$ so $\cok(X_n')$ converges to CL.
\end{proof}
\end{comment}

%-----------------------------------------
\subsection{An example: stair-shaped zeros}

In this section, we prove a 
% necessary and 
sufficient condition that $\cok(X_n)$ converges to CL where the zero entries of $X_n$ are stair-shaped. First we consider the case that each step has height $1$ and width $1$. 

% Stair (height 1, width 1)
\begin{thm} \label{thm3d}
Let $(t_n)_{n \ge 1}$ be a sequence of positive integers such that $t_n \le n$ for each $n$, and let
$$
\sg_{n,i} = \left\{\begin{matrix}
[t_n+(i-1)] & (1 \le i \le n-t_n)\\ 
[n] & (i \ge n-t_n+1)
\end{matrix}\right.
$$
for each $n$ and $i$. 
If $\underset{n \to \infty}{\lim} (t_n - \log_p n) = \infty$, then $\cok(X_n)$ converges to CL.
% Then $\cok(X_n)$ converges to CL if and only if $\underset{n \to \infty}{\lim} (t_n - \log_p n) = \infty$.
\end{thm}
\begin{figure}[ht]
\begin{equation*}
\begin{pmatrix}
* & * & * & * & *\\ 
* & * & * & * & *\\ 
0 & * & * & * & *\\ 
0 & 0 & * & * & *\\ 
0 & 0 & 0 & * & *
\end{pmatrix}
\end{equation*}
\caption{A matrix $X_n \in \M_n(\zp)$ for $(n, t_n)=(5,2)$}
\label{fig1}
\end{figure}

\begin{proof}
For every $F \in \Sur(R^n, G)$, we have $FV_{\sg_{n,1}} \subseteq \cdots \subseteq FV_{\sg_{n,n-t_n}} \subseteq FV_{\sg_{n,n-t_n+1}} = \cdots = FV_{\sg_{n,n}}=G$. By Proposition \ref{prop3b}, $\cok(X_n)$ converges to CL if 
% and only if
$$
\underset{n \to \infty}{\lim} \sum_{\substack{G_1 \le \cdots \le G_{n-t_n} \\ G_1 \neq G}}  d_{G_1, \ldots, G_{n-t_n}, G, \ldots, G} = 0
$$
for every finite abelian $p$-group $G$. \\

{\bf Case I:} $G=\Z/p\Z$. In this case, 
\begin{align*}
& \sum_{\substack{G_1 \le \cdots \le G_{n-t_n} \\ G_1 \neq G}}  d_{G_1, \ldots, G_{n-t_n}, G, \ldots, G} \\
= & \sum_{k=1}^{n-t_n} \frac{ | \lt \{ F \in \Sur(R^n, G) \mid FV_{\sg_{n,i}} = \lt \{ 0 \rt \} \text{ for } 1 \le i \le k, \, FV_{\sg_{n,i}} = G \text{ for } i > k \rt \} |}{1^k p^{n-k}} \\
= & \sum_{k=1}^{n-t_n} \frac{(p-1)p^{n-t_n-k}}{p^{n-k}}  \\
= & \frac{(p-1)(n-t_n)}{p^{t_n}}.
\end{align*}
It is clear that $\underset{n \to \infty}{\lim} \frac{(p-1)(n-t_n)}{p^{t_n}} = 0$ if and only if $\underset{n \to \infty}{\lim} (t_n - \log_p n) = \infty$. \\

{\bf Case II:} General case. For every $G$, it is enough to show that $\underset{n \to \infty}{\lim} (t_n - \log_p n) = \infty$ implies that
$$
\underset{n \to \infty}{\lim} \sum_{\substack{G_1 \le \cdots \le G_{n-t_n} \\ G_1 \neq G}}  d_{G_1, \ldots, G_{n-t_n}, G, \ldots, G} = 0.
$$
Let $\left| G \right| = p^m$ and consider the set 
$$
\mr{CS}_G := \lt \{ (H_1, \ldots, H_{r+1}) \mid 1 \le r \le m \text{ and } H_1 \lneq H_2 \lneq \cdots \lneq H_{r+1} = G \rt \}. 
$$
Then we have
\begin{align*}
& \sum_{\substack{G_1 \le \cdots \le G_{n-t_n} \\ G_1 \neq G}}  d_{G_1, \ldots, G_{n-t_n}, G, \ldots, G} \\
= & \sum_{(H_1, \ldots, H_{r+1}) \in \mr{CS}_G} \sum_{0 = i_0 < \cdots < i_{r} \le n-t_n} \frac{ | \lt \{ F \in \Sur(R^n, G) \mid FV_{\sg_{n,i}} = H_j \text{ if } i_{j-1} < i \le i_j \rt \} |}{ | H_1 |^{i_1} | H_2 |^{i_2-i_1} \cdots | H_{r} |^{i_{r}-i_{r-1}} |G|^{n-i_r}} \\
\le & \sum_{(H_1, \ldots, H_{r+1}) \in \mr{CS}_G} \sum_{0 = i_0 < \cdots < i_{r} \le n-t_n} 
\frac{ | H_1 |^{t_n+i_1-1} | H_2 |^{i_2-i_1} \cdots | H_{r} |^{i_{r}-i_{r-1}} |G|^{(n-t_n)-i_r+1}}{ | H_1 |^{i_1} | H_2 |^{i_2-i_1} \cdots | H_{r} |^{i_{r}-i_{r-1}} |G|^{n-i_r}} \\
= & \sum_{(H_1, \ldots, H_{r+1}) \in \mr{CS}_G} \sum_{0 = i_0 < \cdots < i_{r} \le n-t_n} \frac{ | H_1 |^{t_n-1}}{ |G|^{t_n-1}} \\
\le & \sum_{(H_1, \ldots, H_{r+1}) \in \mr{CS}_G} \binom{n-t_n}{r} \frac{1}{p^{r(t_n-1)}} \\
< & | \mr{CS}_G | \sum_{r=1}^{n-1} \binom{n-1}{r} \frac{1}{p^{r(t_n-1)}} \\
= & | \mr{CS}_G | \left ( \left ( 1 + \frac{1}{p^{t_n-1}} \right )^{n-1} - 1 \right ).
\end{align*}
If $\underset{n \to \infty}{\lim} (t_n - \log_p n) = \infty$, then $\underset{n \to \infty}{\lim} \frac{n-1}{p^{t_n-1}} = 0$ so $\underset{n \to \infty}{\lim} \left ( 1 + \frac{1}{p^{t_n-1}} \right )^{n-1} = 1$ (e.g., see Lemma \ref{lem: exp bound lem}).
\end{proof}

Next we consider the case that each step of the zero entries of $X_n$ has height $1$ and width $d \ge 2$.

% Stair (height d, width 1)
\begin{thm} \label{thm3d'}
Let $d \ge 2$ be a positive integer. Let $(t_n)_{n \ge 1}$ be a sequence of positive integers such that $t_n \le n$ for each $n$, and let
$$
\sg_{n,i} = \left\{\begin{matrix}
[t_n+(\lceil \frac{i}{d} \rceil -1)] & (1 \le i \le d(n-t_n))\\ 
[n] & (i \ge d(n-t_n)+1)
\end{matrix}\right.
$$
for each $n$ and $i$. 
If $\underset{n \to \infty}{\lim} (n- d(n-t_n)) = \infty$, then $\cok(X_n)$ converges to CL.
%Then $\cok(X_n)$ converges to CL if and only if $\underset{n \to \infty}{\lim} (n- d(n-t_n)) = \infty$.
\end{thm}

\begin{proof}
In order that $\cok(X_n)$ converges to CL, we should have $n- d(n-t_n) > 0$ for sufficiently large $n$ as otherwise $X_n$ would have a zero row for infinitely many $n$. From now on, we assume that $n- d(n-t_n) > 0$ for sufficiently large $n$. As in the proof of Theorem \ref{thm3d}, $\cok(X_n)$ converges to CL if 
%and only if
$$
\underset{n \to \infty}{\lim} \sum_{\substack{G_1 \le \cdots \le G_{n-t_n} \\ G_1 \neq G}}  d_{G_1, \ldots, G_{n-t_n}, G, \ldots, G} = 0
$$
for every finite abelian $p$-group $G$. \\

{\bf Case I:} $G=\Z/p\Z$. In this case, 
\begin{align*}
& \sum_{\substack{G_1 \le \cdots \le G_{n-t_n} \\ G_1 \neq G}}  d_{G_1, \ldots, G_{n-t_n}, G, \ldots, G} \\
= & \sum_{k=1}^{n-t_n} \frac{ | \lt \{ F \in \Sur(R^n, G) \mid FV_{\sg_{n,i}} = \lt \{ 0 \rt \} \text{ for } 1 \le i \le dk, \, FV_{\sg_{n,i}} = G \text{ for } i > dk \rt \} |}{1^{dk} p^{n-dk}} \\
= & \sum_{k=1}^{n-t_n} \frac{(p-1)p^{n-t_n-k}}{p^{n-dk}} \\
= & \sum_{k=1}^{n-t_n} \frac{(p-1)p^{(d-1)k}}{p^{t_n}} \\
= & \frac{(p-1)(p^{(d-1)(n-t_n)}-1)p^{d-1}}{p^{t_n}(p^{d-1}-1)}.
\end{align*}
It is clear that 
$$
\underset{n \to \infty}{\lim} \frac{(p-1)(p^{(d-1)(n-t_n)}-1)p^{d-1}}{p^{t_n}(p^{d-1}-1)} = 0
$$
if and only if $\underset{n \to \infty}{\lim} (t_n- (d-1)(n-t_n)) = \underset{n \to \infty}{\lim} (n - d(n-t_n)) = \infty$. \\

{\bf Case II:} General case. For every $G$, it is enough to show that $\underset{n \to \infty}{\lim} (n- d(n-t_n)) = \infty$ implies that
$$
\underset{n \to \infty}{\lim} \sum_{\substack{G_1 \le \cdots \le G_{n-t_n} \\ G_1 \neq G}}  d_{G_1, \ldots, G_{n-t_n}, G, \ldots, G} = 0.
$$
Let $\left| G \right| = p^m$ and consider the set 
$$
\mr{CS}_G := \lt \{ (H_1, \ldots, H_{r+1}) \mid 1 \le r \le m \text{ and } H_1 \lneq H_2 \lneq \cdots \lneq H_{r+1} = G \rt \}. 
$$
Then we have
\begin{align*}
& \sum_{\substack{G_1 \le \cdots \le G_{n-t_n} \\ G_1 \neq G}}  d_{G_1, \ldots, G_{n-t_n}, G, \ldots, G} \\
= & \sum_{(H_1, \ldots, H_{r+1}) \in \mr{CS}_G} \sum_{0 = i_0 < \cdots < i_{r} \le n-t_n} \frac{ | \lt \{ F \in \Sur(R^n, G) \mid FV_{\sg_{n,i}} = H_j \text{ if } di_{j-1} < i \le di_j \rt \} |}{ | H_1 |^{di_1} | H_2 |^{d(i_2-i_1)} \cdots | H_{r} |^{d(i_{r}-i_{r-1})} |G|^{n-di_r}} \\
\le & \sum_{(H_1, \ldots, H_{r+1}) \in \mr{CS}_G} \sum_{0 = i_0 < \cdots < i_{r} \le n-t_n} 
\frac{ | H_1 |^{t_n+i_1-1} | H_2 |^{i_2-i_1} \cdots | H_{r} |^{i_{r}-i_{r-1}} |G|^{(n-t_n)-i_r+1}}{ | H_1 |^{di_1} | H_2 |^{d(i_2-i_1)} \cdots | H_{r} |^{d(i_{r}-i_{r-1})} |G|^{n-di_r}} \\
= & \sum_{(H_1, \ldots, H_{r+1}) \in \mr{CS}_G} \sum_{0 = i_0 < \cdots < i_{r} \le n-t_n} 
\frac{ | H_1 |^{t_n-1} }{ | H_1 |^{(d-1)i_1} | H_2 |^{(d-1)(i_2-i_1)} \cdots | H_{r} |^{(d-1)(i_{r}-i_{r-1})} |G|^{t_n - (d-1)i_r-1}} \\
\le & | \mr{CS}_G | \sum_{r=1}^{m} \sum_{0 = i_0 < \cdots < i_{r} \le n-t_n} 
\frac{1}{p^{(d-1)(i_r-i_1)}\cdot p^{r(t_n - (d-1)i_r-1)}},
\end{align*}
where the last inequality follows from the fact that $ | H_j | \ge p | H_1 |$ for each $j \ge 2$ and $p^r | H_1 | \le |G|$. Now it is enough to show that for every $1 \le r \le m$, 
$$
\underset{n \to \infty}{\lim} \sum_{0 = i_0 < \cdots < i_{r} \le n-t_n} 
\frac{1}{p^{(d-1)(i_r-i_1)}\cdot p^{r(d-1)(n-t_n-i_r)} \cdot p^{r(t_n - (d-1)(n-t_n)-1)}} = 0.
$$
Since we have $\underset{n \to \infty}{\lim} (t_n -(d-1)(n-t_n) -1) = \infty$ by the assumption, it suffices to show that the sum
$$
\sum_{0 = i_0 < \cdots < i_{r} \le n-t_n} 
\frac{1}{p^{(d-1)(i_r-i_1)}\cdot p^{r(d-1)(n-t_n-i_r)}} 
$$
is bounded above. We have
\begin{align*}
& \sum_{0 = i_0 < \cdots < i_{r} \le n-t_n} 
\frac{1}{p^{(d-1)(i_r-i_1)}\cdot p^{r(d-1)(n-t_n-i_r)}} \\
\le & \sum_{0 = i_0 < \cdots < i_{r} \le n-t_n} 
\frac{1}{p^{(d-1)(n-t_n-i_1)}} \\
= & \sum_{i_1=1}^{n-t_n-r+1} \binom{n-t_n-i_1}{r-1} \frac{1}{p^{(d-1)(n-t_n-i_1)}} \\
\le & \sum_{k=r-1}^{\infty} \frac{k^{r-1}}{p^{(d-1)k}}
\end{align*}
so the sum is bounded above by a constant which is independent of $n$.
\end{proof}

%-----------------------------------------
%-----------------------------------------
\section{Minimal number of random entries} \label{Sec4}

%-----------------------------------------
\subsection{A lower bound of \texorpdfstring{$| \Sg_n |$}{Sigman}} \label{Sub41}

Let $X_n \in \M_n(\zp)$ be a Haar-random matrix supported on $\Sg_n$ and assume that $\cok(X_n)$ converges to CL. Since the probability that $X_n$ does not have a column whose entries are all divisible by $p$ is
$$ 
\prod_{i=1}^{n} (1 - p^{-| \sg_{n,i} |}) \le (1 - p^{-\frac{| \Sg_n |}{n}})^n,
$$
we have
$$
\liminf_{n \to \infty} (1 - p^{-\frac{| \Sg_n |}{n}})^n 
\ge \underset{n \to \infty}{\lim} \bP(\cok(X_n) = 0) = \prod_{k=1}^{\infty} (1-p^{-k})
$$
and
\begin{equation} \label{eq41a}
\liminf_{n \to \infty} (\frac{| \Sg_n |}{n} - \log_p n) \ge - \log_p \log \prod_{k=1}^{\infty} (1-p^{-k})^{-1}.
\end{equation}
In fact, we have the following stronger result. 

\begin{thm} \label{thm41a}
If $\cok(X_n)$ converges to CL, then $\displaystyle \underset{n \to \infty}{\lim} (\frac{| \Sg_n |}{n} - \log_p n) = \infty$.
\end{thm}

Let $\ol{X_n} \in \M_n(\fp)$ be the reduction of $X_n$ modulo $p$. If $\cok(X_n)$ converges to CL, then \cite[Theorem 6.3]{CL84} implies that for every nonnegative integer $m$, we have
\begin{equation} \label{eq41b}
\underset{n \to \infty}{\lim} \bP(\dim \ker \ol{X_n} = m) 
= \nu_p(m) := p^{-m^2} \prod_{k=1}^{m} (1-p^{-k})^{-2} \prod_{k=1}^{\infty} (1-p^{-k}).
\end{equation}

We prove Theorem \ref{thm41a} by showing that if $\underset{n \to \infty}{\liminf} \left ( \frac{| \Sg_n |}{n} - \log_p n \right ) < \infty$, then the distribution of $\dim \ker \ol{X_n}$ has heavier tail than the distribution $\nu_p$. Note that a similar argument can be found in the proofs of \cite[Theorem 2]{Mes24a} and \cite[Theorem 3]{Mes24b}.

For $x_1, \ldots, x_n \in [0,1]$ and $0 \le m \le n$, define
$$
f_{m,n}(x_1, \ldots, x_n) := \sum_{S \subset [n], \, \left| S \right| \ge m} \left ( \prod_{j \in S} x_j \right )\left ( \prod_{j \in [n] \setminus S} (1-x_j) \right ).
$$

\begin{lem} \label{lem41b}
Let $2 \le m \le n-2$, $t_1, \ldots, t_n \in [0,1]$ and $t := (t_1 \cdots t_n)^{1/n}$. Then
$$
f_{m,n}(t_1, \ldots, t_n) \ge f_{m,n}(t, \ldots, t).
$$
\end{lem}

\begin{proof}
A simple computation gives
\begin{equation*}
\begin{split}
f_{m,n}(x_1, \ldots, x_n) & = x_1x_2 f_{m-2, n-2}(x_3, \ldots, x_n) \\
& + (x_1(1-x_2) + x_2(1-x_1))f_{m-1, n-2}(x_3, \ldots, x_n) \\
& + (1-x_1)(1-x_2)f_{m, n-2}(x_3, \ldots, x_n) \\
& = x_1x_2A(x_3, \ldots, x_n) + (x_1+x_2)B(x_3, \ldots, x_n) + C(x_3, \ldots, x_n)
\end{split}    
\end{equation*}
for some polynomials $A$, $B$ and $C$. Since $B(x_3, \ldots, x_n) = f_{m-1, n-2}(x_3, \ldots, x_n) - f_{m, n-2}(x_3, \ldots, x_n) \ge 0$ for every $x_3, \ldots, x_n \in [0,1]$, we have
\begin{equation} \label{eq41c}
f_{m,n}(t_1, t_2, t_3, \ldots, t_n) \ge f_{m,n}(\sqrt{t_1t_2}, \sqrt{t_1t_2}, t_3, \ldots, t_n).
\end{equation}
Now we define a sequence of $n$-tuples of real numbers $(t(m)_1, \ldots, t(m)_n)$ ($m \ge 0$) as follows.
\begin{enumerate}
    \item $(t(0)_1, \ldots, t(0)_n)=(t_1, \ldots, t_n)$.

    \item For a given $(t(m)_1, \ldots, t(m)_n)$, choose any $i, j \in [n]$ such that $t(m)_i = \max(t(m)_1, \ldots, t(m)_n)$, $t(m)_j = \min(t(m)_1, \ldots, t(m)_n)$ and $i \neq j$. Define $(t(m+1)_1, \ldots, t(n+1)_n)$ by $t(m+1)_i = t(m+1)_j = \sqrt{t(m)_it(m)_j}$ and $t(m+1)_k = t(m)_k$ for every $k \in [n] \setminus \{ i,j \}$.
\end{enumerate}
Then we have $\underset{m \to \infty}{\lim} t(m)_k = t$ for every $1 \le k \le n$. By the continuity of $f_{m,n}(x_1, \ldots, x_n)$ and the inequality \eqref{eq41c}, we have $f_{m,n}(t_1, \ldots, t_n) \ge f_{m,n}(t, \ldots, t)$.
\end{proof}

\begin{proof}[Proof of Theorem \ref{thm41a}]
Suppose that $\cok(X_n)$ converges to CL. By \eqref{eq41a}, there is a constant $c_1$ such that $\log_p n + c_1 \le \frac{| \Sg_n |}{n}$ for all sufficiently large $n$. Now assume that there is a constant $c_2 \ge 0$ such that $\frac{| \Sg_n |}{n} \le \log_p n + c_2$ for infinitely many $n$. 

Since the probability that the $i$-th column of $\ol{X_n}$ is zero is $u_i := p^{-| \sg_{n,i} |} \in [0,1]$, we have
$$
\bP(\dim \ker \ol{X_n} \ge m) \ge \bP(\ol{X_n} \text{ has at least } m \text{ zero columns}) = f_{m,n}(u_1, \ldots, u_n).
$$
By Lemma \ref{lem41b}, we have $f_{m,n}(u_1, \ldots, u_n) \ge f_{m,n} (C_n, \ldots, C_n)$ for $C_n = (u_1 \cdots u_n)^{1/n} > 0$. Note that $C_n = p^{- \frac{| \Sg_n |}{n}}$ so $\frac{p^{-c_2}}{n} \le C_n \le \frac{p^{-c_1}}{n}$ for infinitely many $n$. If $\frac{p^{-c_2}}{n} \le C_n \le \frac{p^{-c_1}}{n}$, then
\begin{equation*}
\begin{split}
f_{m,n} (C_n, \ldots, C_n) & \ge \binom{n}{m} C_n^{m} (1-C_n)^{n-m} \\
& \ge \left ( \frac{n}{m} \right )^m \left ( \frac{p^{-c_2}}{n} \right )^m \left ( 1 - \frac{p^{-c_1}}{n} \right )^{n} \\
& > \frac{p^{-c_2 m}}{m^m} \frac{1}{3^{p^{-c_1}}}
\end{split}
\end{equation*}
when $n$ is sufficiently large. This shows that for every $m \in \bN$, there are infinitely many $n \in \bN$ such that
$$
\bP(\dim \ker \ol{X_n} \ge m) \ge \frac{p^{-c_2 m}}{m^m} \frac{1}{3^{p^{-c_1}}}.
$$
This contradicts the equation \eqref{eq41b}, so there is no constant $c_2$ such that $\frac{| \Sg_n |}{n} \le \log_p n + c_2$ for infinitely many $n$. We conclude that $\displaystyle \underset{n \to \infty}{\lim} (\frac{ | \Sg_n |}{n} - \log_p n) = \infty$.
\end{proof}

\subsection{Uniform case} \label{Sub42}

Assume that $X_n$ a Haar-random matrix supported on $\Sg_n$ and let $ | \Sg_n   | := \sum_{i=1}^{n}  | \sg_{n,i}   |$. In this section, we study the lower bound of $ | \Sg_n   |$ that $\cok(X_n)$ converges to CL. 

\begin{conj} \label{conj42a}
For every sequence $(a_n)_{n \ge 1}$ such that $n \le a_n \le n^2$ and $\underset{n \to \infty}{\lim} (\frac{a_n}{n} - \log_p n) = \infty$, there is a sequence $(\Sg_n)_{n \ge 1}$ such that $\cok(X_n)$ converges to CL and $ | \Sg_n   |=a_n$ for all $n$. 
\end{conj}

By Theorem \ref{thm41a}, Conjecture \ref{conj42a} is the best possible result that we can expect. 

\begin{rmk} \label{rmk42b}
Let $t_n = \left \lfloor \frac{a_n}{n} \right \rfloor$. Then $n \le nt_n \le a_n \le n^2$ and $\underset{n \to \infty}{\lim} (t_n - \log_p n) = \infty$. By Proposition \ref{prop3c}, in order to prove Conjecture \ref{conj42a} it is enough to show that for every sequence $(t_n)_{n \ge 1}$ of positive integers such that $t_n \le n$ for each $n$ and $\underset{n \to \infty}{\lim} (t_n - \log_p n) = \infty$, there are subsets $\sg_{n,i} \subseteq [n]$ such that $| \sg_{n,i} | \le t_n$ and 
$$
\underset{n \to \infty}{\lim} E_n(G) = \underset{n \to \infty}{\lim} \bE(\# \Sur(\cok(X_n), G)) = 1
$$
for every finite abelian $p$-group $G$. 
\end{rmk}

Although we did not prove Conjecture \ref{conj42a} in general, we obtain the following partial result which supports Conjecture \ref{conj42a}. The problem of computing the moment for $G = \Z/p\Z$ could be translated into a graph theory problem, allowing us to utilize tools from graph theory. However, when $G$ is a larger group, this translation seems to be difficult or even impossible. In fact, the proof of Theorem \ref{mainthm1} is already quite complicated that the proof extends throughout the entirety of Section \ref{Sec5}.

% Partial result for Conjecture \ref{conj42a}
\begin{thm} \label{mainthm1}
Let $(t_n)_{n \ge 1}$ be a sequence of positive integers such that $t_n \le n$ for each $n$ and $\underset{n \to \infty}{\lim} (t_n - \log_p n) = \infty$. Then there are $\sg_{n,1}, \ldots, \sg_{n,n} \subseteq [n]$ such that $1 \le | \sg_{n,i} | \le t_n$, $\bigcup_{i=1}^{n} \sg_{n,i} = [n]$ and
\begin{equation} \label{eq42b}
\underset{n \to \infty}{\lim} E_n(\Z/p\Z) 
= 1.
\end{equation}
\end{thm}

A proof of Theorem \ref{mainthm1} is given in Section \ref{Sec5}. We reformulate the theorem in terms of random regular bipartite multigraphs (Theorem \ref{thm5a}) and prove it by combinatorial arguments.

%-----------------------------------------
\subsection{General i.i.d. case} \label{Sub43}

% identical random entries
\begin{thm} \label{thm43a}
Let $\xi$ be a random variable taking values in $\zp$ and $Y_n \in \M_n(\zp)$ be a random matrix supported on $\Sg_n$ whose random entries are i.i.d. copies of $\xi$. 
\begin{enumerate}
    \item If $\cok(Y_n)$ converges to CL, then $\underset{n \to \infty}{\lim} (|\Sg_n| - n) = \infty$.

    \item Assume that $p$ is odd. For every sequence of integers $(a_n)_{n \ge 1}$ such that $0 \le a_n \le n^2$ and $\underset{n \to \infty}{\lim} (a_n - n) = \infty$, there is a random variable $\xi \in \zp$ and a sequence $(\Sg_n)_{n \ge 1}$ such that $\cok(Y_n)$ converges to CL and $| \Sg_n  |=a_n$ for all $n$. 
\end{enumerate}
\end{thm}

\begin{proof}
For (1), assume that $\cok(Y_n)$ converges to CL and $|\Sg_n| - n$ does not go to infinity as $n \to \infty$. Then there exists a constant $c \in \Z_{\ge 0}$ and a sequence $n_1 < n_2 < \cdots$ such that $|\Sg_{n_k}| - n_k \leq c$ for each $k$. Since $\underset{n \to \infty}{\lim}\bP( \det(Y_n) \neq 0) = 1$, we may assume that $i \in \sg_{n, i}$ for each $i \in [n]$ for a sufficiently large $n$ by permuting rows and columns. (Note that permutations of rows and columns do not change the cokernel of a matrix, i.e. if $A \in \M_n(\zp)$ and $P, Q \in \GL_n(\zp)$, then $\cok(A) \cong \cok(PAQ)$.) The inequality $|\Sg_{n_k}| - n_k \leq c$ implies that 
    $$
    |\lt \{ i \in [n] : i \in \sg_{n_k, j} \text{ for some } j \neq i \rt \}\cup \lt \{ j \in [n] : i \in \sg_{n_k, j} \text{ for some } i \neq j \rt \} | \leq 2c
    $$
    for a sufficiently large $k$. Thus we may assume that for every sufficiently large $k$, 
    $$
    \Sg_{n_k} = (\sg_{n_k, 1}, \ldots, \sg_{n_k, 2c}, \lt \{ 2c+1 \rt \}, \lt \{ 2c+2 \rt \}, \ldots, \lt \{ n \rt \})
    $$
    for some $\sg_{n_k, 1}, \ldots, \sg_{n_k, 2c} \subseteq [2c]$ by permuting rows and columns.

    Now let $\bP(\xi \equiv 0 \pmod{p}) = \alpha$. If $\alpha>0$, then $\lim_{k \to \infty} \bP(\cok(Y_{n_k})=0) \le \lim_{k \to \infty} (1 - \alpha)^{n_k-2c} = 0$ so $\cok(Y_{n_k})$ does not converge to CL as $k \to \infty$. If $\alpha=0$, then we have $\cok(Y_{n_k}) \cong \cok(Y'_{n_k})$ where $Y'_{n_k} \in \M_{2c}(\zp)$ is the upper left $2c \times 2c$ submatrix of $Y_{n_k}$. Since the $p$-rank of $\cok(Y'_{n_k})$ is at most $2c$, $\cok(Y'_{n_k})$ does not converge to CL as $k \to \infty$. This implies that $\cok(Y_{n_k})$ does not converge to CL as $k \to \infty$, which is a contradiction.

Now we prove (2). Let $\xi$ be a random variable given by $\bP(\xi = 1) = \bP(\xi = -1) = \frac{1}{2}$. Assume that $n$ is sufficiently large so that $a_n \ge n$. Let $d_n \le n$ be the largest positive integer such that $a_n \ge n+d_n(d_n-1)$. Then $a_n < n+(d_n+1)d_n = n+d_n(d_n-1)+2d_n$. Now we choose $\Sg_n$ as follows. (For $n$ such that $a_n < n$, choose arbitrary $\Sg_n$ such that $| \Sg_n |=a_n$.)

    \begin{enumerate}
        \item $d_n \le n-2$ : Let $\sg_{n,i} =[d_n]$ for $1 \le i \le d_n$, $\sg_{n,i} = \lt \{ i \rt \}$ for $i \ge d_n+3$, $\sg_{n, d_n+1} = \tau_1 \cup \lt \{ d_n+1 \rt \}$ and $\sg_{n, d_n+2} = \tau_2 \cup \lt \{ d_n+2 \rt \}$ for any $\tau_1, \tau_2 \subset [d_n]$ such that $| \tau_1 | + | \tau_2 | = a_n-n-d_n(d_n-1)$. Then $| \Sg_n |=a_n$ and $\cok(Y_n) \cong \cok(Z_n)$ where $Z_n$ is the upper left $d_n \times d_n$ submatrix of $Y_n$.

        \item $d_n = n-1$, $a_n \le n^2-n+1$ : Choose arbitrary $[n-1] \subseteq \sg_{n,1}, \ldots, \sg_{n,n-1} \subseteq [n]$ such that $\sum_{i=1}^{n-1} |\sg_{n,i} | = a_n-1$ and $\sg_{n,n}=\lt \{ n \rt \}$. Then $\cok(Y_n) \cong \cok(Z_n)$ where $Z_n$ is the upper left $(n-1) \times (n-1)$ submatrix of $Y_n$.

        \item $n^2-n+2\le a_n \le n^2$ : Let $\sg_{n,i} = [n-1]$ for $1\le i \le n^2-a_n$ and $\sg_{n,i} = [n]$ for $n^2-a_n+1 \le i \le n$. 

        Let $\mc{R} = (r_1 \; r_2 \; \cdots \; r_n)$ be a row vector of length $n$ over $\zp$ such that $r_i=0$ for $1 \le i \le n^2-a_n$ and $r_i \in \left\{ 1, -1 \right\}$ for $n^2-a_n+1 \le i \le n$. Let $\mc{C} = (c_1 \; c_2 \; \cdots \; c_n)^T$ be a column vector of length $n$ over $\zp$ such that $c_i \in \left\{ 1, -1 \right\}$ for $1 \le i \le n$ and $r_n=c_n$. Let $Y_n(\mc{R}, \mc{C})$ be a random $n \times n$ matrix whose $n$-th row and $n$-th column are fixed to be $\mc{R}$ and $\mc{C}$ respectively, and the remaining entries are i.i.d. copies of $\xi$. Let $Z_n'$ be a random $n \times n$ matrix over $\zp$ whose $i$-th column is defined by
        $$ (Z_n')_{*i} = 
        \begin{cases}
        Y_n(\mc{R}, \mc{C})_{*i} & \text{ if } 1 \le i \le n^2-a_n \text{ or } i =n \\
        Y_n(\mc{R}, \mc{C})_{*i} + Y_n(\mc{R}, \mc{C})_{*n}& \text{ if } n^2-a_n+1 \le i \le n-1 \text{ and } Y_n(\mc{R}, \mc{C})_{ni} = -Y_n(\mc{R}, \mc{C})_{nn} \\
        Y_n(\mc{R}, \mc{C})_{*i} - Y_n(\mc{R}, \mc{C})_{*n}& \text{ if } n^2-a_n+1 \le i \le n-1 \text{ and } Y_n(\mc{R}, \mc{C})_{ni} = Y_n(\mc{R}, \mc{C})_{nn}.
        \end{cases}
        $$
       (For a matrix $A$, denote the $i$-th column of $A$ by $A_{*i}$.) Here we do elementary column operations on $Y_n(\mc{R}, \mc{C})$ to make the first $n-1$ entries of $n$-th row zero. 
       Indeed, we have that $(Z_n')_{nj} = 0$ for all $1 \le j \le n-1$ and $(Z_n')_{nn} = 1$ or $-1$.
       Now let $Z_n$ be the submatrix of $Z_n'$ obtained by choosing the first $n-1$ columns and rows. Then we have
       $$
       \cok(Y_n(\mc{R}, \mc{C})) \cong \cok(Z_n') \cong  \cok(Z_n). 
       $$
       Let $S$ be the set of all possible pairs $(\mc{R}, \mc{C})$. Then we have $\bP(((Y_n)_{n*}, (Y_n)_{*n}) = (\mc{R}, \mc{C})) = \left| S \right|^{-1}$ for every $(\mc{R}, \mc{C}) \in S$. Thus for every finite abelian $p$-group $G$, we have
       $$
       \bP(\cok(Y_n) \cong G) = \sum_{(\mc{R}, \mc{C}) \in S} \frac{1}{\left| S \right|}\bP(\cok(Y_n(\mc{R}, \mc{C})) \cong G)
       = \bP(\cok(Z_n) \cong G).
       $$
       \end{enumerate}
    In each case, $\cok(Z_n)$ converges to CL by \cite[Theorem 1.2]{Woo19}, so does $\cok(Y_n)$. \qedhere  
\end{proof}

%-----------------------------------------
\subsection{An example with small number of random entries} \label{Sub44}

For each positive integer $n \ge p$, let $t_n = \left \lfloor 2 \log_p n \right \rfloor - 1$, $k_n = \left \lfloor \frac{n}{t_n} \right \rfloor$, $u_n = (t_n+1)k_n - n$ and $v_n = k_n - u_n = n-t_nk_n$. For $1 \le i \le j \le n$, denote $[i, j] = \lt \{ i, i+1, \ldots, j \rt \}$. Let $\tau_k = [(k-1)t_n+1, kt_n]$ for $1 \le k \le u_n$ and $\tau_k = [(k-1)(t_n+1)-u_n+1, k(t_n+1) - u_n]$ for $u_n+1 \le k \le k_n$. Note that $n = u_nt_n + v_n(t_n+1)$ and $[n] = \bigsqcup_{k=1}^{k_n} \tau_k$. 

For $1 \le n<p$, choose $\sg_{n,1} = \cdots = \sg_{n,n} = \varnothing$. For $n \ge p$, define $\Sg_n = (\sg_{n,1}, \ldots, \sg_{n,n})$ as follows. 
\begin{enumerate}
    \item For each $1 \le q \le u_n$ and $1 \le r \le t_n$, let $\sg_{n, (q-1)t_n+r} = \tau_q \cup \tau_{q+1}$.

    \item For each $u_n+1 \le q \le k_n$ and $1 \le r \le t_n+1$, let $\sg_{n, (q-1)(t_n+1)-u_n+r} = \tau_q \cup \tau_{q+1}$. (Here we use the convention that $\tau_{k_n+1} = \tau_1$.)
\end{enumerate}

\begin{figure}[ht]
\begin{equation*}
\begin{pmatrix}
* & * & 0 & 0 & * & * & *\\ 
* & * & 0 & 0 & * & * & *\\ 
* & * & * & * & 0 & 0 & 0\\ 
* & * & * & * & 0 & 0 & 0\\ 
0 & 0 & * & * & * & * & *\\ 
0 & 0 & * & * & * & * & *\\ 
0 & 0 & * & * & * & * & *
\end{pmatrix}
\end{equation*}
\caption{A matrix $X_n \in \M_n(\zp)$ for $(n, p, t_n, k_n, u_n, v_n)=(7, 5, 2, 3, 2, 1)$}
\label{fig4}
\end{figure}

Let $X_n$ be a Haar-random $n \times n$ matrix over $\zp$ supported on $\Sg_n$. For each $n \in \bN$, let
$$
S'_{H_1, \ldots, H_{k_n}} := \lt \{ F \in \Sur(R^n, G) \mid FV_{\tau_q \cup \tau_{q+1}} = H_q \text{ for } 1 \le q \le k_n \rt \}
$$
and
$$
d'_{H_1, \ldots, H_{k_n}} := \sum_{F \in S'_{H_1, \ldots, H_{k_n}}} \frac{1}{| FV_{\sg_{n,1}}| \ldots | FV_{\sg_{n,n}} |} = \frac{| S'_{H_1, \ldots, H_{k_n}} |}{(| H_1 | \cdots | H_{u_n} |)^{t_n} (| H_{u_n+1} | \cdots | H_{k_n} |)^{t_n+1}}.
$$
By Proposition \ref{prop3b}, we have $\underset{n \to \infty}{\lim} E_n(G)=1$ if and only if
\begin{equation*}
\underset{n \to \infty}{\lim} \sum_{\substack{(H_1 \cdots, H_{k_n}) \\ \neq (G, \ldots ,G)}}  d'_{H_1, \ldots, H_{k_n}} = 0.
\end{equation*}
For an element $F \in S'_{H_1, \ldots, H_{k_n}}$, we have $FV_{\tau_q} \in H_q \cap H_{q-1}$ (denote $H_{0} := H_{k_n}$) for each $q$ so 
$$
| S'_{H_1, \ldots, H_{k_n}} | \leq 
( \prod_{j=0}^{u_n-1} | H_{j} \cap H_{j+1} | )^{t_n}
( \prod_{j=u_n}^{k_n-1} | H_{j} \cap H_{j+1} | )^{t_n+1}
$$
and
\begin{align*}
d'_{H_1, \ldots, H_{k_n}}
& \le \frac{( \prod_{j=0}^{u_n-1} | H_{j} \cap H_{j+1} | )^{t_n}
( \prod_{j=u_n}^{k_n-1} | H_{j} \cap H_{j+1} | )^{t_n+1}}{(| H_1 | \cdots | H_{u_n} |)^{t_n} (| H_{u_n+1}| \cdots | H_{k_n}|)^{t_n+1}} \\
& \leq \left ( \frac{| H_0 \cap H_1 | | H_1 \cap H_2 | \cdots | H_{k_n-1} \cap H_{k_n} |}{ | H_1  |  | H_2  | \cdots  | H_{k_n}  |} \right )^{t_n} \\
& \leq \left ( \frac{ | H_1 \cap H_2  |  | H_2 \cap H_3  | \cdots  | H_{k_n-1} \cap H_{k_n}  |}{ | H_1  |  | H_2  | \cdots  | H_{k_n-1}  |} \right )^{t_n}.
\end{align*}
Now in order to prove that $\cok(X_n)$ converges to CL, it is enough to show that
\begin{equation} \label{eq44a}
\underset{n \to \infty}{\lim} \sum_{\substack{H_1, \ldots, H_{k_n} \le G \\ H_i \neq H_j \text{ for some } i, j}} \left ( \frac{ | H_1 \cap H_2  |  | H_2 \cap H_3  | \cdots  | H_{k_n-1} \cap H_{k_n}  |}{ | H_1  |  | H_2  | \cdots  | H_{k_n-1}  |} \right )^{t_n} = 0
\end{equation}
for every finite abelian $p$-group $G \neq \lt \{ 1 \rt \}$. (Note that if $H_1 = \cdots = H_{k_n} \ne G$, then $S'_{H_1, \ldots, H_{k_n}} = \varnothing$ so we may exclude this case.)

\begin{lem} \label{lem44a}
Let $H_0, \ldots, H_r$ be finite abelian $p$-groups such that $H_i \ne H_{i+1}$ for each $0 \le i \le r-1$. Then we have
$$
\frac{ | H_0 \cap H_1  |   | H_1 \cap H_2  | \cdots  | H_{r-1} \cap H_r  |}{ | H_0  |  | H_1  | \cdots  | H_{r-1}  |} \le \frac{1}{p^{\frac{r}{2}}}.
$$
\end{lem}

\begin{proof}
By the second isomorphism theorem for groups, the square of the LHS is given by
\begin{equation*}
\left ( \prod_{i=0}^{r-1} \frac{ | H_i \cap H_{i+1}  | }{ | H_i  | } \right )^2
= \left ( \prod_{i=0}^{r-1} \frac{ | H_i \cap H_{i+1}  | }{ | H_i  | } \right )
\left ( \prod_{i=0}^{r-1} \frac{ | H_{i+1}  | }{ | H_i+H_{i+1}  | } \right )
= \left ( \prod_{i=0}^{r-1} \frac{ | H_i \cap H_{i+1}  | }{ | H_i+H_{i+1}  | } \right ).
\end{equation*}
Since $H_i \ne H_{i+1}$ for each $0 \le i \le r-1$, we have $\displaystyle \frac{ | H_i \cap H_{i+1}  | }{ | H_i+H_{i+1}  | } \le \frac{1}{p}$ for each $0 \le i \le r-1$.
\end{proof}

\begin{prop} \label{prop44b}
The equation (\ref{eq44a}) holds for every finite abelian $p$-group $G \neq \{ 1 \}$. 
\end{prop}

\begin{proof}
If $H_1, \ldots, H_{k_n} \le G$ are not all same, there are $r \ge 1$ and $1 \le i_1 < \cdots < i_r \le k_n-1$ such that 
$$
H_q = T_0 \; (1 \le q \le i_1), \, H_q = T_1 \; (i_1+1 \le q \le i_2), \, \ldots , H_q = T_r \; (i_r+1 \le q \le k_n)
$$
where $T_0, T_1, \ldots, T_r \le G$ and $T_i \ne T_{i+1}$ for each $0 \le i \le r-1$. In this case, we have
$$
\left ( \prod_{j=1}^{k_n-1} \frac{ | H_j \cap H_{j+1}  | }{ | H_j  | } \right )^{t_n}
= \left ( \prod_{i=0}^{r-1} \frac{ | T_i \cap T_{i+1}  | }{ | T_i  | } \right )^{t_n} \leq \left ( \frac{1}{p^{\frac{t_n}{2}}} \right )^{r}
$$
by Lemma \ref{lem44a}. 
For a given $r \ge 1$, there are $\binom{k_n-1}{r}$ possible choices of $i_1 < \cdots < i_r$ and at most $m_G^{r+1}$ choices of $T_0, \ldots, T_r \le G$ where $m_G$ denotes the number of subgroups of $G$. Now we have
\begin{align*}
& \sum_{\substack{H_1, \ldots, H_{k_n} \le G \\ H_i \neq H_j \text{ for some } i, j}} \left ( \frac{ | H_1 \cap H_2  |  | H_2 \cap H_3  | \cdots  | H_{k_n-1} \cap H_{k_n}  |}{ | H_1  |  | H_2  | \cdots  | H_{k_n-1}  |} \right )^{t_n} \\
\le & \sum_{r \ge 1} \binom{k_n-1}{r} m_G^{r+1} \left ( \frac{1}{p^{\frac{t_n}{2}}} \right )^{r} \\
= & m_G \left ( \left ( 1 + \frac{m_G}{p^{\frac{t_n}{2}}} \right )^{k_n-1} -1 \right ).
\end{align*}
Since $\displaystyle \frac{k_n-1}{p^{\frac{t_n}{2}}} \le \frac{n}{t_n p^{\frac{t_n}{2}}} \le \frac{n}{t_n p^{\frac{2 \log_p n - 2}{2}}} = \frac{p}{t_n} \to 0$ as $n \to \infty$, we have $\displaystyle \underset{n \to \infty}{\lim} \left (1 + \frac{m_G}{p^{\frac{t_n}{2}}}  \right )^{k_n-1} =1$.
\end{proof}

In the above construction, the number of random entries of $X_n$ satisfies $|\Sg_n | \le (2t_n+2)n \le 4n \log_p n$. By Proposition \ref{prop44b}, there exists a sequence $(\Sg_n)_{n \ge 1}$ such that $|\Sg_n| \le 4n \log_p n$ for every positive integer $n$ and $\cok(X_n)$ converges to CL.

\begin{rmk} \label{rmk44c}
Mészáros \cite{Mes24b} provides an example such that $|\Sg_n| = (2+o(1))n \log_p n$ ($o(1) \to 0$ as $n \to \infty$) and $\cok(X_n)$ converges to CL. Let $(w_n)_{n \ge 1}$ be a sequence of positive integers and $X_n \in \M_n(\zp)$ be a Haar-random matrix supported on $\Sg_n$ where $\sg_{n,i} := \left\{ j \in [n] : |i-j| \le w_n \right\}$. Mészáros \cite[Theorem 1]{Mes24b} proved that $\cok(X_n)$ converges to CL if and only if $\underset{n \to \infty}{\lim} (w_n - \log_p n) = \infty$. Note that
$$|\Sg_n|=\sum_{j=1}^{w_n} (w_n+j) + (2w_n+1)(n-2w_n) + \sum_{j=1}^{w_n} (w_n+j)
= n(2w_n+1) - w_n^2 - w_n,
$$
so we have $|\Sg_n| = (2+o(1))n \log_p n$ if we take $w_n = (1+o(1))\log_p n$.
\end{rmk}

%-----------------------------------------
%-----------------------------------------
\section{Proof of Theorem \ref{mainthm1}} \label{Sec5}

To prove Theorem \ref{mainthm1}, we begin by expressing an upper bound for  $E_n(\mathbb{Z}/p\mathbb{Z})$ as a sum of $p$-powers, which we interpret combinatorially via a bipartite graph. We then employ the configuration model for random bipartite multigraphs and apply the probabilistic method to establish the desired convergence. This approach ultimately leads to the proof of Theorem \ref{mainthm1}.

For a finite abelian $p$-group $G$, let $\mr{Sub}_G$ be the set of all subgroups of $G$. By the assumption $\bigcup_{i=1}^{n} \sg_{n,i} = [n]$ (see Remark \ref{rmk2n1}), for every $F \in \Sur(R^n, G)$ we have $FV_{\sg_{n,1}} + \cdots + FV_{\sg_{n,n}} = G$. Thus we have
\begin{align*}
E_n(G) & = \sum_{F \in \Sur(R^n,G)} \frac{1}{|FV_{\sg_{n,1}}|\cdots|FV_{\sg_{n,n}}|} \\
& = \sum_{\substack{H_1, \ldots, H_n \in \mr{Sub}_G, \\ H_1 + \cdots + H_n = G}} \frac{ | \{ F \in \Sur(R^n, G) \mid FV_{\sg_{n,i}} = H_i \} |}{| H_1  | \cdots | H_n |}.
\end{align*}

Now let $G = \Z/p\Z$ and assume that $H_1, \ldots, H_n \in \mr{Sub}_{\Z/p\Z}$ satisfy the conditions $H_1 + \cdots + H_n = \Z/p\Z$ and
$$
\lt \{ F \in \Sur(R^n, \Z/p\Z) \mid FV_{\sg_{n,i}} = H_i \rt \} \neq \varnothing.
$$
In this case, the set 
$$
S = \lt \{ i \in [n] \mid H_i = 0 \rt \} \subsetneq [n]
$$ 
satisfies $\sg_{n,j} \not\subseteq \bigcup_{i \in S} \sg_{n,i}$ for all $j \in [n] \backslash S$ and
$$
\lt | \bigcap_{i \in \tau_{n,j}} H_i  \rt | = \lt\{\begin{matrix}
1 & (j \in \bigcup_{i \in S} \sg_{n,i}) \\ 
p & (\text{otherwise})
\end{matrix}\rt. 
$$
where $\tau_{n,j} := \{ i \in [n] \mid j \in \sigma_{n,i} \}$.
Thus we have 
\begin{equation} \label{eq5a}
\begin{split}
E_n(\Z / p\Z) 
& \leq \sum_{\substack{S \subsetneq [n] \\ \bigcup_{i \in S} \sg_{n,i} \neq [n] \\ \forall j \in [n]\backslash S, ~\sg_{n,j} \not\subseteq \bigcup_{i \in S} \sg_{n,i} }} \frac{\lt | \lt \{ F \in \Sur(R^n, \Z/p\Z) \mid FV_{\sg_{n,i}} = 0 \text{ for all } i \in S \rt \} \rt |}{p^{n -  | S  |}} \\
& \leq \sum_{\substack{S \subsetneq [n] \\ \bigcup_{i \in S} \sg_{n,i} \neq [n] \\ \forall j \in [n]\backslash S, ~\sg_{n,j} \not\subseteq \bigcup_{i \in S} \sg_{n,i} }} p^{|S| - | \bigcup_{i \in S} \sg_{n,i} |}. 
\end{split}    
\end{equation}

The last term of~\eqref{eq5a} can be rephrased in terms of the neighborhoods of a bipartite graph. For this, we will use the following notations in this section.

\begin{itemize}
    \item 
    A \emph{multigraph} $G$ is a pair $(V(G), E(G))$ with the set $V(G)$ of vertices and the set $E(G)$ of edges, where each edge $e \in E(G)$ is equipped with the set $V(e) \subseteq V(G)$ of endpoints of size two.
    If $V(e) \ne V(f)$ for all distinct $e,f \in E(G)$, then $G$ is a \emph{graph}.

    \item 
    A multigraph $G$ is called \emph{bipartite} if there exist disjoint sets $A,B$ with $V(G) = A \cup B$ such that 
    $|V(e) \cap A| = |V(e) \cap B| = 1$ for all $e \in E(G)$. 
    We call $\{ A , B \}$ a \emph{bipartition} of $G$.

    \item 
    For a multigraph $G$ and a set $S \subseteq V(G)$, 
    the \emph{neighborhood} of $S$, denoted $N_G(S)$, is the set 
    $\bigcup \{ V(e) \setminus S \mid e \in E(G),\: V(e) \cap S \ne \varnothing \}$.
    We also let $N_G(v) := N_G(\{ v \})$ for each $v \in V(G)$.

    \item
    A multigraph $G$ is $d$-\emph{regular} if the number of edges $e \in E(G)$ with $v \in V(e)$ is $d$ for all $v \in V(G)$.
\end{itemize}

Let $A_n = \{ a_1 , \dots , a_n \}$ and $B_n = \{ b_1 , \dots , b_n \}$ be disjoint sets. 
For any bipartite multigraph $G_n$ with bipartition $\{A_n , B_n \}$, let 
\begin{align}\label{eqn:goal}
    c(G_n) \coloneqq \sum_{S \in \cF_{A_n}(G_n)} p^{|S|-|N_{G_n}(S)|},
\end{align}
where $\cF_{A_n} (G_n) \coloneqq \{ \varnothing \ne S \subsetneq A_n \mid N_{G_n}(S) \ne B_n,\: N_{G_n}(w) \not\subseteq N_{G_n}(S)\text{ for all $w \in A \backslash S$} \}$. \\

For $\Sg_n = (\sg_{n,1} , \dots , \sg_{n,n})$, 
consider a bipartite graph $G_{\Sg_n}$ with $V(G_{\Sg_n}) = A_n \cup B_n$ and $E(G_{\Sg_n}) = \{ e_{i,j} \mid j \in \sg_{n,i} \}$ with $V(e_{i,j}) = \{ a_i , b_j \}$. 
Then $N_{G_{\Sg_n}}(a_i) = \{ b_j \mid j \in \sg_{n,i} \}$ and $N_{G_{\Sg_n}}(b_i) = \{ a_j \mid j \in \tau_{n,i} \}$
for each $i \in [n]$. 
Thus, the last term of~\eqref{eq5a} is equal to
\begin{align*}
    1 + \sum_{\substack{\varnothing \ne S \subsetneq A_n \\ N_{G_{\Sg_n}}(S) \neq B_n \\ \forall w \in A_n\backslash S, ~N_{G_{\Sg_n}}(w) \not\subseteq N_{G_{\Sg_n}}(S)}} p^{|S| - |N_{G_{\Sg_n}}(S)|} = 1 + c(G_{\Sg_n}).
\end{align*}

Without loss of generality, we may assume that $|\sg_{n,i}| \leq \min \{ t_n , \log_p n + \log \log n \}$ for all $i \in [n]$ by Proposition \ref{prop3c}.
Hence, the following theorem implies Theorem \ref{mainthm1}.

\begin{thm}\label{thm5a}
Let $(t_n)_{n \geq 1}$ be a sequence of positive integers such that $t_n \leq n$ for each $n$ and $\underset{n \to \infty}{\lim} (t_n - \log_p n) = \infty$. Then there exists a sequence $(G_n)_{n \ge 1}$ of bipartite graphs such that $G_n$ has a bipartition $\{ A_n , B_n \}$ for disjoint sets $A_n, B_n$ of size $n$, $1 \leq |N_{G_n}(a)|, |N_{G_n}(b)| \leq \min \{ t_n , \log_p n + \log \log n \}$ for all $a \in A_n$ and $b \in B_n$, and $\underset{n \to \infty}{\lim} c({G_n}) = 0$.
\end{thm}

We now give a brief sketch of the proof of Theorem \ref{thm5a}. 
The following idea suggests that we need to construct a good bipartite expander $G_n$.
For each $s \geq 1$, if there exists $\alpha = \alpha(s) \in (0,1)$ with $|N_{G_n}(S)| \geq (1 - \alpha) t_n |S|$ for all subsets $S \subseteq A_n$ of size $s$, then since $t_n = \log_p n + \omega(1)$,
\begin{align*}
    p^{|S| - |N_{G_n}(S)|} \leq p^s p^{-t_n (1 - \alpha)s} = p^{-\omega(s)} n^{-s + \alpha s}.
\end{align*}
(We use the standard asymptotic notations $o(.)$, $\omega(.)$, and $O(.)$ to describe the limiting behavior of functions as $n$ tends to infinity.)
Since there are at most $\binom{n}{s} \leq (en/s)^s$ subsets $S \in \cF_{A_n}(G_n)$ of size $s$, the summation $\sum_{|S| = s} p^{|S| - |N_{G_n}(S)|} \leq (en/s)^s p^{-\omega(s)} n^{-s + \alpha s} = o(1)$ if $n^{\alpha s} \ll (s/e)^s$, giving $(1 - \alpha) t_n = t_n - O(\log s)$ (later we will take $t_n <\log_p n + \log\log n).$

In Theorem \ref{thm5b}, 
we will show that a random $t_n$-regular bipartite multigraph $G_n$ satisfies $c(G_n) = o(1)$ with probability at least 0.9.
To see this, for the regime $s \in [1 , \frac{n}{2t_n}]$, we have $|N_{G_n}(S)| \geq (t_n - O(\log s)) s$ for all $S \subseteq A$ or $S \subseteq B$ with $|S| = s$ (see Lemma \ref{lem:size_nbr1}). 
For the other regime $s \geq \frac{n}{2t_n}$, $|N_{G_n}(S)|$ is at least linear in $n$, and moreover if $s \geq n/O(t_n^{1/2})$ then $N_{G_n}(S)$ actually contains almost all vertices from the other side (see Lemmas \ref{lem:size_nbr2} and \ref{lem:cutoff}). 
Utilizing those facts and the definition of $\cF_{A_n}$, we can show that $c(G_n) = o(1)$ with probability at least 0.9.

%-----------------------------------------
\subsection{Proof of Theorem \ref{thm5a}}

We will define a random regular bipartite multigraph using a configuration model due to Bollob\'{a}s \cite{bollobas1980}.
For an overview of probabilistic models on random regular graphs, we refer to the survey \cite{wormald1999}.

For $d \in \bN$ and disjoint sets $A,B$ of size $n$, 
let $\Omega_{A,B,d}$ be the set of bijective functions from $A \times [d]$ to $B \times [d]$.

Let $\iota : (A \times [d]) \times (B \times [d]) \to A \times B$ be a map given by $((a,i),(b,j)) \mapsto (a,b)$, and 
let $\mc{G}_{A, B, d}$ be the set of $d$-regular bipartite multigraphs with bipartition $\{ A , B \}$. 
For each $f \in \Omega_{A,B,d}$, let $\varphi(f)$ be the $d$-regular bipartite multigraph with bipartition $\{ A , B \}$ which satisfies
$$
E(\varphi(f)) = \{ e_{a,i} \mid (a,i) \in A \times [d] \}
$$
where $V(e_{a,i}) := \iota((a,i),f(a,i))$ for each $(a,i) \in A \times [d]$. This gives a map $\varphi : \Omega_{A,B,d} \to \mc{G}_{A, B, d}$. 

\begin{figure}[ht]
\centering
\begin{tikzpicture}
\node (a1_label) at (0.6, 3) {$a_1$};
\node (a2_label) at (0.6, 1.5) {$a_2$};
\node (a3_label) at (0.6, 0) {$a_3$};
\node[circle,fill=black,inner sep=1.5pt,draw] (a1) at (1, 3) {};
\node[circle,fill=black,inner sep=1.5pt,draw] (a2) at (1, 1.5) {};
\node[circle,fill=black,inner sep=1.5pt,draw] (a3) at (1, 0) {};

\node (b1_label) at (4.4, 3) {$b_1$};
\node (b2_label) at (4.4, 1.5) {$b_2$};
\node (b3_label) at (4.4, 0) {$b_3$};
\node[circle,fill=black,inner sep=1.5pt,draw] (b1) at (4, 3) {};
\node[circle,fill=black,inner sep=1.5pt,draw] (b2) at (4, 1.5) {};
\node[circle,fill=black,inner sep=1.5pt,draw] (b3) at (4, 0) {};

\node (e1_label) at (2.5, 3.3) {$e_1$};
\node (e2_label) at (2.5, 2.7) {$e_2$};
\node (e3_label) at (2.5, 1.7) {$e_3$};
\node (e4_label) at (2.1, 1.15) {$e_4$};
\node (e5_label) at (2.9, 1.15) {$e_5$};
\node (e6_label) at (2.5, 0.2) {$e_6$};

\node (e1_label_a) at (1.2, 3.3) {$1$};
\node (e1_label_b) at (3.8, 3.3) {$2$};
\node (e2_label_a) at (1.2, 2.7) {$2$};
\node (e2_label_b) at (3.8, 2.7) {$1$};
\node (e3_label_a) at (1.2, 1.7) {$2$};
\node (e3_label_b) at (3.8, 1.7) {$2$};
\node (e4_label_a) at (1.2, 1.2) {$1$};
\node (e4_label_b) at (3.8, 0.3) {$2$};
\node (e5_label_a) at (1.2, 0.3) {$1$};
\node (e5_label_b) at (3.8, 1.2) {$1$};
\node (e6_label_a) at (1.2, -0.2) {$2$};
\node (e6_label_b) at (3.8, -0.2) {$1$};

\draw[thick] (a1) parabola[bend pos=0.5] bend + (0,-0.1) (b1);
\draw[thick] (a1) parabola[bend pos=0.5] bend + (0,0.1) (b1);
\draw[thick] (a2)--(b2); 
\draw[thick] (a2)--(b3);
\draw[thick] (a3)--(b2); 
\draw[thick] (a3)--(b3); 

\node (a11_label) at (7.4, 3.5) {$(a_1,1)$};
\node (a12_label) at (7.4, 3) {$(a_1,2)$};
\node (a21_label) at (7.4, 2.0) {$(a_2,1)$};
\node (a22_label) at (7.4, 1.5) {$(a_2,2)$};
\node (a31_label) at (7.4, 0.5) {$(a_3,1)$};
\node (a32_label) at (7.4, 0) {$(a_3,2)$};

\node[circle,fill=black,inner sep=1.5pt,draw] (a11) at (8, 3.5) {};
\node[circle,fill=black,inner sep=1.5pt,draw] (a12) at (8, 3) {};
\node[circle,fill=black,inner sep=1.5pt,draw] (a21) at (8, 2.0) {};
\node[circle,fill=black,inner sep=1.5pt,draw] (a22) at (8, 1.5) {};
\node[circle,fill=black,inner sep=1.5pt,draw] (a31) at (8, 0.5) {};
\node[circle,fill=black,inner sep=1.5pt,draw] (a32) at (8, 0) {};

\node (b11_label) at (11.6, 3.5) {$(b_1,1)$};
\node (b12_label) at (11.6, 3) {$(b_1,2)$};
\node (b21_label) at (11.6, 2.0) {$(b_2,1)$};
\node (b22_label) at (11.6, 1.5) {$(b_2,2)$};
\node (b31_label) at (11.6, 0.5) {$(b_3,1)$};
\node (b32_label) at (11.6, 0) {$(b_3,2)$};

\node[circle,fill=black,inner sep=1.5pt,draw] (b11) at (11, 3.5) {};
\node[circle,fill=black,inner sep=1.5pt,draw] (b12) at (11, 3) {};
\node[circle,fill=black,inner sep=1.5pt,draw] (b21) at (11, 2.0) {};
\node[circle,fill=black,inner sep=1.5pt,draw] (b22) at (11, 1.5) {};
\node[circle,fill=black,inner sep=1.5pt,draw] (b31) at (11, 0.5) {};
\node[circle,fill=black,inner sep=1.5pt,draw] (b32) at (11, 0) {};

\draw[thick, ->] (a11)--(b12);
\draw[thick, ->] (a12)--(b11);
\draw[thick, ->] (a22)--(b22); 
\draw[thick, ->] (a21)--(b32);
\draw[thick, ->] (a31)--(b21); 
\draw[thick, ->] (a32)--(b31); 
\end{tikzpicture}  
\caption{A labeled 2-regular bipartite multigraph (left) and its corresponding bijective function in $\Omega_{A,B,2}$ (right)}
\label{fig:regbip_labeled}
\end{figure}

\begin{obs}\label{obs:dist}
Let $G \in \mc{G}_{A,B,d}$. Then
$$
|\varphi^{-1}(\{ G \})| = \frac{(d!)^{2n}}{\prod_{(a,b) \in A \times B} \mu_G(a,b)!},
$$
where $\mu_G(a,b)$ denotes the number of edges $e \in E(G)$ with $V(e) = \{ a , b \}$.
\end{obs}
\begin{proof}
Since $\iota$ is a function that removes a label $(i,j)$ from a pair $((a,i),(b,j))$, any element $f \in \varphi^{-1}(\{ G \})$ can be achieved as follows (see Figure~\ref{fig:regbip_labeled}): for each $v \in A \cup B$, we choose a bijection $\phi_v : E_v \to [d]$, where $E_v := \{ e \in E(G) \mid v \in V(e) \}$. 
Then every edge $e \in E(G)$ with $V(e) \cap A = \{ a \}$ and $V(e) \cap B = \{ b \}$ receives a `label' $(\phi_a (e) , \phi_b(e))$, and it will correspond to a pair $((a , \phi_a(e)) , (b , \phi_b(e)))$ to define $f \in \Omega_{A,B,d}$.

There are $(d!)^{2n}$ choices of bijective functions $\phi_v : E_v \to [d]$ for all $v \in A \cup B$, since $|A \cup B| = 2n$. However, some elements in $\Omega_{A,B,d}$ will be counted multiple times, as we  obtain the same element in $\Omega_{A,B,d}$ if we permute the labels of the edges with the same endpoints; for example, in Figure~\ref{fig:regbip_labeled}, the edges $e_1$ and $e_2$ have the same set of endpoints $\{ a_1 , b_1 \}$, and they have the labels $(1,2)$ and $(2,1)$ respectively. If we switch both labels so that $e_1$ receives a label $(2,1)$ and $e_2$ receives a label $(1,2)$, then the resulting bijective function in $\Omega_{A,B,2}$ is still the same.  
Thus, each element in $\Omega_{A,B,d}$ is counted exactly $\prod_{(a,b) \in A \times B} \mu_G(a,b)!$ times.
\end{proof}

Let ${\bf f} \sim {\rm Unif}(\Omega_{A,B,d})$ be chosen uniformly at random from $\Omega_{A,B,d}$. Then by Observation \ref{obs:dist},
\begin{align*}
    \bP(\varphi({\bf f}) = G) = \frac{|\varphi^{-1}(G)|}{|\Omega_{A,B,d}|} = \frac{(d!)^{2n}}{ (dn)! \prod_{(a,b) \in A \times B} \mu_G(a,b)!}.
\end{align*}

In particular, if $G_1 , G_2 \in \mc{G}_{A,B,d}$ are graphs, then $\bP(\varphi({\bf f}) = G_1) = \bP(\varphi({\bf f}) = G_2)$.

For a multigraph $G$, let ${\rm simp}(G)$ be any graph such that $V({\rm simp}(G)) = V(G)$ and $E({\rm simp}(G))$ is a maximal subset of $E(G)$ such that $V(e) \ne V(f)$ for all $e,f \in E(G)$.
Then $c({\rm simp}(G)) = c(G)$ since $N_G(S) = N_{{\rm simp}(G)}(S)$ for all $S \subseteq V(G)$. Moreover, if $G$ is $d$-regular then $1 \leq |N_{{\rm simp}(G)}(v)| \leq d$ for all $v \in V(G)$.
Thus, the following theorem implies Theorem \ref{thm5a} by taking ${\rm simp}(G_n)$.

\begin{thm}\label{thm5b} 
Let $(t_n)_{n \geq 1}$ be a sequence of positive integers such that $t_n < \log_p n + \log \log n$ for each $n$ and $\underset{n \to \infty}{\lim} (t_n - \log_p n) = \infty$. 

Let $A_n, B_n$ be disjoint sets of size $n$, let ${\bf f}_n \sim {\rm Unif}(\Omega_{A_n,B_n,t_n})$ be chosen uniformly at random from $\Omega_{A_n,B_n,t_n}$, and let $G_n = \varphi({\bf f}_n)$.
Then for any $\dt > 0$, there exists $n_0 \in \bN$ such that $\bP(c(G_n) < \dt) \geq 0.9$ for all $n \geq n_0$.
\end{thm}

First of all, we begin with several ingredients which give lower bounds of $|N_{G_n}(S)|$ for a set $S \subseteq A_n$ or $S \subseteq B_n$.

\begin{lem}\label{lem:prob_nbr}
Let $S \subseteq A_n$ and $T \subseteq B_n$ (or $S \subseteq B_n$ and $T \subseteq A_n$). Then $\bP(N_{G_n}(S) \subseteq T) \leq (|T|/n)^{t_n |S|}$.
\end{lem}
\begin{proof}
Note that $N_{G_n}(S) \subseteq T$ if and only if ${\bf f}_n (s,i) \in T \times [t_n]$ for all $(s,i) \in S \times [t_n]$, which occurs with probability
\begin{align*}
    \frac{t_n |T| (t_n |T| - 1) \cdots (t_n |T| - t_n |S| + 1)}{t_n n (t_n n - 1) \cdots (t_n n - t_n |S| + 1)} \leq (|T|/n)^{t_n |S|},
\end{align*}
as desired.
\end{proof}

\begin{lem}\label{lem:size_nbr1}
For any $\epsilon \in (0,1)$, there exists $n_0 \in \bN$ such that for each $n \geq n_0$, 
with probability at least $0.99$, $|N_{G_n}(S)| > (t_n - 100 - \log_p |S|)|S|$ for all $S \subseteq A_n$ or $S \subseteq B_n$ with $1 \leq |S| \leq \frac{(1 - \epsilon)n}{t_n}$.
\end{lem}

\begin{proof}
For any integer $1 \leq s \leq (1 - \epsilon)n / t_n$, let $t(s) := \lfloor (t_n - 100 - \log_p s) s \rfloor$. 
By Lemma \ref{lem:prob_nbr}, the probability that $|N_{G_n}(S)| \leq (t_n - 100 - \log_p |S|)|S|$ for some $S \subseteq A_n$ or $S \subseteq B_n$ with $1 \leq |S| \leq (1 - \epsilon)n / t_n$, is at most
\begin{align*}
    \sum_{s=1}^{\lfloor (1-\epsilon)n/t_n \rfloor} \sum_{\substack{S \subseteq A_n \\ |S|=s}} \sum_{\substack{T \subseteq B_n \\ |T|=t(s)}} (t(s)/n)^{t_n s} 
    +
    \sum_{s=1}^{\lfloor (1-\epsilon)n/t_n \rfloor} \sum_{\substack{S \subseteq B_n \\ |S|=s}} \sum_{\substack{T \subseteq A_n \\ |T|=t(s)}} (t(s)/n)^{t_n s},
\end{align*}
which is at most
\begin{align}\label{eqn:sum_bdd}
    2 \sum_{s=1}^{\lfloor (1-\epsilon)n/t_n \rfloor} \binom{n}{s} \binom{n}{t(s)} \left ( \frac{t(s)}{n} \right )^{t_n s}.
\end{align}

Our aim is to prove~\eqref{eqn:sum_bdd} is at most 0.01.
Since $\binom{m}{k} \leq (em/k)^k$ for all integers $m \geq k \geq 1$, each term of the summation in~\eqref{eqn:sum_bdd} is at most
\begin{align*}
    \left ( \frac{en}{s} \cdot \left ( \frac{en}{t(s)} \right )^{t_n - 100 - \log_p s} \cdot \left ( \frac{t(s)}{n} \right )^{t_n} \right)^s &\leq \left ( \frac{en}{s} \cdot e^{t_n - 100 - \log_p s} \cdot \left ( \frac{t_n s}{n} \right )^{100 + \log_p s} \right )^s\\
    &= e^{-49s} \left ( e^{t_n - 50 - \log_p s} \cdot t_n \left ( \frac{t_n s}{n} \right )^{99 + \log_p s} \right )^s.
\end{align*}

To that end, if $n$ is sufficiently large, we will show $e^{t_n - 50 - \log_p s} \cdot t_n \left ( \frac{t_n s}{n} \right )^{99 + \log_p s} < 1$ for all $1 \leq s \leq (1-\epsilon)n/t_n$. Then~\eqref{eqn:sum_bdd} is at most $2\sum_{s \geq 1}e^{-49s} \leq \frac{2e^{-49}}{1 - e^{-49}} < 0.01$, which proves the lemma.

Taking $\log$ on both sides of $e^{t_n - 50 - \log_p s} \cdot t_n \left ( \frac{t_n s}{n} \right )^{99 + \log_p s} < 1$ and rearranging terms, we have
\begin{align}\label{eqn:sum_bdd2}
    \frac{t_n - 50 - \log_p s + \log t_n}{99 + \log_p s} + \log s < \log n - \log t_n.
\end{align}

This is equivalent to solving $g(\log s) < 0$ for some monic quadratic polynomial $g$, so the range of $s$ satisfying this inequality is an interval. 
Thus, it suffices to verify this for $s = 1$ and $s = (1-\epsilon)n/t_n$. 

For $s=1$, as we assumed $t_n < \log_p n + \log \log n$, 
the LHS of~\eqref{eqn:sum_bdd2} is at most $\frac{\log n}{99 \log p} + O_p(\log \log n)$ while the RHS is $\log n - \log \log n + O_p(1)$, so~\eqref{eqn:sum_bdd2} holds if $n$ is sufficiently large. 

On the other hand, for $s = (1-\epsilon)n/t_n$, since $t_n < \log_p n + \log \log n$ by the assumption of Theorem \ref{thm5a}, the LHS of~\eqref{eqn:sum_bdd2} is at most $O_p(\log \log n / \log n) + \log s = \log n - \log t_n + \log (1-\epsilon) + o(1)$, which is clearly smaller than the RHS of~\eqref{eqn:sum_bdd2} if $n$ is sufficiently large, as desired.
\end{proof}

Although Lemma \ref{lem:size_nbr1} gives an efficient bound for any small subset $S$, the bound is crude if $|S| \geq n^{1-o(1)}$, since $t_n - 100 - \log_p |S|$ is only $o(\log n)$. 
To complement this, we prove the following two lemmas.

\begin{lem}\label{lem:size_nbr2}
With probability $1-o(1)$, $|N_{G_n}(S)| > n/64$ for all $S \subseteq A_n$ or $S \subseteq B_n$ with $|S| \geq n/(2 t_n)$.
\end{lem}
\begin{proof}
Let $s \coloneqq \lceil n/(2 t_n) \rceil$.
By Lemma \ref{lem:prob_nbr}, the probability that $|N_{G_n}(S)| \leq n/64$ for some $S \subseteq A_n$ or $S \subseteq B_n$ of size $s$, is at most
\begin{align*}
    \sum_{\substack{S \subseteq A_n \\ |S|=s}} \sum_{\substack{T \subseteq B_n \\ |T| = \lfloor n/64 \rfloor}} (|T|/n)^{t_n s}
    +  
    \sum_{\substack{S \subseteq B_n \\ |S|=s}} \sum_{\substack{T \subseteq A_n \\ |T| = \lfloor n/64 \rfloor}} (|T|/n)^{t_n s}
    &\leq
    2 \cdot 2^n \cdot 2^n \cdot (1/64)^{t_n s}\\ 
    &\leq 2 \cdot 4^n \cdot (1/64)^{n/2} = o(1),
\end{align*}
as desired.
\end{proof}

\begin{lem}\label{lem:cutoff}
With probability $1-o(1)$, for all $S \subseteq A_n$ or $S \subseteq B_n$ with $|S| \geq \frac{n}{2 \sqrt{t_n}}$, $|N_{G_n}(S)| > (1 - \frac{1}{\sqrt{t_n}})n$.
\end{lem}
\begin{proof}
Let $s \coloneqq \lceil \frac{n}{2\sqrt{t_n}} \rceil$.
By Lemma \ref{lem:prob_nbr}, the probability that $|N_{G_n}(S)| \leq (1 - 1/\sqrt{t_n}) n$ for some $S \subseteq A_n$ or $S \subseteq B_n$ of size $s$ is at most
\begin{align*}
    2 \binom{n}{s} \binom{n}{\lfloor (1 - t_n^{-1/2}) n \rfloor} \left ( \frac{\lfloor (1 - t_n^{-1/2}) n \rfloor}{n} \right )^{t_n s}
    &\leq 2 \binom{n}{s} \binom{n}{\lceil t_n^{-1/2} n \rceil} (1 - t_n^{-1/2})^{t_n s} \\
    &\leq 2 (en/s)^s \cdot (e \sqrt{t_n})^{n t_n^{-1/2} + 1} \cdot e^{- s \sqrt{t_n}}\\
    &\leq 2 (2e \sqrt{t_n} \cdot e^{-\sqrt{t_n} / 2})^{s} \cdot (e \sqrt{t_n})^{n t_n^{-1/2} + 1} \cdot e^{- s \sqrt{t_n} / 2}\\
    &< 2^{-s+1} = o(1).
\end{align*}

To see this, as $(e \sqrt{t_n})^{n t_n^{-1/2} + 1} \leq \exp(2n \log t_n / \sqrt{t_n} )$, (when $n$ is sufficiently large) it is smaller than $e^{st_n^{1/2} / 2} \geq e^{n/4}$, so $(e \sqrt{t_n})^{n t_n^{-1/2} + 1} \cdot e^{- s t_n^{1/2} / 2} < 1$. It is also straightforward to see that (when $n$ is sufficiently large) $2e t_n^{1/2} \cdot e^{-t_n^{1/2} / 2} < 1/2$, as $t_n = \omega(1)$.
\end{proof}

\begin{lem}\label{lem:reduct}
For any $\dt > 0$, there exists $n_0 \in \bN$ such that for all $n \geq n_0$, with probability at least $0.95$, 
\begin{align}\label{eqn:sum_res}
\sum_{\substack{S \subseteq A_n \\ |S| \in [1, \frac{n}{\sqrt{t_n}}]}} p^{|S| - |N_{G_n}(S)|} + \sum_{\substack{S \subseteq B_n \\ |S| \in [1, \frac{n}{\sqrt{t_n}}]}} p^{|S| - |N_{G_n}(S)|} < \dt.
\end{align}
\end{lem}
\begin{proof}
If $n$ is sufficiently large, 
by Lemmas \ref{lem:size_nbr1} and \ref{lem:size_nbr2} for $\epsilon = 1/2$, with probability at least $0.95$, each term of the LHS of~\eqref{eqn:sum_res} is at most
\begin{align*}
    \sum_{s \in [1 , \frac{n}{2 t_n}]} \left ( \frac{en}{s} \cdot p^{- t_n + 101 + \log_p s} \right )^s + \sum_{s \in (\frac{n}{2 t_n} , \frac{n}{\sqrt{t_n}} ]} (en/s)^s p^{-n/100}.
\end{align*}

Now we show that this is less than $\dt$ when $n$ is sufficiently large. 
Since we have $(en/s) \cdot p^{- t_n + 101 + \log_p s} = e p^{101 - \omega(1)} = o(1)$, 
the first term $\sum_{s \in [1 , \frac{n}{2 t_n}]} \left ( \frac{en}{s} \cdot p^{- t_n + 101 + \log_p s} \right )^s$ is $o(1)$. 
To show that the second term $\sum_{s \in (\frac{n}{2 t_n} , \frac{n}{\sqrt{t_n}} ]} (en/s)^s p^{-n/100}$ is also $o(1)$, observe that since the function $x \mapsto x \log (en/x)$ is an increasing function in $[1 , n]$ and $t_n = \omega(n)$, we have for sufficiently large $n$ that
\begin{align*}
(en/s)^s p^{-n/100} = \exp(s \log \left ( \frac{en}{s} \right ) - \frac{n \log p}{100}) & \leq \exp( \frac{n}{\sqrt{t_n}} \log (e \sqrt{t_n}) - \frac{n \log p}{100}) \\ 
& < \exp(-\frac{n \log p}{200}) = p^{-n/200}.
\end{align*}
 Thus, 
$$\sum_{s \in (\frac{n}{2 t_n} , \frac{n}{\sqrt{t_n}} ]} (en/s)^s p^{-n/100} < n p^{-n/200} = o(1)
$$
as desired.
\end{proof}

Recall that $\cF_{A_n} (G_n) = \{ \varnothing \ne S \subsetneq A_n \mid N_{G_n}(S) \ne B_n,\: N_{G_n}(w) \not\subseteq N_{G_n}(S)\text{ for all $w \in A \backslash S$} \}$.

\begin{lem}\label{lem:dual}
For any $S \in \cF_{A_n}({G_n})$, we have $N_{G_n}(B_n \backslash N_{G_n}(S)) = A_n \backslash S$.
\end{lem}
\begin{proof}
Let $S' \coloneqq B_n \backslash N_{G_n}(S)$. Then $N_{G_n}(S') \subseteq A_n \backslash S$. 
If there exists $v \in A_n \backslash (S \cup N_{G_n}(S'))$ then $N_{G_n}(v) \subseteq B_n \backslash S' = N_{G_n}(S)$, contradicting the assumption that $S \in \cF_{A_n}(G_n)$, as $v \in A_n \backslash S$. 
Thus, $N_{G_n}(S') = A_n \backslash S$.
\end{proof}

Now we are ready to prove Theorem \ref{thm5b}.

\begin{proof}[Proof of Theorem \ref{thm5b}]
Consider a map $\psi_{AB} : \cF_{A_n} ({G_n}) \to 2^{B_n}$ with $\psi_{AB}(S) = B_n \backslash N_{G_n}(S)$ for any $S \in \cF_{A_n}$. Then by Lemma \ref{lem:dual}, $\psi_{AB}$ is injective, and $|S| - |N_{G_n}(S)| = -(n-|S|) + (n - |N_{G_n}(S)|) = -|N_{G_n}(\psi_{AB}(S))| + |\psi_{AB}(S)|$ for all $S \in \cF_{A_n}(G_n)$. Thus,
\begin{align*}
    \sum_{\substack{S \in \cF_{A_n}(G_n) \\ |S| > \frac{n}{2 \sqrt{t_n}}}} p^{|S| - |N_{G_n}(S)|} &= \sum_{\substack{S \in \cF_{A_n}(G_n) \\ |S| > \frac{n}{2 \sqrt{t_n}}}} p^{|\psi_{AB}(S)| - |N_{G_n}(\psi_{AB}(S))|} \\
    &\leq \sum_{\substack{T \subseteq B_n \\ |T| \in [1 , \frac{n}{\sqrt{t_n}} ]}}
    p^{|T| - |N_{G_n}(T)|},
\end{align*}
where the last inequality follows with probability $1 - o(1)$ by Lemma \ref{lem:cutoff}. Therefore,
\begin{align*}
    c({G_n}) &\leq \sum_{\substack{S \subseteq A_n \\ |S| \in [1, \frac{n}{2\sqrt{t_n}}]}} p^{|S|-|N_{G_n}(S)|}
    + \sum_{\substack{S \subseteq A_n \\ |S| \in (\frac{n}{2\sqrt{t_n}}, n]}} p^{|S|-|N_{G_n}(S)|}\\
    &\leq 
    \sum_{\substack{S \subseteq A_n \\ |S| \in [1, \frac{n}{2\sqrt{t_n}}]}} p^{|S|-|N_{G_n}(S)|}
    + \sum_{\substack{T \subseteq B_n \\ |T| \in [1, \frac{n}{\sqrt{t_n}}]}}  p^{|T|-|N_{G_n}(T)|}\\
    &< \dt
\end{align*}
with probability at least $0.95 - o(1)$ by Lemma \ref{lem:reduct} if $n$ is sufficiently large, completing the proof.
\end{proof}

%-----------------------------------------
%-----------------------------------------
\section{Non-universality} \label{Sec6}

Let $X_n \in \M_n(\zp)$ be a Haar-random matrix supported on $\Sg_n$ and $Y_n \in \M_n(\zp)$ be a random $n \times n$ matrix such that $(Y_n)_{i,j}=0$ for $i \notin \sg_{n,j}$ and the entries $(Y_n)_{i,j}$ with $i \in \sg_{n,j}$ are $\epsilon$-balanced and independent. It is natural to ask whether the universality result of Wood \cite[Theorem 1.2]{Woo19} can be extended to the random matrices with fixed zero entries, i.e. does $\cok(Y_n)$ converge to CL if $\cok(X_n)$ converges to CL? The following theorem gives a negative answer to this question.

% Universality not holds for non-uniform matrices
\begin{thm} \label{thm6a}
Let $\xi \in \zp$ be a random variable such that $\bP(\xi = 0) > 1/p$. Let $X_n$ and $Y_n$ be random matrices defined as above and assume that the entries $(Y_n)_{i,j}$ with $i \in \sg_{n,j}$ are i.i.d. copies of $\xi$. Then there exists a sequence $(\Sg_n)_{n \ge 1}$ such that $\cok(X_n)$ converges to CL and $\cok(Y_n)$ does not converge to CL.
\end{thm}

% Statement_old version
\begin{comment}
\begin{thm} \label{thm6a}
Let $0 < \epsilon < 1 - \frac{1}{p}$ and $\xi \in \zp$ be a random variable given by $\bP(\xi = 0) = 1 - \epsilon$ and $\bP(\xi = 1) = \epsilon$. Let $X_n$ and $Y_n$ be random matrices defined as above and assume that the entries $(Y_n)_{i,j}$ with $i \in \sg_{n,j}$ are i.i.d. copies of $\xi$. Then there exists a sequence $(\Sg_n)_{n \ge 1}$ such that $\cok(X_n)$ converges to CL and $\cok(Y_n)$ does not converge to CL.
\end{thm}
\end{comment}

\begin{proof}
Let $(t_n)_{n \ge 1}$ and $(k_n)_{n \ge 1}$ be sequences of positive integers such that $t_nk_n \le n$ for each $n$. Define $\sg_n = (\sg_{n,1}, \ldots, \sg_{n,n})$ for each $n$ by $\sg_{n, i} = \lt \{ (i-1)t_n+1, (i-1)t_n+2, \ldots, it_n \rt \} \subset [n]$ for $1 \le i \le k_n$ and $\sg_{n, i}=[n]$ for $k_n < i \le n$. Let $X_n$ and $Y_n$ be as before and $c := \bP(\xi = 0) > 1/p$. Choose any $(t_n)_{n \ge 1}$ and $(k_n)_{n \ge 1}$ such that $\underset{n \to \infty}{\lim} \frac{k_n}{p^{t_n}} = 0$ and $\underset{n \to \infty}{\lim} c^{t_n}k_n = \infty$.

\begin{figure}[ht]
\begin{equation*}
\begin{pmatrix}
* & 0 & 0 & * & * & * & *\\ 
* & 0 & 0 & * & * & * & *\\ 
0 & * & 0 & * & * & * & *\\ 
0 & * & 0 & * & * & * & *\\ 
0 & 0 & * & * & * & * & *\\ 
0 & 0 & * & * & * & * & *\\ 
0 & 0 & 0 & * & * & * & *
\end{pmatrix}
\end{equation*}
\caption{Matrices $X_n, Y_n \in \M_n(\zp)$ for $(n, t_n, k_n)=(7,2,3)$}
\label{fig5}
\end{figure}

\begin{enumerate}
    \item The probability that $Y_n$ does not have a zero column is bounded above by $(1 - c^{t_n})^{k_n}$. By the assumption $\underset{n \to \infty}{\lim} c^{t_n}k_n = \infty$, we have $\underset{n \to \infty}{\lim} (1 - c^{t_n})^{k_n} = 0$ so $\cok(Y_n)$ does not converge to CL. 

    \item Let $X_n'$ be a random $n \times n$ matrix over $\zp$ which is supported on $\Sg_n$, 
    $$
    (X_n')_{\sg_{n, i}, \lt \{ i \rt \}} = \begin{pmatrix}
    1 & 0 \, \ldots \, 0 
    \end{pmatrix}^T \in \M_{t_n \times 1}(\zp)
    $$
    for each $1 \le i \le k_n$ and the other random entries are independent and Haar-random. (Recall that for $\tau, \tau' \subset [n]$, $(X_n')_{\tau, \tau'}$ denotes the submatrix of $X_n'$ which is obtained by choosing $i$-th rows for $i \in \tau$ and $j$-th columns for $j \in \tau'$. See Section \ref{Sub21}.) By applying Lemma \ref{lem2a} to the blocks $(X_n)_{\sg_{n, i}, \lt \{ i \rt \}}$ ($1 \le i \le k_n$), we have
$$
| \bP(\cok(X_n) \cong H) - \bP(\cok(X_n') \cong H)  | \leq 1 - \left ( 1 - \frac{1}{p^{t_n}} \right )^{k_n} 
$$
for every positive integer $n$ and a finite abelian $p$-group $H$. By the assumption $\underset{n \to \infty}{\lim} \frac{k_n}{p^{t_n}} = 0$, we have $\underset{n \to \infty}{\lim}  | \bP(\cok(X_n) \cong H) - \bP(\cok(X_n') \cong H) | = 0$. Now the distribution of $\cok(X_n')$ is same as the distribution of $\cok(Z_{n-k_n})$ for each $n$ where $Z_{n-k_n} \in \M_{n-k_n}(\zp)$ is Haar-random. Therefore $\cok(X_n)$ converges to CL. \qedhere
\end{enumerate}
\end{proof}

%-----------------------------------------
%-----------------------------------------
\section{The setting for the universality theorem}\label{universality first section}

Let us first recall the notion of a random $\epsilon$-balanced variable \cite{Woo19}.

\begin{defn} \label{def7a}
Let $\epsilon < 1$ be a positive real number. Let $\mc{R}$ be either $\Z$, $\Z_p$ for a prime $p$, or a quotient of $\Z$. We say a random variable $\xi$ in $\mc{R}$ is \emph{$\epsilon$-balanced} if for every maximal ideal $\mc{P}$ of $\mc{R}$ and for every $r \in \mc{R}/\mc{P}$, we have
$$
\bP(\xi \equiv r \pmod{\mc{P}}) \le 1 - \epsilon. 
$$
\end{defn}

\begin{exmp}
By definition, the Haar-random variable $\xi$ in $\zp$ satisfies for each $r \in \Z/p\Z$
$$
\bP(\xi \equiv r \pmod{p}) = \frac{1}{p}
$$
so $\xi$ is $\epsilon$-balanced for any $0< \epsilon \le (p-1)/p$. To give a more drastic example, let $\mu$ be a random variable in $\zp$ defined as follows:
$$
\bP(\mu \equiv r \pmod{p}) = \begin{cases}
0.99999 &~\text{if $r=0$} \\
0.00001 &~\text{if $r=1$}.
\end{cases}
$$
Then $\mu$ is also $\epsilon$-balanced for sufficiently small $\epsilon$. 
\end{exmp}

Let $M$ be a random $n \times n$ matrix over $\mc{R}$. For a positive integer $k$, let $1 \le \alpha_n^{(k)} < \alpha_n^{(k-1)} < \cdots < \alpha_n^{(1)} \le n$ and $n \ge \beta_n^{(k)} > \beta_n^{(k-1)} > \cdots > \beta_n^{(1)} \ge 1$ be positive integers and we define
$$
\alpha_n^{(0)} = \beta_n^{(k+1)} = n \quad \text{and} \quad \alpha_n^{(k+1)} = \beta_n^{(0)} = 0.
$$
For every $1 \le l \le k$, let $M_{i, j}=0$ if $1 \le i \le \alpha_n^{(l)}$ and $1 \le j \le \beta_n^{(l)}$.
The other entries of $M$ are given by independent $\epsilon$-balanced random variables in $\mc{R}$. In this case, we say $M$ is an \emph{$\epsilon$-balanced random matrix over $\mc{R}$ having $k$-step stairs of $0$ with respect to $\alpha_n^{(i)}$ and $\beta_n^{(i)}$}. 

\begin{figure}[ht]
\centering
\begin{circuitikz}[scale=1.4]
\tikzstyle{every node}=[font=\scriptsize]
\draw  (8.75,15) rectangle (13.75,10);
\draw [dashed] (8.75,12.75) -- (10.75,12.75);
\draw [dashed] (10.75,12.75) -- (10.75,14);
\draw [dashed] (10.75,14) -- (11.75,14);
\draw [dashed] (11.75,14) -- (11.75,15);
\draw [<->, >=Stealth, dashed] (8.75,12.35) -- (10.75,12.35);
\draw [<->, >=Stealth, dashed] (8.75,11.7) -- (11.75,11.7);
\draw [<->, >=Stealth, dashed] (12.25,15) -- (12.25,14);
\draw [<->, >=Stealth, dashed] (13,15) -- (13,12.75);
\node [font=\Huge] at (9.75,14) {0};
\node [font=\normalsize] at (12.5,14.5) {$\alpha_n^{(2)}$};
\node [font=\normalsize] at (13.25,14) {$\alpha_n^{(1)}$};
\node [font=\normalsize] at (9.75,12.16) {$\beta_n^{(1)}$};
\node [font=\normalsize] at (10.25,11.51) {$\beta_n^{(2)}$};
\end{circuitikz}
\caption{The shape of an $\epsilon$-balanced matrix with $k$-step stairs of $0$ with respect to $\alpha_n^{(i)}$ and $\beta_n^{(i)}$ when $k=2$}
\label{fig:my_label}
\end{figure}

\begin{thm}
\label{thm: universality main theorem}
Let $M$ be an $\epsilon$-balanced random $n \times n$ matrix over $\Z_p$ having $k$-step stairs of $0$ with respect to $\alpha_n^{(i)}$ and $\beta_n^{(i)}$. Suppose that for every $1 \le i \le k$, 
$$
\underset{n \to \infty}{\lim} (n - \alpha_n^{(i)} - \beta_n^{(i)}) = \infty.
$$
Then $\cok(M)$ converges to CL, i.e. for every finite abelian $p$-group $G$, we have
$$
\underset{n \to \infty}{\lim}\bP(\cok(M) \cong G) = \frac{1}{|\Aut(G)|}\prod_{i = 1}^\infty (1- p^{-i}).
$$
\end{thm}
Of course, $M$ depends on its dimension $n$. However, we suppress $n$ to ease the notation since there is no danger of confusion. 

\begin{rmk}
\label{rem: counter example for the converse of universality theorem}
Unlike the Haar measure case in Proposition \ref{prop2b}, the converse of Theorem \ref{thm: universality main theorem} does not hold. For example, let $k =1$, $\alpha_n^{(1)} = n-1$, and $\beta_n^{(1)} = 1$ for all $n$. In particular, $n - \alpha_n^{(1)} - \beta_n^{(1)} = 0$ for all $n$. Let $p$ be an odd prime and assume that $M_{n,1}$ is given by the random variable $\xi$ such that
\begin{equation*}
\bP(\xi = 1) = \bP(\xi = -1) = 1/2.
\end{equation*}
Let $M'$ be the $(n-1) \times (n-1)$ upper right submatrix of $M$. Since $M_{n,1}$ is always a unit in $\zp$, we have
$$
\cok(M) \cong \cok(M'). 
$$
Note that all entries of $M'$ are given by $\epsilon$-balanced variables in $\zp$. By \cite[Corollary 3.4]{Woo19}, we see that $\cok(M')$ converges to CL, so does $\cok(M)$.
\end{rmk}

Now let $a \ge 2$ be a positive integer. Throughout most of the remainder of this paper, we will work over the finite ring $$R := \Z/ a\Z.$$
Let $V = R^n$ and $v_1, v_2, \ldots, v_n$ denote the standard basis for $V$. For $1\le l \le k+1$, let $V_l$ be the $R$-submodule of $V$ generated by $v_i$ for all $i \in [n] \backslash [\al_n^{(l)}]$. Recall that we write $[t] = \{1,2,\ldots, t\}$ for a positive integer $t$. The rest of the paper will be devoted to prove the following result, which implies Theorem \ref{thm: universality main theorem}. (To see how Theorem \ref{thm: universality moment} induces Theorem \ref{thm: universality main theorem}, see \cite[Theorem 3.1 and Corollary 3.4]{Woo19}.)
\begin{thm}
\label{thm: universality moment}
Let $M$ be an $\epsilon$-balanced random $n \times n$ matrix over $R = \Z/a\Z$ having $k$-step stairs of $0$ with respect to $\alpha_n^{(i)}$ and $\beta_n^{(i)}$. If we have
$$
\underset{n \to \infty}{\lim} (n - \alpha_n^{(i)} - \beta_n^{(i)}) = \infty
$$
for every $1\le i \le k$, then for every finite abelian group $G$ whose exponent divides $a$, we have
$$
\underset{n \to \infty}{\lim} \bE(\#\Sur(\cok(M),G)) = 1.
$$
\end{thm} 
Since every surjection $\cok(M) \twoheadrightarrow G$ can be lifted uniquely to a surjection $V \twoheadrightarrow G$, we see that
$$
\bE(\#\Sur(\cok(M),G)) = \sum_{F \in \Sur(V,G)} \bP(FM = 0).
$$
On the right hand side, we view $M$ as a function from $W (= R^n)$ to $V$, so the identity $FM = 0$ means the composition $F \circ M : W \to G$ is the zero homomorphism.
Therefore, it is enough to show that
\begin{equation}
\label{eq: universality main theorem goal}
    \underset{n \to \infty}{\lim} \sum_{F \in \Sur(V,G)} \bP(FM = 0) = 1. 
\end{equation}
Since the entries of $M$ are independent, we have
$$
\bP(FM = 0) = \prod_{i=1}^n \bP(FM_i = 0),
$$
where 
$$
FM_i = \sum_{l=1}^n F(v_l)M_i(l). 
$$
Here and after we write $M_i$ for the $i$-th column of $M$ and let $M_i(j)$ be the $j$-th entry of $M_i$, i.e. $M_i(j) = M_{i,j}$. As \eqref{eq: universality main theorem goal} is clearly true when $G = \{0\}$, we assume that $|G| > 1$ for the rest of the paper. 
For every group homomorphism $F\in \Hom(V, G)$, define 
$$F_l:= F|_{V_l} \in \Hom(V_l, G)$$
the restriction of $F$ to $V_l$. 
Following \cite{Woo17}, we define the notion of code and $\dt$-depth as follows.
\begin{defn}
For a positive real number $d$, we say $F_l \in \Hom(V_l,G)$ is a \emph{code of distance $d$} if for every $\sg \subseteq [n] \backslash [\al_n^{(l)}]$ with $|\sg| < d$, we have
$$
F_l((V_l)_{\backslash \sg}) = G. 
$$    
\end{defn}
Here, we write $(V_l)_{\backslash \sg}$ for the subgroup of $V_l$ generated by $\left\{v_i: i \in [n]\backslash ([\alpha_n^{(l)}] \cup \sg)\right\}$.
For a positive integer $n = p_1^{a_1}p_2^{a_2}\cdots p_t^{a_t} \ge 2$, where $p_i$'s are distinct primes, let
$$
\ell(D) := a_1 + a_2 + \cdots + a_t,
$$
and let $\ell(1) := 0$. 
\begin{defn}
For a constant $\dt>0$, the \emph{$\dt$-depth} of $F_l \in \Hom(V_l,G)$ is defined to be the largest positive integer $D$ such that there exists $\sg \subseteq [n]\backslash [\alpha_n^{(l)}]$ with $|\sg| < \ell(D) \dt (n- \alpha_n^{(l)})$ and $[G:F((V_l)_{\backslash \sg})] = D$. If there is no such $D$, then we define the $\dt$-depth of $F_l$ be $1$.   
\end{defn}
Note that if $F_l \in \Hom(V_l,G)$ has $\dt$-depth $1$, then $F_l$ is a code of distance $\dt(n - \alpha_n^{(i)})$. 
We choose small constants $\eta>0$ and $\dt_i >0$ for $1 \le i \le k+1$ as in Section \ref{The constants}.
Note that $F_l \in \Hom(V_l, G)$ is a code of distance $\dt_l(n-\alpha_n^{(l)})$ or the $\dt_l$-depth of $F_l$ is $D_l$ for some $D_l >1$.
If $H$ is a proper subgroup of $G$, we define
\begin{align*}
   A_{H}^{(i)} :=  \{F \in \Sur(V,G) & : F_i \text{ is of $\dt_i$-depth $[G:H]$ and there exists $\sg \subseteq [n]\backslash [\alpha_n^{(i)}]$ with}  \\
   & |\sg| < \ell([G:H]) \dt_i (n-\alpha_n^{(i)}) \text{ such that } F_i((V_i)_{\backslash \sg}) = H \text{ and } [F_i(V_i) : H] > 1 \}.
\end{align*}
Let 
$$
A_{G}^{(i)} := \left\{F \in \Sur(V,G) : F_i \text{ is a code of distance $\dt_i (n- \alpha_n^{(i)})$}\right\}. 
$$
For a proper subgroup $H$ of $G$, define
$$
B_{H}^{(i)} = \{F \in \Sur(V,G) : F_i \text{ is of $\dt_i$-depth $[G:H]$ and $F_i(V_i) = H$} \}. 
$$
It is clear by definition that for every $F \in \Sur(V,G)$ with the $\dt_i$-depth of $F_i$ equal to $D>1$, there exists a proper subgroup $H$ of $G$ of index $D$ such that $F \in A_{H}^{(i)}$ or $F \in B_{H}^{(i)}$.

If $\epsilon' \ge \epsilon > 0$, then an $\epsilon'$-balanced variable is $\epsilon$-balanced. Thus, we may and will assume that $\epsilon < 1/2$. In particular, for every proper subgroup $H$ of $G$, we have
\begin{equation}
 \label{eq: b_G < b_H}
\frac{1}{|G|} < \frac{1}{|H|}(1-\epsilon). 
\end{equation}
For our purpose, we assume that for all $1\le i\le k$, we have
$$
\underset{n \to \infty}{\lim} (n - \alpha_n^{(i)} - \beta_n^{(i)}) = \infty.
$$
Since we will work with sufficiently large $n$, we may and will assume that for every $1\le i\le k$, 
$$
n - \alpha_n^{(i)} - \beta_n^{(i)} \ge 1. 
$$

\subsection{The outline of the proof of Theorem \ref{thm: universality moment}} \label{The outline of the proof of the universality theorem}
In this subsection, we give a brief outline of the proof of Theorem \ref{thm: universality moment}. Note first that for any $F \in \Sur(V,G)$ and $1 \le i \le k+1$, $F_i$ is either a code of distance $\dt_i (n - \alpha_n^{(i)})$ or the $\dt_i$-depth of $F_i$ is $D$ for some positive integer $D>1$. In the latter case, there exists a proper subgroup $H$ of $G$ such that $[G:H]=D$, and $F \in A_H^{(i)}$ or $F \in B_H^{(i)}$. Moreover, we note that for any $F \in \Sur(V,G)$, it is impossible that $F \in B_{H}^{(k+1)}$ for a proper subgroup $H$ of $G$ because it would mean $H= F(V) = G$, which is absurd. Let $H_1, H_2, \ldots, H_{k+1}$ be subgroups of $G$. Then we see $F$ falls into one of the following three categories:
\begin{enumerate}
\item
For every $1 \le i \le k+1$, $F_i$ is a code of distance $\dt_i(n-\alpha_n^{(i)})$. 
\item 
At least one of $H_i$ is a proper subgroup of $G$ and
$$F \in \bigcap_{i=1}^{k+1} A_{H_i}^{(i)}.$$ 
\item 
For some $1\le j \le k$ with $H_j$ a proper subgroup of $G$,
$$F \in B_{H_j}^{(j)} ~\bigcap~ \left(\bigcap_{i= j+1}^{k+1} A_{H_{i}}^{(i)} \right).$$     
\end{enumerate}
Let 
$$
\mc{F}_1 = \bigcap_{i=1}^{k+1} A_G^{(i)},
$$
i.e. the set of those $F \in \Sur(V,G)$ such that $F_i$ is a code of distance $\dt_i(n-\alpha_n^{(i)})$ for every $1\le i \le k+1$. 
We first prove in Proposition \ref{prop: code sum to 1} that
$$
\underset{n \to \infty}{\lim} \sum_{F \in \mc{F}_1} \bP(FM = 0) = 1. 
$$
By Proposition \ref{prop: AA for k converges to 0}, if at least one of $H_i$ is a proper subgroup of $G$, then
$$
\underset{n \to \infty}{\lim} \sum_{F \in \cap_{i=1}^{k+1} A_{H_i}^{(i)}} \bP(FM = 0) = 0. 
$$   
Now let us assume $H_j$ is a proper subgroup of $G$. For simplicity, let us write
$$
R_j = R_j(H_j, \ldots, H_{k+1}) = B_{H_j}^{(j)} ~\bigcap~ \left(\bigcap_{i=j+1}^{k+1} A_{H_i}^{(i)}\right).
$$
Let $\eta>0$ be a positive integer, $N_j := \{n \in \bN : n - \alpha_n^{(j)} \ge \eta n\}$ and $N_j^c := \{n \in \bN : n - \alpha_n^{(j)} < \eta n\}$. 
Proposition \ref{prop: BA for k converges to 0 cor} implies that either $N_j$ is finite or 
$$
\lim_{\substack{ n \in N_j \\n \to \infty }}\sum_{F \in R_j} \bP(FM=0) = 0. 
$$
Moreover, Proposition \ref{prop: universality last prop} proves that either $N_j^c$ is finite or 
$$
\lim_{\substack{ n \in N_j^c \\n \to \infty }}\sum_{F \in R_j} \bP(FM=0) = 0. 
$$
Thus we have
$$
\underset{n \to \infty}{\lim} \sum_{F \in R_j} \bP(FM=0) = 0,
$$
which completes the proof of Theorem \ref{thm: universality moment}. 

\subsection{The constants} \label{The constants}
Recall that we have a fixed positive constant $\epsilon < 1/2$.
In this subsection, we fix the constants $\eta, \gamma, \dt_i$ that will be used throughout the proof of Theorem \ref{thm: universality moment}. First, fix a constant $\eta > 0$ satisfying
$$
\eta < \frac{\epsilon}{2\epsilon + 5\log |G|}. 
$$
Also, we fix a constant $\gamma>0$ such that
$$
\gamma <  \frac{\epsilon(1-2\eta)}{5(k+1)}. 
$$
It is well-known that $\binom{n}{\ld n} \le 2^{H(\ld) n}$ where $H(\ld)$ is the binary entropy of $\ld$ and $H(\ld) \to 0$ as $\ld \to 0$. 
Take a sufficiently small constant $\dt_1>0$ such that the following two statements hold. 
\begin{enumerate}
\item 
For all $0<\ld \le \ell(|G|)\dt_1$, the following holds for all sufficiently large $n$:
$$
\binom{n}{\ld n} < e^{\gamma n} < e^{\frac{\epsilon(1-2\eta)}{5(k+1)}},
$$
where the last inequality follows from our choice of $\gamma$ above. 
\item 
$$
\dt_1  < \frac{\epsilon(1-2\eta)}{5(k+1) \ell(|G|)\log |G|} < \frac{\log 2}{5k \ell(|G|)\log |G|}.
$$
\end{enumerate}
Also,  for $2 \le i \le k+1$, we let $\dt_i$ be an arbitrary positive constant satisfying
$$
\dt_i < \frac{\dt_{i-1}\eta}{\ell(|G|)}.
$$
In particular,
$$
\dt_1 > \dt_2 > \cdots > \dt_{k+1}. 
$$

\subsection{The main term of the moment}

In this subsection, we prove that the sum of $\bP(FM = 0)$ over the set
$$
\mc{F}_1 = \bigcap_{i=1}^{k+1} A_{G}^{(i)}  = \left\{F \in \Sur(V,G) : F_i \text{ is a code of distance $\dt_i(n- \alpha_n^{(i)})$ for all $1 \le i \le k+1$}\right\}
$$
converges to $1$. When $X$ is a column vector of $n$ entries over $R$ and $F \in \Hom(V,G)$, we write
$$
FX = \sum_{l=1}^n F(v_l)X(l),
$$
where $X(l)$ denotes the $l$-th entry of $X$.

\begin{lem}
\label{lem: lem for code bound}
Let $\sg \subseteq [n]$ such that $|\sg| = m$ for some positive integer $1 \le m \le n$. Let $X$ be a random column vector (over $R$) of $n$ entries with the $i$-th entries for $i \in \sg$ are fixed to be $0$ and the other entries are independent and $\epsilon$-balanced in $R$. Let $V' = V_{\backslash \sg}$ and suppose that $F|_{V'} \in \Hom(V',G)$ is a code of distance $\dt (n-m)$ for a constant $\dt>0$. Then for every $g \in G$,
$$
\left|\bP(FX = g) -\frac{1}{|G|}\right| \le e^{-\epsilon \dt (n-m)/a^2}. 
$$
\end{lem}

\begin{proof}
This follows similarly as in the proof of \cite[Lemma 2.1]{Woo19}. 
\end{proof}

\begin{lem}
\label{lem: exp bound lem}
Let $k > 0$ and $f(n)$, $g(n)$ be functions from $\bN$ to $\bR$. Suppose that
$$\underset{n \to \infty}{\lim} f(n) = \underset{n \to \infty}{\lim} f(n)g(n) = 0$$ 
Then
$$
\underset{n \to \infty}{\lim}(1 + f(n))^{g(n)}  = 1. 
$$   
\end{lem}

\begin{proof}
Since $\underset{n \to \infty}{\lim} f(n) = 0$, we have $\underset{n \to \infty}{\lim} (1+ f(n))^{1/f(n)} = e$. Then it follows that
$$
\underset{n \to \infty}{\lim}(1 + f(n))^{g(n)} = \underset{n \to \infty}{\lim}\left((1 + f(n))^{1/f(n)}\right)^{f(n)g(n)} = e^0 = 1. \eqno \qedhere
$$
\end{proof}

Recall that we are assuming that for every $1 \le i \le k$, 
$$
\underset{n \to \infty}{\lim} (n- \alpha_n^{(i)} - \beta_n^{(i)}) = \infty. 
$$

\begin{prop}
\label{prop: code sum to 1}
We have
$$\underset{n \to \infty}{\lim}\sum_{F \in \mc{F}_1} \bP(FM = 0) = 1.$$   
\end{prop}

\begin{proof}
First we show that 
$$
\underset{n \to \infty}{\lim} \frac{|\mc{F}_1|}{|G|^n} = 1. 
$$
By \cite[Lemma 2.6]{Woo19}, the number of $F \in \Hom(V,G)$ such that the $\dt_i$-depth of $F_i = F|_{V_i} \in \Hom(V_i, G)$ is greater than $1$ is bounded above by
$$
\sum_{1< D|\#G} C \binom{n- \alpha_n^{(i)}}{\lceil \ell(D)\dt_i (n- \alpha_n^{(i)}) \rceil - 1}|G|^n|D|^{-(n- \alpha_n^{(i)})) + \ell(D)\dt_i (n- \alpha_n^{(i)})} =: c_n^{(i)},
$$
for some constant $C>0$. 
By our choice of $\dt_i$ and $\gamma$ in Section \ref{The constants}, it follows that for every $1 \le i \le k+1$, the following holds for sufficiently large $n$:
$$
\frac{c_n^{(i)}}{|G|^n} \le C\sum_{1< D|\#G} \frac{e^{\gamma (n - \alpha_n^{(i)})}e^{\log(D) \ell(D)\dt_1 (n- \alpha_n^{(i)})}}{D^{(n-\alpha_n^{(i)})}} < C\sum_{1< D|\#G}\frac{e^{(n-\alpha_n^{(i)})/5}}{e^{(\log 2)(n-\alpha_n^{(i)})}}.
$$
Therefore, we have
$$
\underset{n \to \infty}{\lim} \frac{c_n^{(i)}}{|G|^n} = 0. 
$$
Moreover, 
$$
|\Sur(V,G)| \ge |\Hom(V,G)| - \sum_{H < G} |\Hom(V,H)| = |G|^n - \sum_{H <G} |H|^n,
$$
where the sums vary over all proper subgroups $H$ of $G$, so we have
$$
\underset{n \to \infty}{\lim} \frac{|\Sur(V,G)|}{|G|^n} = 1. 
$$
Since we have
$$
|\Sur(V,G)| - \sum_{i=1}^{k+1} c_{n}^{(i)}\le |\mc{F}_1| \le |G|^n,
$$
it follows that
$$
\underset{n \to \infty}{\lim} \frac{|\mc{F}_1|}{|G|^n} = 1. 
$$
If $F \in \mc{F}_1$, we have by Lemma \ref{lem: lem for code bound} that
\begin{align*}
 \bP(FM=0) = \prod_{l =1}^n \bP(FM_l = 0) &\le \prod_{i=1}^{k+1}\left(\frac{1}{|G|} + e^{-\epsilon\dt_i(n-\alpha_n^{(i)})/a^2}\right)^{\beta_n^{(i)}-\beta_n^{(i-1)}}  \\
 &= \frac{1}{|G|^n} \prod_{i=1}^{k+1}\left(1+|G|e^{-\epsilon\dt_i(n-\alpha_n^{(i)})/a^2}\right)^{\beta_n^{(i)}-\beta_n^{(i-1)}}. 
\end{align*}
Note that since $n - \alpha_n^{(i)} - \beta_n^{(i)} \to \infty$ as $n \to \infty$, we have
$$
\underset{n \to \infty}{\lim} \frac{\beta_n^{(i)}}{e^{\epsilon\dt_i(n-\alpha_n^{(i)})/a^2}}= 0. 
$$
Then it follows from Lemma \ref{lem: exp bound lem} that
$$
\underset{n \to \infty}{\lim} \left(1+|G|e^{-\epsilon\dt_i(n-\alpha_n^{(i)})/a^2}\right)^{\beta_n^{(i)}-\beta_n^{(i-1)}} = 1,
$$
hence we have 
$$
\underset{n \to \infty}{\lim} \prod_{i=1}^{k+1}\left(1+|G|e^{-\epsilon\dt_i(n-\alpha_n^{(i)})/a^2}\right)^{\beta_n^{(i)}-\beta_n^{(i-1)}} = 1.
$$
Therefore, 
$$
\underset{n \to \infty}{\lim} \sum_{F \in \mc{F}_1} \bP(FM = 0) \le \underset{n \to \infty}{\lim} \frac{|\mc{F}_1|}{|G|^n} = 1.
$$
Similarly, for $F \in \mc{F}_1$, it follows from Lemma \ref{lem: lem for code bound} that
\begin{align*}
 \bP(FM=0) &\ge \prod_{i=1}^{k+1}\left(\frac{1}{|G|} - e^{-\epsilon\dt_i(n-\alpha_n^{(i)})/a^2}\right)^{\beta_n^{(i)}-\beta_n^{(i-1)}}  \\
 &= \frac{1}{|G|^n} \prod_{i=1}^{k+1}\left(1-|G|e^{-\epsilon\dt_i(n-\alpha_n^{(i)})/a^2}\right)^{\beta_n^{(i)}-\beta_n^{(i-1)}},   
\end{align*}
and Lemma \ref{lem: exp bound lem} shows that
$$
\underset{n \to \infty}{\lim} \prod_{i=1}^{k+1}\left(1-|G|e^{-\epsilon\dt_i(n-\alpha_n^{(i)})/a^2}\right)^{\beta_n^{(i)}-\beta_n^{(i-1)}} = 1.
$$
Therefore, 
$$
1= \underset{n \to \infty}{\lim} \frac{|\mc{F}_1|}{|G|^n}   \le \underset{n \to \infty}{\lim} \sum_{F \in \mc{F}_1} \bP(FM = 0),
$$
and this completes the proof. 
\end{proof}

% ---------------
\subsection{Auxiliary results}

In this subsection, we record some auxiliary results that will be used in the proof of Theorem \ref{thm: universality moment}. 

\begin{lem}
\label{lem: function counting lemma}
Let $m$ be a positive integer and let $H_1, H_2, \ldots, H_m$ be subgroups of $G$ such that
$$
|H_1| \le |H_2| \le \cdots \le |H_m|. 
$$
Let $A_1, A_2, \ldots, A_m$ be subsets of $[n]$ with
$$
|A_1| \le |A_2| \le \cdots \le |A_m|. 
$$
Write $a_i=|A_i|$ and $a_0=0$. Then we have 
$$
\#\{F \in \Hom(V,G) : F(v_j) \in H_i \text{ for all $1\le i \le m$ and $j \in A_i$} \} \le |G|^{n-a_m} \prod_{i=1}^{m} |H_i|^{a_i - a_{i-1}}. 
$$
\end{lem}

\begin{proof}
Let $B_1 := A_1$ and for $2 \le i \le m$, let
$$B_i := A_i ~\bigcap~ \left(\bigcup_{l = 1}^{i-1} A_{l}\right)^c.$$
Then for every $1 \le i \le m$ we have that 
$$
A_i \subseteq \bigcup_{l=1}^{i} A_l = \bigcup_{l=1}^{i} B_l,
$$
where the latter is a disjoint union. Letting $b_i:=|B_i|$, it follows that
\begin{equation}
\label{eq: eq for ai bi}
    a_i  \le b_1 + b_2 + \cdots + b_i.
\end{equation}
Note that if $F(v_j) \in H_i$ for all $1\le i \le m$ and $j \in A_i$, then $F(v_j) \in H_i$ for all $1\le i \le m$ and $j \in B_i$, so the left hand side of the desired inequality is bounded above by
$$
|G|^{n-(b_1+\cdots+b_m)} \prod_{i=1}^m |H_i|^{b_i}. 
$$
Now it is enough to prove the following inequality:
\begin{equation}\label{eq: inequality in function counting lemma}
|G|^{n-(b_1+\cdots+b_m)} \prod_{i=1}^m |H_i|^{b_i} \le |G|^{n-a_m} \prod_{i=1}^{m} |H_i|^{a_i - a_{i-1}}.
\end{equation}
If $b_i = a_i - a_{i-1}$ for all $1\le i \le m$, then \eqref{eq: eq for ai bi} and \eqref{eq: inequality in function counting lemma} are equalities. Otherwise, there exists the smallest positive integer $j$ such that $b_j > a_{j} - a_{j-1}$ (note that $\sum_{i=1}^{m} b_i \ge a_m = \sum_{i=1}^{m} (a_i-a_{i-1})$). Let
$$
b_i':= \begin{cases}
a_j - a_{j-1} &~\text{if $i = j$} \\
b_{j+1} + b_j - (a_j - a_{j-1}) &~\text{if $i = j+1$} \\
b_i &~\text{if $i \neq j, j+1$}    
\end{cases}.
$$
Then the inequality \eqref{eq: eq for ai bi} holds when $b_i$'s are replaced with $b_i'$'s. 
Since $|H_j| \le |H_{j+1}|$, it follows that
$$
|G|^{n-(b_1+\cdots+b_m)} \prod_{i=1}^m |H_i|^{b_i} \le |G|^{n-(b_1'+\cdots+b_m')} \prod_{i=1}^m |H_i|^{b_i'}. 
$$
Moreover, the smallest positive integer $j'$ such that $b_{j'}' > a_{j'} - a_{j'-1}$ (if exists) is strictly larger than $j$. Repeating this argument finitely many times, we deduce \eqref{eq: inequality in function counting lemma}.
\end{proof}

\begin{lem}
\label{lem: right most case1 group subgroup}
Let $1 \le j \le k$ be a positive integer. Suppose that for some $n \in \bN$, $n-\alpha_{n}^{(j)} \ge \eta n$ and 
$$
B_{H_j}^{(j)} ~\bigcap~ \left(\bigcap_{i=j+1}^{k+1} A_{H_i}^{(i)}\right) \neq \varnothing.
$$ 
Then $H_j$ is a subgroup of $H_i$ for every $j+1 \le i \le k+1$. 
\end{lem}

\begin{proof}
Let
$$
F \in B_{H_j}^{(j)} ~\bigcap~ \left(\bigcap_{i=j+1}^{k+1} A_{H_i}^{(i)}\right).
$$
Note first that since $F \in B_{H_j}^{(j)}$, for every $\tau \subseteq [n]\backslash [\alpha_n^{(j)}]$ with $|\tau| < \ell([G:H_j])\dt_j(n-\alpha_n^{(j)})$, we have
\begin{equation}
\label{eq: B_H condition}
F(V_j) = H_j = F((V_j)_{\backslash \tau}).
\end{equation}
Let $j+1 \le i \le k+1$.
If $H_i = G$, then $H_i$ clearly contains $H_j$ as a subgroup. If $H_i$ is a proper subgroup of $G$, then there exists $\sg_i \subseteq [n]\backslash [\alpha_n^{(i)}]$ such that 
$$
|\sg_i| < \ell([G: H_i])\dt_i(n-\alpha_n^{(i)}) \quad \text{and} \quad F((V_i)_{\backslash \sg_i}) = H_i.
$$
By our choice of $\dt_i$, we have 
$$
|\sg_i \cap [n]\backslash [\alpha_n^{(j)}]| \le |\sg_i| < \ell([G: H_i])\dt_i(n-\alpha_n^{(i)}) < \ell([G:H_j])\dt_j(n- \alpha_n^{(j)}). 
$$
Then \eqref{eq: B_H condition} implies that
$$
H_j = F\left((V_j)_{\backslash(\sg_i \cap [n]\backslash [\alpha_n^{(j)}])}\right) \subseteq  F\left((V_i)_{\backslash \sg_i}\right) = H_i. \eqno \qedhere
$$ 
\end{proof}

\begin{lem}
\label{lem: case2 FM_j=0 probability}
Let $H$ be a proper subgroup of $G$ and let $m$ be a positive integer such that $1\le m \le k+1$.
Then for every $\beta_n^{(m-1)}+1\le l \le \beta_n^{(m)}$, the following hold.  
\begin{enumerate}
    \item 
If $F \in A_{H}^{(m)} \cup B_{H}^{(m)}$, then
$$
\bP(FM_l = 0) \le \left(\frac{1}{|H|} + e^{-\epsilon\dt_m (n-\alpha_n^{(m)})/a^2}\right)\bP\left(\sum_{i = \alpha_n^{(m)} +1 }^n F(v_i)M_l(i) \in H\right)
$$
and
$$
\bP(FM_l = 0) \ge \left(\frac{1}{|H|} - e^{-\epsilon\dt_m (n-\alpha_n^{(m)})/a^2}\right)\bP\left(\sum_{i = \alpha_n^{(m)} +1 }^n F(v_i)M_l(i) \in H\right).
$$

\item If $F \in B_{H}^{(m)}$, then
$$
\left|\bP(FM_l = 0) - \frac{1}{|H|}\right| <  e^{-\epsilon\dt_m (n-\alpha_n^{(m)})/a^2}.
$$    
\end{enumerate}
\end{lem}

\begin{proof}
For (1), we closely follow the proof of \cite[Lemma 2.7]{Woo19}. Since $F \in A_{H}^{(m)} \cup B_{H}^{(m)}$, there exists $\sg \subseteq [n] \backslash [\alpha_n^{(m)}]$ such that 
$$
|\sg| < \ell([G:H])\dt_m(n - \alpha_n^{(m)}) \quad \text{and} \quad F((V_m)_{\backslash \sg}) = H.
$$
Then
\begin{equation} \label{eq712a}
\bP (FM_l = 0)  = \bP\left(\sum_{i \in \sg}F(v_i)M_l(i) \in H\right)\bP\left(\sum_{i \not\in \sg } F(v_i)M_l(i) = -\sum_{i \in \sg} F(v_i)M_l(i) \mid \sum_{i \in \sg}F(v_i)M_l(i) \in H \right).
\end{equation}
Since $F((V_m)_{\backslash \sg}) = H$, we have $\sum_{i \in \sg}F(v_i)M_l(i) \in H$ if and only if $\sum_{i = \alpha_n^{(m)}+1}^n F(v_i)M_l(i) \in H$ so
\begin{equation} \label{eq712b}
\bP\left(\sum_{i \in \sg}F(v_i)M_l(i) \in H\right) = \bP\left(\sum_{i = \alpha_n^{(m)} +1 }^n F(v_i)M_l(i) \in H\right).
\end{equation}
Since $M_l(i)=0$ for each $i \in [\alpha_n^{(m)}]$ (by the condition $\beta_n^{(m-1)}+1\le l \le \beta_n^{(m)}$), we have
\begin{equation} \label{eq712c}
\begin{split}
& \bP \left(\sum_{i \not\in \sg }  F(v_i)M_l(i) = -\sum_{i \in \sg} F(v_i)M_l(i) \mid \sum_{i \in \sg}F(v_i)M_l(i) \in H \right) \\
= \, & \bP\left(\sum_{i \in [n]\backslash ([\alpha_n^{(m)}] \cup \sg )} F_m(v_i)M_l(i) = -\sum_{i \in \sg} F_m(v_i)M_l(i) \mid \sum_{i \in \sg}F_m(v_i)M_l(i) \in H \right)
\end{split}    
\end{equation}
where $F_m: V_m \to H$ is the map $F$ whose domain and codomain are restricted to $V_m$ and $H$, respectively. 
Also note that the restriction of $F_m: V_m \to H$ to $(V_m)_{\backslash \sg}$ is a code of distance $\dt_m(n - \alpha_n^{(m)})$. Otherwise, there exists $\tau \subseteq [n] \backslash ([\alpha_n^{(m)}] \cup \sg)$ such that $|\tau| < \dt_m(n - \alpha_n^{(m)})$ and $F_m((V_m)_{\backslash (\sg \cup \tau)}) \subsetneq H$, which contradicts the assumption that $F_m$ is of $\dt_m$-depth $[G:H]$. 

Now the equations (\ref{eq712a}), (\ref{eq712b}), (\ref{eq712c}) and Lemma \ref{lem: lem for code bound} finishes the proof of (1). 
If $F \in  B_{H}^{(m)}$, then $F(v_i) \in H$ for all $i \in [n] \backslash [\alpha_n^{(m)}]$. Then (2) is an immediate consequence of (1). 
\end{proof}

\begin{defn}
For every subgroup $H \le G$, define a constant
$$
b_H := \begin{cases}
\frac{1}{|H|}(1-\epsilon) &\text{ if $H \neq G$} \\
\frac{1}{|G|} &\text{ if $H = G$}.
\end{cases}
$$
\end{defn}

\begin{rmk}
\label{rem: bH inequality}
The assumption $\epsilon < 1/2$ implies that for every $H,K \le G$ with $|H| \ge |K|$, we have
$$
b_H \le b_K
$$
and equality holds if and only if $|H|=|K|$.
\end{rmk}

\begin{lem}
\label{lem: AA upper bound}
Let $H$ be a subgroup of $G$ and let $F \in A_{H}^{(m)}$ for some $1 \le m \le k+1$. 
\begin{enumerate}
    \item 
There exists a constant $C>0$ (which is independent of $F$) such that for every $n \in \bN$,
$$
\prod_{l =\beta_n^{(m-1)}+1}^{\beta_n^{(m)}} \bP(FM_l = 0) \le Cb_H^{\beta_n^{(m)} - \beta_n^{(m-1)}}.
$$
    \item 
There exists a positive integer $N_{\epsilon}$ such that if $n > N_\epsilon$, then for any $ \beta_n^{(m-1)}+1 \le l \le \beta_n^{(m)}$, we have
$$
\bP(FM_l = 0) \le 1 - \epsilon. 
$$    
\end{enumerate}
\end{lem}

\begin{proof}
By Lemma \ref{lem: lem for code bound} and the proof of \cite[Lemma 2.7]{Woo19}, for every  $\beta_n^{(m-1)}+1 \le l \le \beta_n^{(m)}$, we have that
$$
\bP(FM_l = 0) \le 
\begin{cases}
1- \epsilon &\text{ if $H = \{0\}$} \\
\left(\frac{1}{|H|} + e^{-\epsilon\dt_m(n-\alpha_n^{(m)})/a^2}\right)(1-\epsilon)   &\text{ if $H\neq \{0\}$ and $H \neq G$}\\
\frac{1}{|G|} + e^{-\epsilon\dt_m(n-\alpha_n^{(m)})/a^2} &\text{ if $H = G$}.
\end{cases}
$$ 
Then (1) follows from Lemma \ref{lem: exp bound lem}. 
Since we are assuming $|G| >1$ and $\epsilon < 1/2$, (2) also follows.
\end{proof}

\begin{lem}
\label{lem: Smith normal form}
Let $N$ be a $n \times n$ matrix over $\Z$ (or $R = \Z/a\Z$). Then we have
$$
\cok(N) \cong \cok(N^T),
$$
where $N^T$ is the transpose of $N$. 
\end{lem}

\begin{proof}
Let $D$ be the Smith normal form of an $n\times n$ matrix $N$ over $\Z$, i.e. $D=PNQ$ for $P, Q \in \GL_n(\Z)$. Since $D$ is a diagonal matrix, we have $D = D^T = Q^TN^TP^T$ so $\cok(N) \cong \cok(D) \cong \cok(N^T)$. If $N$ is defined over $R$, we lift $N$ to a matrix $\mc{N}$ over $\Z$ so that $\mc{N}_{i,j} = N_{i,j}$ modulo $a$ and run the same argument as above to get $\cok(\mc{N}) \cong \cok(\mc{N}^T)$, and then reduce this modulo $a$. 
\end{proof}

%-----------------------------------------
%-----------------------------------------
\section{The universality theorem for \texorpdfstring{$k=1$}{k=1}}\label{universality second section}
We continue to assume that $\epsilon < 1/2$. In this section, we prove Theorem \ref{thm: universality moment} in the special case that $k=1$. Namely, writing $\alpha_n = \alpha_n^{(1)}$ and $\beta_n = \beta_n^{(1)}$, the goal of this section is to prove the following. 

\begin{thm}
\label{thm: universality k=1 theorem}
Let $M$ be a random $n \times n$ matrix over $R$ with $M_{i,j} = 0$ for all $1 \le i \le \alpha_n$, $1 \le j \le \beta_n$ and the other entries are given as (independent) $\epsilon$-balanced random variables in $R$. Suppose that 
$$
\underset{n \to \infty}{\lim} (n-\al_n-\beta_n) = \infty.
$$
Then for every finite abelian group $G$ whose exponent divides $a$, we have
$$
\underset{n \to \infty}{\lim} \bE(\#\Sur(\cok(M), G ))= \underset{n \to \infty}{\lim} \sum_{F \in \Sur(V,G)} \bP(FM = 0) = 1.
$$
\end{thm}

\begin{lem}
\label{lem: upper bound for all type2 case converges to 0 for k=1}
Let $H_1$ and $H_2$ be subgroups of $G$ and assume that $H_1 \neq G$ or $H_2 \neq G$. Then
$$
\underset{n \to \infty}{\lim} \left|A_{H_1}^{(1)} \bigcap A_{H_2}^{(2)} \right|
b_{H_1}^{\beta_n} b_{H_2}^{n- \beta_n}= 0. 
$$    
\end{lem}

\begin{proof}
When $|H_1| < |H_2|$, the assertion is just a special case of Lemma \ref{lem: special case1 for upper bound for all type2 case converges to 0 proposition} and Lemma \ref{lem: special case2 for upper bound for all type2 case congverges to 0 proposition}, which hold for an arbitrary positive integer $k$. 
Now Suppose that 
$$|H_1| \ge |H_2|.$$ 
Then $H_2$ is a proper subgroup of $G$ and $b_{H_1} \le b_{H_2}$ by Remark \ref{rem: bH inequality}. It follows that
$$
\left|A_{H_1}^{(1)} \bigcap A_{H_2}^{(2)} \right|
b_{H_1}^{\beta_n} b_{H_2}^{n- \beta_n} \le \left|A_{H_2}^{(2)}\right|b_{H_2}^n.
$$
Let $D = [G:H_2]$. Then by \cite[Lemma 2.6]{Woo19}, there exists a constant $C>0$ such that the following holds for sufficiently large $n$:
$$
\left|A_{H_2}^{(2)}\right|b_{H_2}^n \le C \binom{n}{\lceil \ell(D)\dt_2 n\rceil}|G|^n D^{-n + \ell(D)\dt_2 n} (1-\epsilon)^n\left(\frac{D}{|G|}\right)^n \le C e^{-\epsilon n}e^{\gamma n}D^{ \ell(D)\dt_2 n}, 
$$
where the right hand side converges to $0$ as $n \to \infty$ by the choice of the constants $\gamma, \dt_2$ as in Section \ref{The constants}. Therefore, the proposition follows. 
\end{proof}

\begin{lem}
\label{lem: AA for k=1 converges to 0}
Let $H_1$ and $H_2$ be subgroups of $G$ such that $H_1 \neq G$ or $H_2 \neq G$. Then
$$
\underset{n \to \infty}{\lim} \sum_{F \in A_{H_1}^{(1)} \cap A_{H_2}^{(2)}} \bP(FM = 0) = 0.
$$
\end{lem}

\begin{proof}
For $F  \in A_{H_1}^{(1)} \cap A_{H_2}^{(2)}$, it follows from Lemma \ref{lem: AA upper bound}(1) that
$$
\bP(FM = 0) = \prod_{i = 1}^n \bP(FM_i = 0) \le C 
b_{H_1}^{\beta_n} b_{H_2}^{n- \beta_n}
$$
for some constant $C>0$ which is independent of $n$ and $F$. 
Now the desired result follows by Lemma \ref{lem: upper bound for all type2 case converges to 0 for k=1}. 
\end{proof}

\begin{lem}
\label{lem: BA for k=1 converges to 0}
Suppose that for all large enough $n$, $n - \alpha_n \ge \eta n$. Suppose that  $H_1$ is a proper subgroup of $G$. Then 
$$
\underset{n \to \infty}{\lim} \sum_{F \in B_{H_1}^{(1)} \cap A_{H_2}^{(2)}} \bP(FM = 0) = 0.
$$
\end{lem}

\begin{proof}
Let 
$$
\mf{C} = B_{H_1}^{(1)} \bigcap A_{H_2}^{(2)}. 
$$
We may assume that $\mf{C}$ is non-empty. Then by Lemma \ref{lem: right most case1 group subgroup}, we have $H_1$ is a subgroup of $H_2$. For $\beta_n + 1 \le i \le n$ and for $F \in \mf{C}$, we have by Lemma \ref{lem: case2 FM_j=0 probability} that
\begin{align*}
\bP(FM_i = 0) & \le \left(\frac{1}{|H_2|} + e^{-\epsilon \dt_2n/a^2}\right)\bP\left(\sum_{j=1}^n F(v_j)M_i(j) \in H_2 \right) \\
& = \left(\frac{1}{|H_2|}+ e^{-\epsilon \dt_2n/a^2}\right)\bP\left(\sum_{j=1}^{\alpha_n} F(v_j)M_i(j) \in H_2 \right),    
\end{align*}
where the last equality follows since $F \in B_{H_1}^{(1)}$ implies that that $F(v_l) \in H_1 \le H_2$ for all $\alpha_n + 1\le l \le n$. 
Then it follows from Lemma \ref{lem: case2 FM_j=0 probability} and Lemma \ref{lem: exp bound lem} that there exist constants $C_1, C_2 >0$ such that the following holds:
\begin{align*}
  \sum_{F \in \mf{C}} \bP(FM = 0) &=    \sum_{F \in \mf{C}}\left(\prod_{i=1}^n \bP(FM_i = 0) \right) \\
&  \le C_1\sum_{F \in \mf{C}} \left(\frac{1}{|H_1|}\right)^{\beta_n} \left(\prod_{i=\beta_n+1}^n \bP(FM_i = 0)\right) \\
&  \le C_2\sum_{F \in \mf{C}} \left(\frac{1}{|H_1|}\right)^{\beta_n}\left(\prod_{i=\beta_n+1}^n \frac{1}{|H_2|} \bP\left(\sum_{j = 1}^{\alpha_n} 
F(v_j)M_i(j) \in  H_2\right) \right)  \\
&  \le C_2\left(\frac{|H_1|}{|H_2|}\right)^{n- \alpha_n- \beta_n} \sum_{F \in \Sur(V_{[\alpha_n]}, G/H_2)} \left(\prod_{i=\beta_n+1}^n \bP\left(\sum_{j = 1}^{\alpha_n} 
F(v_j)M_i(j) = 0 \text{ in $G/H_2$}\right) \right),
\end{align*}
where the last inequality is a consequence of Lemma \ref{lem: special case fiber}. If $H_2 = G$, then the right hand side converges to $0$, so the result follows. Now suppose that $H_2$ is a proper subgroup of $G$.
Let $M''$ be the upper right $\alpha_n \times \alpha_n$ submatrix of $M$. Then it follows by Lemma \ref{lem: AA upper bound}(2) that for sufficiently large $n$ (let $c = 1- \epsilon$)
\begin{align*}
  \sum_{F \in \mf{C}} \bP(FM = 0) & \le C_2\left(\frac{|H_1|c}{|H_2|}\right)^{n - \alpha_n - \beta_n}  \sum_{F \in \Sur(V_{[\alpha_n]}, G/H_2)}  \left(\prod_{i=n - \alpha_n + 1}^n \bP\left(\sum_{j = 1}^{\alpha_n} 
F(v_j)M_i(j) = 0 \text{ in $G/H_2$}\right) \right)  \\
& \le  C_2\left(\frac{|H_1|c}{|H_2|}\right)^{n - \alpha_n - \beta_n}\sum_{F \in \Sur(V_{[\alpha_n]}, G/H_2)} \bP(FM'' = 0).
\end{align*}
By \cite[Theorem 2.9]{Woo19}, there exists $C_3 > 1$ such that for any $n \in \bN$ (cf. the proof of Lemma \ref{lem: submatrix bound}),
$$
\sum_{F \in \Sur(V_{[\alpha_n]}, G/H_2)} \bP(FM'' = 0) \le C_3,
$$
and this completes the proof since $n - \alpha_n - \beta_n \to \infty$ as $n \to \infty$.
\end{proof}

\begin{lem}
\label{lem: special case fiber}
Let $H_1$ be a proper subgroup of $G$. Suppose that $n - \alpha_n \ge \eta n$. Then the image of the map
$$
\psi: B_{H_1}^{(1)} ~\bigcap~ A_{H_2}^{(2)} \to \Hom(V_{[\alpha_n]}, G/H_2)
$$
given by the composition of the restriction to $V_{[\alpha_n]} := \langle v_1, v_2, \ldots, v_{\alpha_n} \rangle$ with the projection $\gamma_{H_2} : G \to G/H_2$ is contained in $\Sur(V_{[\alpha_n]}, G/H_2)$. Moreover, each fiber has at most $|H_1|^{n - \alpha_n}|H_2|^{\alpha_n}$ elements. 
\end{lem}

\begin{proof}
Let 
$$
F \in B_{H_1}^{(1)} ~\bigcap~ A_{H_2}^{(2)}.
$$
Since $F \in \Sur(V,G)$, obviously $\gamma_{H_2} \circ F \in \Sur(V,G/H_2)$. The condition that $F \in B_{H_1}^{(1)}$ (and $n-\alpha_n \ge \eta n$) implies that for all $\alpha_n + 1 \le  l \le n$, we have
$F(v_l) \in H_1 \le H_2$ by Lemma \ref{lem: right most case1 group subgroup}, from which the first and second assertions follow. 
\end{proof}

\begin{proof}[Proof of Theorem \ref{thm: universality k=1 theorem}]
Note that
$$
 \sum_{F \in \Sur(V,G)} \bP(FM = 0) = \bE(\#\Sur(\cok(M), G)) = \bE(\#\Sur(\cok(M^T), G)) = \sum_{F \in \Sur(V,G)} \bP(FM^T = 0), 
$$
where the second equality follows by Lemma \ref{lem: Smith normal form}. Therefore,
we may assume that $\alpha_n \le \beta_n$ by taking the transpose if necessary. Then clearly, $n - \alpha_n \ge \eta n$ when $n$ is large enough as $n - \al_n - \beta_n \to \infty$. 
As we observed in Section \ref{The outline of the proof of the universality theorem}, if $F \in \Sur(V,G)$, then $F$ falls into one of the following three categories. 
\begin{enumerate}
\item 
$F$ is a code and $F_1$ is also a code. 
\item
$F \in A_{H_1}^{(1)} \cap  A_{H_2}^{(2)}$ for some subgroups $H_1, H_2$ of $G$ where at least one of them is a proper subgroup. 
\item  
$F \in B_{H_1}^{(1)} \cap A_{H_2}^{(2)}$ for some subgroups $H_1, H_2$ of $G$ such that $H_1$ is a proper subgroup.
\end{enumerate}
Then the assertion follows by combining Proposition \ref{prop: code sum to 1}, Lemma \ref{lem: AA for k=1 converges to 0} and Lemma \ref{lem: BA for k=1 converges to 0}. 
\end{proof}

Now we may remove the condition that $n - \alpha_n \ge \eta n$ in Lemma \ref{lem: BA for k=1 converges to 0}:

\begin{cor}
 Let $H_1, H_2$ be subgroups of $G$ such that $H_1$ is a proper subgroup. Then 
$$
\underset{n \to \infty}{\lim} \sum_{F \in B_{H_1}^{(1)} \cap A_{H_2}^{(2)}} \bP(FM = 0) = 0.
$$   
\end{cor}

\begin{proof}
By Theorem \ref{thm: universality k=1 theorem}, we have
$$
\underset{n \to \infty}{\lim}\sum_{F \in \Sur(V,G)} \bP(FM = 0) = 1.
$$
Since $H_1$ is a proper subgroup of $G$, we have
$$
\mc{F}_1 ~\bigcap~ \left( B_{H_1}^{(1)} \bigcap A_{H_2}^{(2)} \right ) = \varnothing.
$$
Then the desired result follows from Proposition \ref{prop: code sum to 1}.
\end{proof}

%-----------------------------------------
%-----------------------------------------
\section{The universality theorem for an arbitrary \texorpdfstring{$k$}{k}}\label{universality third section}

{\bf Induction hypothesis:} Now let $k \ge 2$ and suppose that \eqref{eq: universality main theorem goal} holds when $k$ is replaced by any positive integer less than $k$.

\subsection{Bounding the error terms for the moment (1)}

Recall that
$$\alpha_n^{(0)} = \beta_n^{(k+1)} = n \quad \text{and} \quad \alpha_n^{(k+1)} = \beta_n^{(0)} = 0.$$
Let $H_1, H_2, \ldots, H_{k+1}$ be subgroups of $G$.

\begin{prop}
\label{prop: AA for k converges to 0}
Suppose that $H_i \neq G$ for some $1\le i \le k+1$. Then
$$
\underset{n \to \infty}{\lim} \sum_{F \in \cap_{i=1}^{k+1} A_{H_i}^{(i)}} \bP(FM = 0) = 0. 
$$    
\end{prop}

We deduce Proposition \ref{prop: AA for k converges to 0} by proving its upper bound converges to $0$.
Recall that for $H \le G$, 
$$
b_H = \begin{cases}
\frac{1}{|H|}(1-\epsilon) &\text{ if $H \neq G$} \\
\frac{1}{|G|} &\text{ if $H = G$}.
\end{cases}
$$

\begin{prop}
\label{prop: upper bound for all type2 case converges to 0}
Suppose that $H_i \neq G$ for some $1\le i \le k+1$.
Then
$$
\underset{n \to \infty}{\lim} \left|\bigcap_{i=1}^{k+1} A_{H_i}^{(i)}\right|
\left(\prod_{i=1}^{k+1} b_{H_i}^{\beta_n^{(i)} - \beta_n^{(i-1)}}\right) = 0. 
$$    
\end{prop}

\begin{proof}[Proof of Proposition \ref{prop: AA for k converges to 0}]
Let 
$$F \in \bigcap_{i=1}^{k+1} A_{H_i}^{(i)}.$$ 
By Lemma \ref{lem: AA upper bound}, there exists a constant $C>0$ (independent of $F$ and $n$) such that
$$
\bP(FM = 0) \le C\prod_{i = 1}^{k+1} b_{H_i}^{\beta_n^{(i)} - \beta_n^{(i-1)}}.
$$
Thus, Proposition \ref{prop: AA for k converges to 0} follows from Proposition \ref{prop: upper bound for all type2 case converges to 0}. 
\end{proof}

Now it remains to prove Proposition \ref{prop: upper bound for all type2 case converges to 0}. We first prove Proposition \ref{prop: upper bound for all type2 case converges to 0} in some special cases and then derive Proposition \ref{prop: upper bound for all type2 case converges to 0} from them.  

\begin{lem}
\label{lem: special case1 for upper bound for all type2 case converges to 0 proposition}
Suppose that all $H_i$ are proper subgroups of $G$ and 
$$
|H_1| < |H_2| < \cdots < |H_{k+1}|. 
$$
Then 
$$
\underset{n \to \infty}{\lim} \left|\bigcap_{i=1}^{k+1} A_{H_i}^{(i)}\right|
\left(\prod_{i=1}^{k+1} b_{H_i}^{\beta_n^{(i)} - \beta_n^{(i-1)}}\right) = 0. 
$$    
\end{lem}

\begin{proof}
Let $D_i = [G:H_i]$ and $$F \in \bigcap_{i=1}^{k+1} A_{H_i}^{(i)}.$$
Then for every $1 \le i \le k+1$, there exists $\sg_i \subseteq [n]\backslash [\alpha_n^{(i)}]$ with $|\sg_i| = \lceil \dt_i \ell(D_i)(n - \alpha_n^{(i)}) \rceil -1$ such that $F((V_i)_{\backslash \sg_i}) = H_i$ and $[F(V_i):H_i] > 1$. Let $A_i = [n] \backslash ([\alpha_n^{(i)}] \cup \sg_i)$ and
$$
a_i = |A_i| = n-\alpha_n^{(i)} - \left(\lceil \ell(D_i)\dt_i(n - \alpha_n^{(i)})\rceil -1\right).
$$
Then by our choices of $\dt_i$, we have
$$
a_1 \le a_2 \le \cdots \le a_{k+1}. 
$$
Put $a_0 = 0$.
It follows from Lemma \ref{lem: function counting lemma} that 
$$
\left|\bigcap_{i=1}^{k+1} A_{H_i}^{(i)}\right| \le  |G|^{n-a_{k+1}}\prod_{i=1}^{k+1}|H_i|^{a_i - a_{i-1}}\prod_{i=1}^{k+1} \binom{n- \alpha_n^{(i)}}{a_i},  
$$
where the binomial term represents the number of ways to choose $\sg_i$. 
There exists a constant $C>0$ such that for all sufficiently large $n$, the following inequality holds. (see Section \ref{The constants} for $\gamma$)
$$
\left|\bigcap_{i=1}^{k+1} A_{H_i}^{(i)}\right|\le C|G|^{\ell(D_{k+1})\dt_{k+1}n}\left(\prod_{i=1}^{k+1}|H_i|^{(n - \alpha_n^{(i)})(1 - \ell(D_{i})\dt_{i}) - (n - \alpha_n^{(i-1)})(1 - \ell(D_{i-1})\dt_{i-1})}\right) e^{(k+1) \gamma n}.  
$$
Then we have for all large enough $n$,
$$
\left|\bigcap_{i=1}^{k+1} A_{H_i}^{(i)}\right|
\left(\prod_{i=1}^{k+1} b_{H_i}^{\beta_n^{(i)} - \beta_n^{(i-1)}}\right)
$$
is bounded above by
\begin{align*}
& C(1-\epsilon)^n |G|^{\ell(D_{k+1})\dt_{k+1}n}\left(\prod_{i=1}^k \left(\frac{|H_i|}{|H_{i+1}|}\right)^{n- \alpha_n^{(i)} - \beta_n^{(i)}}|H_{i+1}|^{\ell(|G|)\dt_{i}n}\right) e^{(k+1) \gamma n}  \\
& \le Ce^{-\epsilon n} e^{(k+1)\gamma n}|G|^{\ell(|G|)\dt_1(k+1)n},
\end{align*}
where the right hand side converges to $0$ by our choice of the constants as in Section \ref{The constants}. 
\end{proof}

\begin{lem}
\label{lem: special case2 for upper bound for all type2 case congverges to 0 proposition}
Suppose that $H_{k+1} = G$ and 
$$
|H_1| < |H_2| < \cdots < |H_{k+1}|. 
$$
Then 
$$
\underset{n \to \infty}{\lim} \left|\bigcap_{i=1}^{k+1} A_{H_i}^{(i)}\right|
\left(\prod_{i=1}^{k+1} b_{H_i}^{\beta_n^{(i)} - \beta_n^{(i-1)}}\right) = 0. 
$$    
\end{lem}

\begin{proof}
Let 
$$
F \in \bigcap_{i=1}^{k+1} A_{H_i}^{(i)},
$$
For every $1 \le i \le k$, choose $\sg_i$ and define $A_i,~a_i,~D_i$ as in the proof of Lemma \ref{lem: special case1 for upper bound for all type2 case converges to 0 proposition}. Then we have 
$$
a_1 \le a_2 \le \cdots \le a_k. 
$$
As in the proof of Lemma \ref{lem: special case1 for upper bound for all type2 case converges to 0 proposition}, it follows from Lemma \ref{lem: function counting lemma} that 
$$
\left|\bigcap_{i=1}^{k+1} A_{H_i}^{(i)}\right|\le |G|^{n - a_k}\prod_{i=1}^{k}|H_i|^{a_i - a_{i-1}} \prod_{i=1}^{k} \binom{n- \alpha_n^{(i)}}{a_i}. 
$$
By the choice of $\dt_i,~ \gamma$ as in Section \ref{The constants}, we see that the following holds for all large enough $n$: 
$$
\left|\bigcap_{i=1}^{k+1} A_{H_i}^{(i)}\right|\le C |G|^{\alpha_n^{(k)}+\ell(D_k)\dt_k(n - \alpha_n^{(k)})}\left(\prod_{i=1}^{k}|H_i|^{(n - \alpha_n^{(i)})(1 - \ell(D_{i})\dt_{i}) - (n - \alpha_n^{(i-1)})(1 - \ell(D_{i-1})\dt_{i-1})}\right)e^{k\gamma(n - \alpha_n^{(k)})}.  
$$
Moreover, we have
$$
\prod_{i=1}^{k+1} b_{H_i}^{\beta_n^{(i)} - \beta_n^{(i-1)}} = \left(\frac{1}{|G|}\right)^{n- \beta_n^{(k)}}(1-\epsilon)^{\beta_n^{(k)}}\prod_{i=1}^k\left(\frac{1}{|H_{i}|}\right)^{\beta_n^{(i)} - \beta_n^{(i-1)}}. 
$$
Then for all large enough $n$, 
$$
\left|\bigcap_{i=1}^{k+1} A_{H_i}^{(i)}\right|
\left(\prod_{i=1}^{k+1} b_{H_i}^{\beta_n^{(i)} - \beta_n^{(i-1)}}\right)
$$
is bounded above by
\begin{equation}\label{eq: must be bounded}
C (1-\epsilon)^{\beta_n^{(k)}} \left(\prod_{i=1}^{k}\left(\frac{|H_i|}{|H_{i+1}|}\right)^{n - \alpha_n^{(i)} - \beta_n^{(i)}}\right)|G|^{k\ell(|G|)\dt_1(n-\alpha_n^{(k)})}e^{k\gamma(n - \alpha_n^{(k)})}
\end{equation}
If $\beta_n^{(k)} \ge (n- \alpha_n^{(k)})/2$, then for sufficiently large $n$, \eqref{eq: must be bounded} is bounded above by
$$
C(1-\epsilon)^{\beta_n^{(k)}}|G|^{k\ell(|G|)\dt_1(n-\alpha_n^{(k)})}e^{k\gamma(n - \alpha_n^{(k)})} \le Ce^{-\epsilon (n - \alpha_n^{(k)})/2}|G|^{k\ell(|G|)\dt_1(n-\alpha_n^{(k)})}e^{k\gamma(n - \alpha_n^{(k)})},
$$
which converges to zero as $n \to \infty$ by our choice of the constants as in Section \ref{The constants}.
If $\beta_n^{(k)} < (n- \alpha_n^{(k)})/2$, then for sufficiently large $n$, \eqref{eq: must be bounded} is bounded above by
$$
C\left(\frac{|H_k|}{|G|}\right)^{n - \alpha_n^{(k)} - \beta_n^{(k)}}|G|^{k\ell(|G|)\dt_1(n-\alpha_n^{(k)})}e^{k\gamma(n - \alpha_n^{(k)})} \le C \left(\frac{1}{2}\right)^{(n - \alpha_n^{(k)})/2} |G|^{k\ell(|G|)\dt_1(n-\alpha_n^{(k)})}e^{k\gamma(n - \alpha_n^{(k)})} 
$$
which also converges to zero as $n \to \infty$ by by our choice of the constants as in Section \ref{The constants}.
\end{proof}

\begin{proof}[Proof of Proposition \ref{prop: upper bound for all type2 case converges to 0}]
We use induction on $k$. When $k=1$, the assertion is Lemma \ref{lem: upper bound for all type2 case converges to 0 for k=1}.
Now suppose that $k \ge 2$ and that the assertion is true for $k-1$. If we have
$$
|H_1| < |H_2| < \cdots < |H_{k+1}|,
$$
then the assertion follows from Lemma \ref{lem: special case1 for upper bound for all type2 case converges to 0 proposition} and Lemma \ref{lem: special case2 for upper bound for all type2 case congverges to 0 proposition}. Otherwise, there exists a positive integer $1 \le j \le k$ such that $|H_j| \ge |H_{j+1}|$. For every $1\le i \le k$, define
$$
\hat{H}_i:= \begin{cases}
H_i ~&\text{if $i < j$} \\
H_{i+1} ~ &\text{if $i \ge j$},
\end{cases}
$$
and 
$$
\hat{\alpha}_n^{(i)}:= \begin{cases}
\alpha_n^{(i)} ~&\text{if $i < j$} \\
\alpha_n^{(i+1)} ~&\text{if $i \ge j$},     
\end{cases}
$$
$$
\hat{\dt}_i:= \begin{cases}
\dt_i ~&\text{if $i < j$} \\
\dt_{i+1} ~&\text{if $i \ge j$},     
\end{cases}
$$
and
$$
\hat{\beta}_n^{(i)}:= \begin{cases}
\beta_n^{(i)} ~&\text{if $i < j$} \\
\beta_n^{(i+1)} ~&\text{if $i \ge j$}.     
\end{cases}
$$
Since $|H_j| \ge |H_{j+1}| = |\hat{H}_j|$, we have $b_{H_j} \le b_{H_{j+1}} = b_{\hat{H}_j}$ and 
$$
b_{H_j}^{\beta_n^{(j)} - \beta_n^{(j-1)}}b_{H_{j+1}}^{\beta_n^{(j+1)} - \beta_n^{(j)}} \le b_{\hat{H}_j}^{\hat{\beta}_n^{(j)} - \hat{\beta}_n^{(j-1)}},
$$
so it follows that
$$
\prod_{i=1}^{k+1} b_{H_i}^{\beta_n^{(i)} - \beta_n^{(i-1)}} \le \prod_{i=1}^{k} b_{\hat{H}_i}^{\hat{\beta}_n^{(i)} - \hat{\beta}_n^{(i-1)}}.
$$
Define $A_{\hat{H}_i}^{(i)}$ similarly as $A_{H_i}^{(i)}$ by replacing $\alpha_n^{(i)}, H_i, \dt_i$ with $\hat{\alpha}_n^{(i)}, \hat{H}_i, \hat{\dt}_i$, respectively in the definition of $A_{H_i}^{(i)}$.
Then we have
$$
A_{\hat{H}_i}^{(i)} = \begin{cases}
A_{H_i}^{(i)} ~&\text{if $i <j$} \\
A_{H_{i+1}}^{(i+1)} ~&\text{if $i \ge j$},
\end{cases}
$$
so it is clear that
$$
\left|\bigcap_{i=1}^{k+1} A_{H_i}^{(i)}\right| \le \left|\bigcap_{i=1}^{k} A_{\hat{H}_i}^{(i)}\right|. 
$$
Now the proposition follows by the induction hypothesis.
\end{proof}

\subsection{Bounding the error terms for the moment (2)}
In this subsection we bound the sum of $\bP(FM= 0)$ over the set 
$$
R_j := R_j(H_j, \ldots, H_{k+1}) := B_{H_j}^{(j)} ~\bigcap~ \left(\bigcap_{i=j+1}^{k+1} A_{H_i}^{(i)}\right),
$$
for the positive integers $n$ in the following set:
$$N_j = \{m \in \bN : m - \alpha_m^{(j)} \ge \eta m\}.$$ 
Here, the constant $\eta$ is fixed in Section \ref{The constants}. 
We start with the special case where $j=1$ and $H_1 = \{0\}$. In this special case, we do not require that $n-\alpha_n^{(1)}\ge \eta n$.

\begin{lem}
\label{lem: right most case1 j=1 and H1 trivial}
If $H_1 = \{0\}$, then
$$
\underset{n \to \infty}{\lim} \sum_{F \in R_1} \bP(FM = 0) = 0. 
$$    
\end{lem}

\begin{proof}
If $R_1 = \varnothing$, then clearly
$$
\sum_{F \in R_1} \bP(FM = 0) = 0. 
$$
Now suppose that $R_1$ is nonempty and let $F \in R_1$.
Since $H_1 = \{0\}$ and $F \in B_{H_1}^{(1)}$, we have $F(v_i) = 0$ for all $i \in [n] \backslash [\alpha_n^{(1)}]$, so it follows that 
$$
\bP(FM_l = 0)=1~\text{for all}~ 1 \le l \le \beta_n^{(1)}.
$$ 
Recall that we are assuming that $n$ is large enough so that $n - \alpha_n^{(1)} - \beta_n^{(1)} > 0$. 
Then we have for sufficiently large $n$, 
\begin{align*}
    \sum_{F \in R_1} \bP(FM = 0) &= \sum_{F \in R_1} \left(\prod_{l=1}^n \bP(FM_l = 0)\right) \\
&=  \sum_{F \in R_1} \left(\prod_{l=\beta_n^{(1)}+1}^{n-\alpha_n^{(1)}} \bP(FM_l = 0)\right) \left(\prod_{l = n-\alpha_n^{(1)}+1}^n \bP(FM_l = 0)\right) \\
& \le \sum_{F \in R_1}  (1-\epsilon)^{n - \alpha_n^{(1)} - \beta_n^{(1)}}  \left(\prod_{l = n-\alpha_n^{(1)}+1}^n \bP(FM_l = 0)\right),
\end{align*}
where the inequality is a consequence of Lemma \ref{lem: AA upper bound}(2). 
Let $M''$ be the upper right $\alpha_n^{(1)} \times \alpha_n^{(1)}$ submatrix of $M$. Write $\alpha_n = \alpha_n^{(1)}$. Note that
$$
FM_l = \sum_{i=1}^n F(v_i)M_l(i) = \sum_{i=1}^{\alpha_n} F(v_i)M_l(i). 
$$
Then we have (recall $V_{[\alpha_n]} = \langle v_1, v_2, \ldots, v_{\alpha_n}\rangle$)
$$
\sum_{F \in R_1}\left(\prod_{l = n-\alpha_n+1}^n \bP(FM_l = 0)\right) \le \sum_{F \in\Sur\left(V_{[\alpha_n]}, G\right)} \bP(F M'' = 0) = \bE\left(\# \Sur\left(\cok(M''), G\right)\right),
$$
where $M''$ is the upper right $\alpha_n \times \alpha_n$ submatrix of $M$, so the result follows from Lemma \ref{lem: submatrix bound}. 
\end{proof}

\begin{lem}
\label{lem: submatrix bound}
Let $M''$ be the upper right $\alpha_n^{(1)} \times \alpha_n^{(1)}$ submatrix of $M$. Let $1\le i \le k$ and let $M'$ be the lower left $(n-\alpha_n^{(i)}) \times (n-\alpha_n^{(i)})$ submatrix of $M$. Then there exist constants $C_1,C_2>0$ such that for all $n>0$,
\begin{align*}
  &\bE\left(\# \Sur\left(\cok(M''), G\right)\right) \le C_1 \\
  &\bE\left(\# \Sur\left(\cok(M'), G\right)\right) \le C_2. 
\end{align*}
\end{lem}

\begin{proof}
If $n$ is large enough so that $n-\alpha_n^{(1)} > \beta_{n}^{(1)}$, then $M''$ does not intersect with the ``first step'' of the $k$ step stairs of $0$ of $M$. For this reason, $M''$ may only have $i$ step stairs of $0$ for $0 \le i \le k-1$. For $0 \le i \le k-1$, define
$$
S_i = \{n \in \bN : \text{$M''$ has $i$ step stairs of $0$}\}. 
$$
In other words,
$$
S_i = \{n \in \bN : \beta_n^{(k-i)} \le n- \alpha_n^{(1)} < \beta_n^{(k-i+1)} \}.
$$

If $i \ge 1$ and $n \in S_i$ , then $\alpha_n^{(1)} > n - \beta_n^{(k-i+1)} > n - \alpha_n^{(k-i+1)} - \beta_n^{(k-i+1)}$. Therefore, if $i \ge 1$ and $S_i$ is infinite,
$$
\lim_{\substack{n \in S_i\\ n \to \infty}} \alpha_n^{(1)} \to \infty. 
$$
Then it follows by induction hypothesis on $k$ that if $i \ge 1$ and $S_i$ is infinite,
$$
\lim_{\substack{n \in S_i\\ n \to \infty}}\bE\left(\# \Sur\left(\cok(M''), G\right)\right)  = 1. 
$$
This implies that there exists a positive integer $N$ such that for any $n > N$ with $n \in S_i$ for some $1 \le i \le k-1$
$$
\bE\left(\# \Sur\left(\cok(M''), G\right)\right)  < 2.  
$$
Also, for $n \le N$, we have
$$
\bE\left(\# \Sur\left(\cok(M''), G\right)\right) \le \bE\left(\# \Sur\left(R^{\alpha_n^{(1)}}, G\right)\right) \le |G|^{\alpha_n^{(1)}} \le |G|^N. 
$$
Moreover, by \cite[Theorem 2.9]{Woo19} there exists a constant $t>1$ such that for any $n \in S_0$, 
$$
\bE\left(\# \Sur\left(\cok(M''), G\right)\right) < t.
$$
Now we may take $C_1 = \max(2, |G|^N, t)$, then the assertion for $M''$ follows. 
The assertion for $M'$ is proved similarly (and are even simpler in this case). 
\end{proof}

\begin{lem}
\label{lem: right most case1 1}
Suppose that $N_1 = \{n \in \bN : n-\alpha_n^{(1)} \ge \eta n\}$ is an infinite set.  
Let $H_1$ be a proper subgroup of $G$. Then
$$
\lim_{\substack{ n \in N_1 \\n \to \infty }} \sum_{F \in R_1} \bP(FM = 0) = 0. 
$$    
\end{lem}

\begin{proof}
We may assume that $R_1 \neq \varnothing$ for infinitely many $n \in N_1$ (note that $R_1$ depends on $n$) since otherwise the assertion clearly holds. Let $n$ be such a positive integer.
Then by Lemma \ref{lem: right most case1 group subgroup}, we have 
$$H_1 \le H_2, H_3, \ldots, H_{k+1}.$$
Now for every $1\le i \le k+1$, let
$$
\bar{H}_i = H_i/H_1. 
$$
In particular, $\bar{H}_1 = 0$. We define $\bar{F} = \gamma_{H_1}\circ F$, where
$$
\gamma_{H_1}: G \to G/H_1
$$
is the projection map. Also, define $\bar{B}_{\bar{H}_i}^{(i)}$, $\bar{A}_{\bar{H}_i}^{(i)}$ similarly as $B_{H_i}^{(i)}$, $A_{H_i}^{(i)}$ for $\bar{G}=G/H_1$ in place of $G$. 
Note that for every $2\le i \le k+1$ and for every $F \in B_{H_1}^{(1)}$ with $\sg_i \subseteq [n]\backslash [\alpha_n^{(i)}]$ such that $|\sg_i| < \ell(|G|)\dt_i(n-\alpha_n^{(i)})$, we have
$$
|\sg_i| < \ell([G:H_1])\dt_1(n- \alpha_n^{(1)}),
$$
so it follows 
$$
F((V_i)_{\backslash \sg_i}) \supseteq F\left((V_1)_{\backslash (\sg_i \cap [n]\backslash [\alpha_n^{(1)}])}\right) = H_1. 
$$
Then it is straightforward to check that for $2\le i \le k+1$ and for $F \in B_{H_1}^{(1)},$ we have
$$
F \in A_{H_i}^{(i)} \Longleftrightarrow \bar{F} \in \bar{A}_{\bar{H}_i}^{(i)}. 
$$
Let 
$$
\bar{R}_1 := \bar{B}_{\bar{H}_1}^{(1)} ~\bigcap~ \left(\bigcap_{i=2}^{k+1} \bar{A}_{\bar{H}_i}^{(i)}\right).
$$
Then the above equivalence implies that for $F \in B_{H_1}^{(1)}$,
$$
F \in R_1 \Longleftrightarrow \bar{F} \in \bar{R}_1.
$$
Note that by Lemma \ref{lem: case2 FM_j=0 probability}(1) and Lemma \ref{lem: exp bound lem} there exists a constant $C>0$ such that for large enough $n$,
$$
\sum_{F \in R_1}  \bP(FM=0) \le C\sum_{F \in R_1} \left(\prod_{l=1}^{k+1}\left(\left(\frac{1}{|H_l|}\right)^{\beta_n^{(l)}-\beta_n^{(l-1)}}\prod_{j=\beta_n^{(l-1)} + 1}^{\beta_n^{(l)}}\bP\left(\sum_{i = \alpha_n^{(l)}+1}^{n} F(v_i)M_j(i) \in H_l\right)\right)\right). 
$$
Also, by Lemma \ref{lem: case2 FM_j=0 probability}(1) and Lemma \ref{lem: exp bound lem}, there exists a constant $\bar{C}>0$ such that for large enough $n \in N_1$,
\begin{align*}
\sum_{\mf{F} \in \bar{R}_1}  \bP(\mf{F}M=0) & \ge \bar{C} \sum_{\mf{F}\in \bar{R}_1}   \left(\prod_{l=1}^{k+1}\left(\left(\frac{1}{|\bar{H}_l|}\right)^{\beta_n^{(l)}-\beta_n^{(l-1)}}\prod_{j=\beta_n^{(l-1)} + 1}^{\beta_n^{(l)}}\bP\left(\sum_{i = \alpha_n^{(l)}+1}^{n} \mf{F}(v_i)M_j(i) \in \bar{H}_l\right)\right)\right) \\
& =  \bar{C}  \sum_{\mf{F}\in \bar{R}_1}  \left( \prod_{l=1}^{k+1}\left(\left(\frac{|H_1|}{|H_l|}\right)^{\beta_n^{(l)}-\beta_n^{(l-1)}}\prod_{j=\beta_n^{(l-1)} + 1}^{\beta_n^{(l)}}\bP\left(\sum_{i = \alpha_n^{(l)}+1}^{n} \mf{F}(v_i)M_j(i) \in \bar{H}_l)\right)\right)\right) \\
& =  \bar{C} |H_1|^n \sum_{\mf{F}\in \bar{R}_1}   \left(\prod_{l=1}^{k+1}\left(\left(\frac{1}{|H_l|}\right)^{\beta_n^{(l)}-\beta_n^{(l-1)}}\prod_{j=\beta_n^{(l-1)} + 1}^{\beta_n^{(l)}}\bP\left(\sum_{i = \alpha_n^{(l)}+1}^{n} \mf{F}(v_i)M_j(i) \in \bar{H}_l\right)\right)\right).
\end{align*}
Note that for $\mf{F} \in \bar{R}_1$, the number of $F \in R_1$  with $\mf{F} = \bar{F}$ is at most $|H_1|^n$. Also, since $H_1 \le H_l$ for $1 \le l \le k+1$, it is clear that for $F \in R_1$,  
$$
\bP\left(\sum_{i = \alpha_n^{(l)}+1}^{n} F(v_i)M_j(i) \in H_l\right) \Longleftrightarrow \bP\left(\sum_{i = \alpha_n^{(l)}+1}^{n} \bar{F}(v_i)M_j(i) \in \bar{H}_l\right). 
$$
Then we have the following inequality holds for all sufficiently large $n$:
$$
\sum_{F \in R_1}\bP(FM=0) \le \frac{C}{\bar{C}}\sum_{\mf{F} \in \bar{R}_1}\bP(\mf{F}M=0).
$$
The right hand side converges to $0$ as $n \to \infty$ by Lemma \ref{lem: right most case1 j=1 and H1 trivial}, and this completes the proof. 
\end{proof}

\begin{prop}
\label{prop: BA for k converges to 0}
Let $1\le j_1 < j_2 < \cdots < j_m  \le k$ be positive integers. Suppose that $H_i$ is a proper subgroup of $G$ for all $i \in \{j_1, j_2, \ldots, j_m\}$.
Let $j = j_m$ and suppose that $N_j = \{n \in \bN : n -\alpha_n^{(j)} \ge \eta n\}$ is an infinite set. Let
$$
\mf{D} = \mf{D}_{(j_1, \ldots, j_m)} := \left(\bigcap_{l =1}^m B_{H_{j_l}}^{(j_l)}\right)~ \bigcap~ \left(\bigcap_{\substack{i=1 \\ i \neq j_1, \ldots, j_m}}^{k+1} A_{H_i}^{(i)}\right). 
$$
Then we have
$$
\lim_{\substack{n \in N_j \\n \to \infty}} \sum_{F \in \mf{D}} \bP(FM=0) = 0. 
$$
\end{prop}

\begin{proof}
We may assume that $\mf{D} \neq \varnothing$ for infinitely many $n \in N_j$. Let $n$ be such a positive integer.
Define an equivalence relation $\sim$ on $\mf{D}$ such that for $F, F' \in \mf{D}$,
$$
F \sim F' \Longleftrightarrow F(v_i) = F'(v_i) \text{ for all $i \in [\alpha_n^{(j)}]$}. 
$$
Let
$$
\mf{D}_j =  \left(\bigcap_{l =1}^m B_{H_{j_l}}^{(j_l)}\right)~ \bigcap~ \left(\bigcap_{\substack{i=1 \\ i \neq j_1, \ldots, j_m}}^{j} A_{H_i}^{(i)}\right).
$$
Note that the intersection in the second parentheses is up to $i=j$ while that for $\mf{D}$ is up to $i=k+1$. Note that $F \in \mf{D}_j$ implies that $F(V_j) = H_j$. Define
$$
\mf{A} := \{F|_{(V_j, H_j)} : F \in \mf{D}_j\} \subseteq \Sur(V_j, H_j).
$$
Here and later in the proof, $F|_{(V_j, H_j)}$ denotes the map from $V_j$ to $H_j$ defined by restricting the domain and codomain of $F$ to $V_j$ and $H_j$, respectively. 
For every $K \in  \mf{A}$ and for every $F \in \mf{D}$, define
$$
(F \wedge K)(v_i) = \begin{cases}
F(v_i) \quad &\text{if $i \in [\alpha_n^{(j)}]$} \\
K(v_i) \quad &\text{if $i \in [n] \backslash [\alpha_n^{(j)}]$}.
\end{cases}
$$
Then we have that
\begin{equation}
\label{eq: closed under wedge product}
F \wedge K \in \mf{D}.  
\end{equation}
To see this, note that for any positive integer $i$ with $j<i\le k+1$ and for any $\sg \subseteq [n]\backslash [\alpha_n^{(i)}]$ with
$$
|\sg| < \ell(|G|)\dt_i(n-\alpha_n^{(i)}),
$$
we have
$ | \sg  | <\ell([G:H_j])\dt_j(n-\alpha_n^{(j)})$. 
By the definition of $F \wedge K$ and the fact that $F, K \in B_{H_j}^{(j)}$, we have
$$
(F \wedge K)\left((V_j)_{\backslash \sg}\right) = K\left((V_j)_{\backslash \sg}\right) = H_j = F\left((V_j)_{\backslash \sg}\right).
$$
Also, we have $(F \wedge K)(v_l)=F(v_l)$ for all $l \in [\alpha_n^{(j)}]$, so it follows that
$$
(F\wedge K)\left((V_i)_{\backslash \sg}\right) = F\left((V_i)_{\backslash \sg}\right).
$$
This implies that
$$
F \in A_{H_i}^{(i)} \Longleftrightarrow F \wedge K \in A_{H_i}^{(i)},
$$
so \eqref{eq: closed under wedge product} holds. 
Moreover, if $F \sim F'$ on $\mf{D}$, letting $K \in \mf{A}$ be such that $K = F'_{(V_j, H_j)}$, we have
$$
F' = F  \wedge K.
$$
Therefore, it follows that for $F,F' \in \mf{D}$, 
$$
F \sim F' \Longleftrightarrow F' = F \wedge K \text{ for some $K \in \mf{A}$}. 
$$
Let $\mf{B}$ be a complete set of representatives for $\mf{D}/\hspace{-0.3em}\sim$. 
Let $b$ be a positive integer satisfying $j+1 \le b \le k+1$. By Lemma \ref{lem: right most case1 group subgroup}, we have $H_j \le H_b$. 
Hence, it follows that for $F \sim F'$ and for $\beta_{n}^{(b-1)} < l \le \beta_n^{(b)}$, we have
\begin{align*}
\bP\left(\sum_{i \in [n]\backslash [\alpha_n^{(b)}]} F(v_i)M_{l}(i) \in H_{b}\right) & = \bP\left(\sum_{i \in [\alpha_n^{(j)}]\backslash [\alpha_n^{(b)}]} F(v_i)M_{l}(i) \in H_{b}\right)  \\
& = \bP\left(\sum_{i \in [\alpha_n^{(j)}]\backslash [\alpha_n^{(b)}]} F'(v_i)M_{l}(i) \in H_{b}\right) \\
& = \bP\left(\sum_{i \in [n]\backslash [\alpha_n^{(b)}]} F'(v_i)M_{l}(i) \in H_{b}\right). 
\end{align*}
Then it follows from Lemma \ref{lem: case2 FM_j=0 probability} that for $\beta_{n}^{(b-1)} < l \le \beta_n^{(b)}$ and for sufficiently large $n$,
$$
\bP(F'M_l = 0) \le \bP(F M_l = 0)\frac{1 + |H_b|e^{-\epsilon\dt_b (n-\alpha_n^{(b)})/a^2}}{1 - |H_b|e^{-\epsilon\dt_b (n-\alpha_n^{(b)})/a^2}}.
$$
Then by Lemma \ref{lem: exp bound lem} there exists a constant $C>0$ such that for any $F \sim F'$ in $\mf{D}$ and for all sufficiently large $n$, 
$$
\prod_{l=\beta_n^{(j)}+1}^n \bP(F'M_l = 0) \le C \prod_{l=\beta_n^{(j)}+1}^n \bP(FM_l = 0).
$$

Let $M'$ be the lower left $(n-\alpha_n^{(j)}) \times (n-\alpha_n^{(j)})$ submatrix of $M$. 
Then $M'$ has $j-1$ step stairs of $0$ with respect to $\alpha_n'^{(i)}$ and $\beta_n'^{(i)}$, where $\alpha_n'^{(i)} = \alpha_n^{(i)}- \alpha_n^{(j)}$ and $\beta_n'^{(i)} = \beta_n^{(i)}$. 
Then by the induction hypothesis we see the following holds for large enough $n$:
$$
\sum_{K \in \mf{A}} \left(\prod_{i=1}^{n-\alpha_n^{(j)}} \bP(KM'_i = 0)\right) = \sum_{K \in \mf{A}}\bP(KM' =0) \le \sum_{K \in \Sur(V_j, H_j)}\bP(KM' =0) < 2. 
$$
If $K \in \mf{A} \subseteq \Sur(V_j, H_j)$, then $K$ is a code of distance of $\dt_j(n-\alpha_n^{(j)})$, so Lemma \ref{lem: lem for code bound} implies that for $\beta_n^{(j)}  < l \le n - \alpha_n^{(j)}$ 
$$
\bP(KM'_l = 0) \ge 1/|H_j| - e^{-\epsilon \dt_j (n - \alpha_n^{(j)})/a^2}.
$$
Using Lemma \ref{lem: exp bound lem}, we see that there exists a constant $C'>0$ such that for large enough $n$,
$$
\sum_{K \in \mf{A}} \left(\prod_{i=1}^{\beta_n^{(j)}} \bP(KM'_i = 0)\right) \le C'|H_j|^{n- \alpha_n^{(j)} - \beta_n^{(j)}}.
$$
Therefore it follows that for large enough $n$:
\begin{align*}
    \sum_{F \in \mf{D}} \bP(FM= 0) & = \sum_{K\in \mf{A}} \left(\left(\prod_{i=1}^{\beta_n^{(j)}} \bP(KM'_i = 0)\right)\sum_{F \in \mf{B}} \left(\prod_{i=\beta_n^{(j)}+1}^n \bP((F \wedge K)M_i = 0)\right)\right)\\
  & \le C \sum_{K\in \mf{A}} \left(\left(\prod_{i=1}^{\beta_n^{(j)}} \bP(KM'_i = 0)\right)\sum_{F \in \mf{B}} \left(\prod_{i=\beta_n^{(j)}+1}^n \bP(FM_i = 0)\right)\right) \\
    & \le CC'|H_j|^{n- \alpha_n^{(j)} - \beta_n^{(j)}} \sum_{F \in \mf{B}} \left(\prod_{i=\beta_n^{(j)}+1}^n \bP(FM_i = 0)\right).
\end{align*}

For $ 1\le i \le k-j+2$, define
$$
\hat{H}_i = H_{j-1+i}
$$
and
\begin{align*}
&\hat{\alpha}_n^{(i)} = \alpha_n^{(j-1+i)} \\
&\hat{\beta}_n^{(i)} = \beta_n^{(j-1+i)} \\
&\hat{\dt}_i = \dt_{j-1+i}.   
\end{align*}
Now define $B_{\hat{H}_i}^{(i)}, A_{\hat{H}_i}^{(i)}$ similarly as $B_{H_i}^{(i)}, A_{H_i}^{(i)}$ by replacing $\alpha_n^{(i)}, H_i, \dt_i$ with $\hat{\alpha}_n^{(i)},  \hat{H}_i, \hat{\dt}_i$, respectively in the definition of $B_{H_i}^{(i)}, A_{H_i}^{(i)}$. 
Similarly as above define an equivalence relation $\approx$ on 
$$
\hat{\mf{D}} := B_{\hat{H}_1}^{(1)} ~\bigcap~ \left(\bigcap_{i=2}^{k-j+2} A_{\hat{H}_i}^{(i)}\right)
$$
by letting for $F, F' \in \hat{\mf{D}}$
$$
F \approx F' \Longleftrightarrow F(v_i) = F'(v_i) \text{ for all $i \in [\hat{\alpha}_n^{(1)}] = [\alpha_n^{(j)}]$}. 
$$
Let $\hat{M}$ be an $\epsilon$-balanced random matrix with $k-j+1$ step stairs of $0$ with respect to $\hat{\alpha}_n^{(i)}$ and $\hat{\beta}_n^{(i)}$.
Similarly as above, define
\begin{align*}
 &\hat{\mf{D}}_1 := B_{\hat{H}_1}^{(1)} = B_{H_j}^{(j)}  \\
&\hat{\mf{A}} := \{F|_{(V_j, \hat{H}_1)} : F \in \hat{\mf{D}}_1\} \subseteq \Sur(V_j,  \hat{H}_1), 
\end{align*}
and let $\hat{\mf{B}}$
be a complete set of representatives for $\hat{\mf{D}}/\hspace{-0.3em}\approx$. Note that since $\mf{D} \subseteq \hat{\mf{D}}$ we may choose $\hat{\mf{B}}$ so that 
$$\mf{B} \subseteq \hat{\mf{B}},$$ 
and we assume this. 
Let $\hat{M}'$ be the lower left $(n-\alpha_n^{(j)}) \times (n-\alpha_n^{(j)})$ submatrix of $\hat{M}$. 
Note that $K \in \hat{\mf{A}}$ implies that $K \in \Sur(V_j, \hat{H}_1)$ is a code of distance $\dt_j(n-\alpha_n^{(j)})$. 
Then similarly as above it follows from Lemma \ref{lem: lem for code bound}, Lemma \ref{lem: exp bound lem} and Lemma \ref{lem: cardinality of mathcal A} that there exist positive constants $\hat{C}$ and $\hat{C}'$ such that the following holds for sufficiently large $n$:
\begin{align*}
    \sum_{F \in \hat{\mf{D}}} \bP(F\hat{M}= 0) & \ge \hat{C} \sum_{K\in \hat{\mf{A}}} \left(\prod_{i=1}^{\beta_n^{(j)}} \bP(K\hat{M}'_i = 0)\right)\sum_{F \in \hat{\mf{B}}} \left(\prod_{i=\beta_n^{(j)}+1}^n \bP(FM_i = 0)\right) \\
    & \ge \hat{C} |\hat{\mf{A}}|\left(\frac{1}{|\hat{H}_1|} - e^{-\epsilon \dt_j (n-\alpha_n^{(j)})}\right)^{\beta_n^{(j)}} \sum_{F \in \hat{\mf{B}}} \left(\prod_{i=\beta_n^{(j)}+1}^n \bP(FM_i = 0)\right) \\
    & \ge \hat{C}\hat{C}' |\hat{H}_1|^{n- \hat{\alpha}_n^{(1)} - \hat{\beta}_n^{(1)}}\sum_{F \in \hat{\mf{B}}} \left(\prod_{i=\beta_n^{(j)}+1}^n \bP(FM_i = 0)\right).
\end{align*}
Then we have
$$
 \sum_{F \in \mf{D}} \bP(FM= 0) \le \frac{CC'}{\hat{C}\hat{C}'} \sum_{F \in \hat{\mf{D}}} \bP(F\hat{M}= 0),
$$
and the result follows from Lemma \ref{lem: right most case1 1}. 
\end{proof}

\begin{lem}
\label{lem: cardinality of mathcal A}
Let $\hat{H_1}, \hat{\alpha}_n^{(1)}, \hat{\mf{A}}, \hat{\dt}_1$ be as in the proof of Proposition \ref{prop: BA for k converges to 0}. Then 
$$
\underset{n \to \infty}{\lim} \frac{|\hat{\mf{A}}|}{|\hat{H}_1|^{n-\hat{\alpha}_n^{(1)}}} = 1.
$$  
\end{lem}

\begin{proof}

Note first that (e.g. see the proof of Proposition \ref{prop: code sum to 1})
$$
\underset{n \to \infty}{\lim}\frac{|\Sur(V_j, \hat{H}_1)|}{|\hat{H}_1|^{n-\hat{\alpha}_n^{(1)}}} = 1. 
$$
Similar to the proof of \cite[Lemma 2.6]{Woo19}, it follows that for some constant $C>0$,
$$
|\hat{\mf{A}|} \ge |\Sur(V_j, \hat{H}_1)| - \sum_{1< D|\#\hat{H}_1} C \binom{n- \hat{\alpha}_n^{(1)}}{\lceil \ell(D[G:\hat{H}_1])\hat{\dt}_1 (n- \hat{\alpha}_n^{(1)}) \rceil -1}|\hat{H}_1|^{n-\hat{\alpha}_n^{(1)}}D^{-(n- \hat{\alpha}_n^{(1)})(1-\ell(D[G:\hat{H}_1])\hat{\dt}_1)}.
$$
Now the result follows by our choice of the constants in Section \ref{The constants}.
\end{proof}

\begin{rmk}
In the proof of Proposition \ref{prop: BA for k converges to 0}, if
$\mathfrak{D} \neq \varnothing$ for some $n \in N_j$, we have by Lemma \ref{lem: right most case1 group subgroup} that
\begin{equation}
\label{eq: subgroup relation}  
H_j \le H_{j+1}, \ldots, H_{k+1}.
\end{equation}
This relation is pivotal for the inductive argument employed in the proof. 
Indeed, if we assume \eqref{eq: subgroup relation}, the same reasoning shows that (even without the $n \in N_j$ condition in the limit)
$$
\underset{n \to \infty}{\lim} \sum_{F \in \mf{D}} \bP(FM = 0) = 0. 
$$
However, in the lemma, the condition that $n \in N_j$ in the limit is essential because there is a possibility that $\mf{D} = \varnothing$ for all $n\in N_j$ and $N_j^c$ is an infinite set.
In this case, we cannot guarantee that \eqref{eq: subgroup relation} holds, hence a different argument is needed, which will be given in the next subsection.
\end{rmk}

Recall that
$$
R_j =  R_j(H_j, \ldots, H_{k+1}) = B_{H_j}^{(j)} ~\bigcap~ \left(\bigcap_{i= j+1}^{k+1} A_{H_{i}}^{(i)} \right).
$$

\begin{prop}
\label{prop: BA for k converges to 0 cor}
Let $j$ be a positive integer such that $1\le j \le k$. Suppose that $H_j$ is a proper subgroup of $G$ and suppose that $N_j = \{ n \in \bN : n - \alpha_n^{(j)} \ge \eta n\}$ is an infinite set. Then
$$
\lim_{\substack{n \in N_j \\ n \to \infty}}\sum_{F \in R_j}\bP(FM = 0) = 0. 
$$
\end{prop}

\begin{proof}
Note that $R_j$ is a union of $\mf{D}_{(j_1, \ldots, j_m)}$ in Proposition \ref{prop: BA for k converges to 0}, where $(j_1, \ldots, j_m)$  runs over all tuples such that $1\le j_1 < j_2 < \cdots < j_m = j$ and also $H_1, H_2, \ldots, H_{j-1}$ run over all subgroups of $G$ (while $H_j, H_{j+1}, \ldots, H_{k+1}$ are fixed). Then Proposition \ref{prop: BA for k converges to 0} yields the desired result. 
\end{proof}

\begin{thm}
\label{thm: universality main theorem with restriction}
Suppose that $N_1 = \{ n \in \bN : n - \alpha_n^{(1)} \ge \eta n\}$ is an infinite set. 
Then 
$$
\lim_{\substack{n \in N_1 \\ n \to \infty}}\bE(\#\Sur(\cok(M), G)) =  \lim_{\substack{n \in N_1 \\ n \to \infty}}\sum_{F \in \Sur(V,G)}\bP(FM = 0) = 1. 
$$
\end{thm}

\begin{proof}
Let $F \in \Sur(V,G)$. Then $F$ falls into one of the following three categories.
\begin{enumerate}
\item
For all $1\le i \le k+1$, $F_i$ is a code of distance $\dt_i(n- \alpha_n^{(i)})$, i.e., 
$$F \in \mc{F}_1.$$
\item 
For $H_1, H_2, \ldots, H_{k+1}$ subgroups of $G$ at least one of them being proper,
$$F \in \bigcap_{i=1}^{k+1} A_{H_i}^{(i)}.$$ 
\item 
For some $1\le j \le k$ with $H_j$ a proper subgroup of $G$ and $H_{j+1}, \ldots, H_{k+1}$ subgroups of $G$, 
$$F \in R_j = R_j(H_j, \ldots, H_{k+1}) = B_{H_j}^{(j)} ~\bigcap~ \left(\bigcap_{i= j+1}^{k+1} A_{H_{i}}^{(i)} \right).$$     
\end{enumerate}

Note that the condition that $n- \alpha_n^{(1)} \ge \eta n$ clearly implies $n - \alpha_n^{(j)} \ge \eta n$.
Now the theorem follows from Proposition \ref{prop: code sum to 1}, Proposition \ref{prop: AA for k converges to 0}, and Proposition \ref{prop: BA for k converges to 0 cor}. 
\end{proof}

\begin{cor}
\label{cor: changing row and column by transpose}
Suppose that $N_1'= \{n \in \bN : n - \beta_n^{(k)} \ge \eta n\}$ is an infinite set. Then
\begin{enumerate}
\item 
$$
\lim_{\substack{n \in N_1' \\ n \to \infty}}\bE(\#\Sur(\cok(M),  G)) =  \lim_{\substack{n \in N_1' \\ n \to \infty}}\sum_{F \in \Sur(V,G)}\bP(FM = 0) = 1. 
$$
\item 
For some $1\le j \le k$ with $H_j$ a proper subgroup of $G$ and $H_{j+1}, \ldots, H_{k+1}$ subgroups of $G$, 
$$
\lim_{\substack{n \in N_1' \\ n \to \infty}} \sum_{F \in R_j} \bP(FM = 0) = 0. 
$$   
\end{enumerate}
\end{cor}

\begin{proof} 
By Lemma \ref{lem: Smith normal form}, we have
$$
\cok(M) \cong \cok(M^T). 
$$
Then we have
$$
1 = \lim_{\substack{n \in N_1' \\ n \to \infty}}\bE(\#\Sur(\cok(M^T),  G)) = \lim_{\substack{n \in N_1' \\ n \to \infty}}\bE(\#\Sur(\cok(M),  G)) = \lim_{\substack{n \in N_1' \\ n \to \infty}}\sum_{F \in \Sur(V,G)}\bP(FM = 0),  
$$
where the first equality is a consequence of Theorem \ref{thm: universality main theorem with restriction}.
Therefore, (1) follows. 
The second assertion (2) follows from (1) and Proposition \ref{prop: code sum to 1} by noting that $R_j \cap \mc{F}_1 = \varnothing$.
\end{proof}

\subsection{Bounding the error terms for the moment (3)}
In this subsection, let $1 \le j \le k$ be a positive integer and assume that $H_j$ is a proper subgroup of $G$. 
Recall that 
$$
N_j^c = \{n \in \bN: n - \alpha_n^{(j)} < \eta n\}. 
$$
The goal of this subsection is to show
\begin{equation}\label{eq: last goal}
\underset{n \to \infty}{\lim} \sum_{F \in R_j} \bP(FM = 0) = 0,
\end{equation}
thereby finishing the proof of Theorem \ref{thm: universality moment}.
If $N_j^c$ is a finite set, this is a consequence of Proposition \ref{prop: BA for k converges to 0 cor}. So, we assume $N_j^c$ is an infinite set from now on. 
For a positive integer $m$ such that $j \le m \le k$, define
$$
N_j^c(m) := N_j^c \cap \{n : n-\beta_n^{(m)} < \eta n\}. 
$$

\begin{lem}
\label{lem: everything less than eta n}
Let $H_{k+1} = G$ and let $m$ be the largest positive integer such that $j \le m \le k$ and $H_m$ is a proper subgroup of $G$, i.e., $H_m \neq G$ and 
$$
H_{m+1} = \cdots = H_{k+1} = G. 
$$
Suppose that $N_j^c(m)$ is an infinite set. 
Then
$$
\lim_{\substack{n \in N_j^c(m) \\ n \to \infty}} \left|\bigcap_{i=j+1}^{k+1} A_{H_i}^{(i)}\right|
\left(\prod_{i=j+1}^{k+1} b_{H_i}^{\beta_n^{(i)} - \beta_n^{(i-1)}}\right) = 0. 
$$      
\end{lem}

We prove a special case of Lemma \ref{lem: everything less than eta n} first. 

\begin{lem}
\label{lem: everything less than eta n spcial case}
Assume all the conditions in Lemma \ref{lem: everything less than eta n}.
Suppose further that
$$
|H_{j+1}| < |H_{j+2}| < \cdots < |H_{m}| < |H_{m+1}| = |G|. 
$$
Then
$$
\lim_{\substack{n \in N_j^c(m) \\ n \to \infty}} \left|\bigcap_{i=j+1}^{k+1} A_{H_i}^{(i)}\right|
\left(\prod_{i=j+1}^{k+1} b_{H_i}^{\beta_n^{(i)} - \beta_n^{(i-1)}}\right) = 0. 
$$   
\end{lem}

\begin{proof}
Let $n \in N_j^c(m)$. Recall that we are assuming $n$ is large enough so that $n-\alpha_n^{(j)} > \beta_n^{(j)}$. Then the condition that 
$$
\beta_n^{(j)} < n - \alpha_n^{(j)} < \eta n < n - \eta n < \beta_n^{(m)}
$$
implies that $j + 1 \le m$ and $\beta_n^{(m)}- \beta_n^{(j)} > n(1-2\eta)$. 
We adopt the notation as in the proof of Lemma \ref{lem: special case1 for upper bound for all type2 case converges to 0 proposition}. We have for $j+1 \le i \le m$
$$
a_i = n-\alpha_n^{(i)} - \left(\lceil \ell(D_i)\dt_i(n - \alpha_n^{(i)})\rceil -1\right),
$$
and
$$
a_{j+1} \le a_{j+2} \le \cdots \le a_{m}.
$$
As in the proof of Lemma \ref{lem: special case2 for upper bound for all type2 case congverges to 0 proposition}, we have for sufficiently large $n$,
$$
\left|\bigcap_{i=j+1}^{k+1} A_{H_i}^{(i)}\right| \le |H_{j+1}|^{a_{j+1}}\left(\prod_{i = j+2}^m |H_i|^{a_i - a_{i-1}} \right)|G|^{n - a_m}e^{k \gamma n}.
$$
Then there exists a constant $C>0$ such that the following holds for sufficiently large $n$:
\begin{align*}
 \left|\bigcap_{i=j+1}^{k+1} A_{H_i}^{(i)}\right| & 
\left(\prod_{i=j+1}^{k+1} b_{H_i}^{\beta_n^{(i)} - \beta_n^{(i-1)}}\right) \\
&\le C(1-\epsilon)^{\beta_n^{(m)}- \beta_n^{(j)}} |G|^{kn\ell(|G|)\dt_{1}}\left(\prod_{i=j+1}^m\left(\frac{|H_i|}{|H_{i+1}|}\right)^{n- \alpha_n^{(i)} - \beta_n^{(i)}}\right)|H_{j+1}|^{\beta_n^{(j)}} e^{k\gamma n} \\
& \le Ce^{-\epsilon (1-2\eta) n} e^{k\gamma n}|G|^{kn \ell(|G|)\dt_1} |H_{j+1}|^{n - \alpha_n^{(j)}} \\
& \le C e^{-\epsilon (1-2\eta) n} e^{k\gamma n}|G|^{kn \ell(|G|)\dt_1} |G|^{\eta n}. 
\end{align*}
By our choice of the constants in \ref{The constants}, the right hand side converges to $0$, so the result follows.
\end{proof}

Now we give a proof of Lemma \ref{lem: everything less than eta n}.

\begin{proof}[Proof of Lemma \ref{lem: everything less than eta n}]
As noted in the proof of Lemma \ref{lem: everything less than eta n spcial case}, we must have $j+1 \le m$. We use induction on $m-j$. When $m-j = 1$, the assertion follows from Lemma \ref{lem: everything less than eta n spcial case}. 
Let $l$ be a positive integer such that $2\le l \le k-j$. Now we assume that the assertion holds when $m-j < l$. Suppose that $m-j = l$. 
If we have
$$
|H_{j+1}| < \cdots < |H_{m}|,
$$
we are done by Lemma \ref{lem: everything less than eta n spcial case}. Otherwise, there exists a positive integer $t$ such that $j+1 \le t \le m-1$ and $|H_{t}| \ge |H_{t+1}|$. 
Now we argue as in the proof of Proposition \ref{prop: upper bound for all type2 case converges to 0}. 
For every $1\le i \le k$, define
$$
\hat{H}_i:= \begin{cases}
H_i ~&\text{if $i < t$} \\
H_{i+1} ~ &\text{if $i \ge t$}
\end{cases}
$$
and 
$$
\hat{\alpha}_n^{(i)}:= \begin{cases}
\alpha_n^{(i)} ~&\text{if $i < t$} \\
\alpha_n^{(i+1)} ~&\text{if $i \ge t$}    
\end{cases}
$$
and
$$
\hat{\dt}_i :=  \begin{cases}
\dt_i ~&\text{if $i < t$} \\
\dt_{i+1} ~&\text{if $i \ge t$}   
\end{cases}
$$
and
$$
\hat{\beta}_n^{(i)}:= \begin{cases}
\beta_n^{(i)} ~&\text{if $i < t$} \\
\beta_n^{(i+1)} ~&\text{if $i \ge t$}
\end{cases}.
$$
As in the proof of Proposition \ref{prop: upper bound for all type2 case converges to 0}, it follows that
$$
\prod_{i=j+1}^{k+1} b_{H_i}^{\beta_n^{(i)} - \beta_n^{(i-1)}} \le \prod_{i=j+1}^{k} b_{\hat{H}_i}^{\hat{\beta}_n^{(i)} - \hat{\beta}_n^{(i-1)}}.
$$
Since we have
$$
A_{\hat{H}_i}^{(i)} = \begin{cases}
A_{H_i}^{(i)} ~&\text{if $i <t$} \\
A_{H_{i+1}}^{(i+1)} ~&\text{if $i \ge t$},
\end{cases}
$$
it is clear that
$$
\left|\bigcap_{i=j+1}^{k+1} A_{H_i}^{(i)}\right| \le \left|\bigcap_{i=j+1}^{k} A_{\hat{H}_i}^{(i)}\right|. 
$$
Now the lemma follows by the induction hypothesis.
\end{proof}

\begin{lem}
\label{lem: less than eta n proper subgroup}
Let $1 \le j \le k$ be a positive integer. Suppose that $N_j^c$ is an infinite set. Suppose that $H_{k+1}$ is a proper subgroup. Then
$$
\lim_{\substack{n \in N_j^c \\n \to \infty}} \left|\bigcap_{i=j+1}^{k+1} A_{H_i}^{(i)}\right|
\left(\prod_{i=j+1}^{k+1} b_{H_i}^{\beta_n^{(i)} - \beta_n^{(i-1)}}\right) = 0. 
$$  
\end{lem}

\begin{proof}
Let $n \in N_j$. Let us first consider a special case where
$$
|H_{j+1}| < |H_{j+2}| < \cdots < |H_{k+1}|.
$$
Similarly as in the proof of Lemma \ref{lem: everything less than eta n spcial case}, we have
$$
\left|\bigcap_{i=j+1}^{k+1} A_{H_i}^{(i)}\right|
\left(\prod_{i=j+1}^{k+1} b_{H_i}^{\beta_n^{(i)} - \beta_n^{(i-1)}}\right)
$$
is bounded above by
\begin{align*}
& C(1-\epsilon)^{n- \eta n} |G|^{n\ell(D_{k+1})\dt_{k+1}}\left(\prod_{i=j+1}^k \left(\frac{|H_i|}{|H_{i+1}|}\right)^{n- \alpha_n^{(i)} - \beta_n^{(i)}}|H_{i+1}|^{n\ell(|G|)\dt_{i}}\right)|H_{j+1}|^{\beta_n^{(j)}} e^{(k+1) \gamma n}  \\
& \le Ce^{-\epsilon n(1-\eta)} e^{(k+1)\gamma n}|G|^{n \ell(|G|)\dt_1 k}|G|^{\eta n},
\end{align*}
which converges to $0$ as $n \to \infty$ by the choice of the constants in \ref{The constants}. 
Now one can argue as in the proof of Lemma \ref{lem: everything less than eta n} (using induction on $k+1-j$) to complete the proof. 
We leave the detail to the reader. 
\end{proof}

Recall that
$$
R_j = B_{H_j}^{(j)} ~\bigcap~ \left(\bigcap_{i= j+1}^{k+1} A_{H_{i}}^{(i)} \right). 
$$

\begin{prop}
\label{prop: universality last prop}
Let $1\le j \le k$ be a positive integer. Let $H_j$ is a proper subgroup of $G$ and suppose that $N_j^c$ is an infinite set. Then
$$
\lim_{\substack{n \in N_j^c \\n \to \infty}} \sum_{F \in R_j} \bP(FM = 0) = 0.
$$
\end{prop}

\begin{proof}
Let $n \in N_j^c$. If $F \in R_j$, by Lemma \ref{lem: AA upper bound}(1) there exists a constant $C>0$ such that the following holds:
$$
\bP(FM= 0) \le \prod_{l = \beta_n^{(j)}+1}^n \bP(FM_l = 0) \le C\prod_{i=j+1}^{k+1} b_{H_i}^{\beta_n^{(i)} - \beta_n^{(i-1)}}.
$$
If $H_{k+1}$ is a proper subgroup of $G$, then the result follows from Lemma \ref{lem: less than eta n proper subgroup}. Therefore, for the rest of the proof we assume that $H_{k+1} = G$. Then there exists a positive integer $m$ such that $H_{m}$ is a proper subgroup of $G$ and
$$
H_{m+1} = \cdots = H_{k+1} = G. 
$$
Necessarily we have $j \le m \le k$. If $N_j^c\backslash N_j^c(m)$ is a finite set, then the desired result is a consequence of Lemma \ref{lem: everything less than eta n}.  Finally, suppose that $N_j^c\backslash N_j^c(m)$ is an infinite set. Let $n \in N_j^c\backslash N_j^c(m)$, i.e., $n - \beta_n^{(m)} \ge \eta n$ and $n - \alpha_n^{(j)} < \eta n$.
Define
\begin{align*}
 \tilde{H}_i &:= \begin{cases}
H_i ~&\text{if $i < m+1$} \\
G ~ &\text{if $i = m+1$}. 
\end{cases}   \\
\tilde{\dt}_i &:=\begin{cases}
\dt_i ~&\text{if $i < m+1$} \\
\dt_{k+1} ~ &\text{if $i = m+1$}. 
\end{cases}
\end{align*}
and define for $1 \le i \le m$
\begin{align*}
\tilde{\alpha}_n^{(i)} & := \alpha_n^{(i)} \\    
\tilde{\beta}_n^{(i)} & := \beta_n^{(i)}
\end{align*}
Let $\tilde{M}$ be a random $n\times n$ matrix having $m$ steps stairs of $0$ with respect to $\tilde{\alpha}_n^{(i)}$ and $\tilde{\beta}_n^{(i)}$. In other words, the random $n\times n$ matrix $\tilde{M}$ is defined by taking the first $m$-step stairs of $0$ of $M$ as the step stairs of $0$ of $\tilde{M}$. In particular, if $m = k$, then $M = \tilde{M}$. 
We define $B_{\tilde{H}_i}^{(i)}, A_{\tilde{H}_i}^{(i)}$ similarly as $B_{H_i}^{(i)}, A_{H_i}^{(i)}$ by replacing $\alpha_n^{(i)}, H_i, \dt_i$ with $\tilde{\alpha}_n^{(i)},  \tilde{H}_i, \tilde{\dt}_i$, respectively in the definition of $B_{H_i}^{(i)}, A_{H_i}^{(i)}$. 
We then have
$$
B_{H_j}^{(j)} ~\bigcap~ \left(\bigcap_{i = j+1}^{k+1} A_{H_i}^{(i)}\right) \subseteq B_{\tilde{H}_j}^{(j)} ~\bigcap~ \left(\bigcap_{i = j+1}^{m+1} A_{{\tilde{H}_i}}^{(i)}\right) =: \tilde{R}_j.
$$
Note that by Lemma \ref{lem: lem for code bound} and Lemma \ref{lem: exp bound lem} there exists a constant $C_1>0$ such that for all $F \in R_j$, the following holds for all large enough $n$:
$$
\prod_{i= \beta_n^{(m)} + 1}^n \bP(FM_i = 0) \le \prod_{i=m+1}^{k+1}\left(\frac{1}{|G|} + e^{-\epsilon \dt_i (n- \alpha_n^{(i)})/a^2} \right)^{\beta_n^{(i)} - \beta_n^{(i-1)}} \le C_1 \left(\frac{1}{|G|}\right)^{n - \beta_{n}^{(m)}}.
$$
Similarly there exists a constant $C_2>0$ such that for all $F \in R_j \subseteq \tilde{R}_j$ the following holds for sufficiently large $n$:
$$
\prod_{i= \beta_n^{(m)} + 1}^n \bP(F\tilde{M}_i = 0) \ge \left(\frac{1}{|G|} - e^{-\epsilon \dt_{k+1}n/a^2} \right)^{n - \beta_n^{(m)}} \ge C_2 \left(\frac{1}{|G|}\right)^{n - \beta_{n}^{(m)}}. 
$$
It follows that there exists a constant $C > 0$ such that for all $F \in R_j$ the following inequality holds for large enough $n$:

$$
\bP(FM = 0) \le C \bP(F\tilde{M} = 0). 
$$
Then we have that
$$
\sum_{F \in R_j} \bP(FM = 0) \le C \sum_{F \in \tilde{R}_j} \bP(F\tilde{M} = 0),
$$
so it is enough to show that the latter sum converges to zero. 
Note that
$$
n - \tilde{\beta}_n^{(m)} = n - \beta_n^{(m)} \ge \eta n.
$$
Then Corollary \ref{cor: changing row and column by transpose}(2) tells us that
$$
\lim_{\substack{n \in N_j^c\backslash N_j^c(m) \\n \to \infty }} \sum_{F \in \tilde{R}_j} \bP(F\tilde{M} = 0) = 0.
$$
Together with Lemma \ref{lem: everything less than eta n}, this implies that
$$
\lim_{\substack{n \in N_j^c\\n \to \infty }} \sum_{F \in \tilde{R}_j} \bP(F\tilde{M} = 0) = 0, 
$$
and this completes the proof. 
\end{proof}

\begin{proof}[Proof of Theorem \ref{thm: universality moment}]
Note again that $F \in \Sur(V,G)$ falls into one of the following three categories. 
\begin{enumerate}
\item
$F \in \mc{F}_1.$
\item 
At least one of $H_i$ is a proper subgroups of $G$ and
$$F \in \bigcap_{i=1}^{k+1} A_{H_i}^{(i)}.$$ 
\item 
For some $1\le j \le k$ with $H_j$ a proper subgroup of $G$
$$F \in R_j = B_{H_j}^{(j)} ~\bigcap~ \left(\bigcap_{i= j+1}^{k+1} A_{H_{i}}^{(i)} \right).$$     
\end{enumerate}
By Proposition \ref{prop: BA for k converges to 0 cor} and Proposition \ref{prop: universality last prop}, it follows that
\begin{equation}
\label{eq: Rj sum 0}
\underset{n \to \infty}{\lim} \sum_{F \in R_j} \bP(FM = 0) = 0.   
\end{equation}
Then the theorem follows by combining Proposition \ref{prop: code sum to 1} and Proposition \ref{prop: AA for k converges to 0}. 
\end{proof}

%-----------------------------------------
%-----------------------------------------
\section{The universality theorems for a random \texorpdfstring{$n \times (n+t)$}{n(n+u)} matrix} \label{Sec10}
Let $t$ be a non-negative integer. In this section, we first consider an $\epsilon$-balanced random $n \times (n+t)$ matrix over $R$ having $k$-step stairs of $0$.  
Let $k$ be a positive integer, and let $1 \le \alpha_n^{(k)} < \alpha_n^{(k-1)} < \cdots < \alpha_n^{(1)} \le n$ and $n+t \ge \beta_n^{(k)} > \beta_n^{(k-1)} > \cdots > \beta_n^{(1)} \ge 1$ be positive integers. In fact, Theorem \ref{thm: universality moment} can be generalized as follows:

\begin{thm}
\label{thm: universality moments with t}
Let $\mc{M}$ be an $\epsilon$-balanced random $n \times (n+t)$ matrix over $R$ having $k$-step stairs of $0$ with respect to $\alpha_n^{(i)}$ and $\beta_n^{(i)}$. If for every $1\le i \le k$
$$
\underset{n \to \infty}{\lim} (n - \alpha_n^{(i)} - \beta_n^{(i)}) = \infty,
$$
then for every finite abelian group $G$ whose exponent divides $a$, we have
$$
\underset{n \to \infty}{\lim}  \bE(\#\Sur(\cok(\mc{M}),G)) = \frac{1}{|G|^t}.
$$    
\end{thm}

\begin{proof}
If $t=0$, this is Theorem \ref{thm: universality moment}. Now let $t \ge 1$. 
Since we have
$$
\bE(\#\Sur(\cok(\mc{M}),G)) = \sum_{F \in \Sur(V,G)} \bP(F\mc{M} = 0), 
$$
it is enough to show that
$$
\underset{n \to \infty}{\lim}  \sum_{F \in \Sur(V,G)} \bP(F\mc{M} = 0) =  \frac{1}{|G|^t}.
$$
By $n - \alpha_n^{(k)} - \beta_n^{(k)} \to \infty$, we may assume that $\beta_n^{(k)} \le n$ when $n$ is large enough. Let $M$ be the $n \times n$ submatrix of $\mc{M}$ which consists of the first $n$ columns of $\mc{M}$. Then we can make use of the results in the previous three sections for $M$. As noted before, $F \in \Sur(V,G)$ falls into one of the following three categories. 
\begin{enumerate}
\item
$F \in \mc{F}_1.$
\item 
At least one of $H_i$ is a proper subgroups of $G$ and
$$F \in \bigcap_{i=1}^{k+1} A_{H_i}^{(i)}.$$ 
\item 
For some $1\le j \le k$ with $H_j$ a proper subgroup of $G$
$$F \in R_j = B_{H_j}^{(j)} ~\bigcap~ \left(\bigcap_{i= j+1}^{k+1} A_{H_{i}}^{(i)} \right).$$     
\end{enumerate}
Noting that the upper bound for the index $l$ of the following product is $n+t$ and not $n$
$$
\bP(F\mc{M} = 0) = \prod_{l=1}^{n+t} \bP(F\mc{M}_l = 0),
$$
we see that the proof of Proposition \ref{prop: code sum to 1} implies that
$$
\underset{n \to \infty}{\lim} \sum_{F \in \mc{F}_1} \bP(F\mc{M} = 0) = \frac{1}{|G|^t}. 
$$
Moreover, if $F \in \Sur(V,G)$, then 
$$
\bP(F\mc{M} = 0) = \prod_{l = 1}^{n+t} \bP(F\mc{M}_l = 0) = \bP(FM = 0) \prod_{l = n+1}^{n+t}\bP(F\mc{M}_l = 0) \le \bP(FM = 0). 
$$
Hence, Proposition \ref{prop: AA for k converges to 0} yields that if $H_i \neq G$ for some $H_i$,
$$
\underset{n \to \infty}{\lim}  \sum_{F \in \cap_{i=1}^{k+1} A_{H_i}^{(i)}} \bP(F\mc{M} = 0) = 0. 
$$
Similarly by \eqref{eq: Rj sum 0}, we have
$$
\underset{n \to \infty}{\lim}  \sum_{F \in R_j} \bP(F\mc{M} = 0) = 0.
$$
This completes the proof of the theorem. 
\end{proof}

\begin{rmk}
Recall that we used a ``transpose'' argument for bounding the error terms for the moments in the case of square matrix ($t =0$). This works because if $\mc{M}$ is a square matrix over $R$, we have (Lemma \ref{lem: Smith normal form})
$$
\cok(\mc{M}) \cong \cok(\mc{M}^T),
$$
which fails if $t$ is a positive integer. This is the main reason why we were unable to work directly with $n \times (n+t)$ matrix.
\end{rmk}

The following two theorems are consequences of Theorem \ref{thm: universality moments with t}, \cite[Theorem 3.1]{Woo19} and \cite[Lemma 3.2]{Woo19}.

\begin{thm}
\label{thm: universality main theorem with t}
Let $t$ be a non-negative integer and $M$ be an $\epsilon$-balanced random $n \times (n+t)$ matrix over $\Z_p$ having $k$-step stairs of zeros with respect to $\alpha_n^{(i)}$ and $\beta_n^{(i)}$. Suppose that for every $1 \le i \le k$, 
$$
\underset{n \to \infty}{\lim} (n - \alpha_n^{(i)} - \beta_n^{(i)}) = \infty.
$$
Then for every finite abelian $p$-group $G$, we have
$$
\underset{n \to \infty}{\lim}\bP(\cok(M) \cong G) = \frac{1}{|\Aut(G)||G|^t}\prod_{i = 1}^\infty (1- p^{-i-t}).
$$
\end{thm}

For a finite abelian group $G$ and a prime $p$, we write $G(p)$ for the Sylow $p$-subgroup of $G$. 

\begin{thm}
Let $t$ be a non-negative integer and $M$ be an $\epsilon$-balanced random $n \times (n+t)$ matrix over $\Z$ having $k$-step stairs of zeros with respect to $\alpha_n^{(i)}$ and $\beta_n^{(i)}$. Suppose that for every $1 \le i \le k$, 
$$
\underset{n \to \infty}{\lim} (n - \alpha_n^{(i)} - \beta_n^{(i)}) = \infty.
$$
Let $G$ be a finite abelian group and $T$ be a finite set of primes containing all prime divisors of $|G|$. Then we have
$$
\underset{n \to \infty}{\lim}\bP(\cok(M)(p) \cong G(p) \text{ for all $p \in T$}) = \frac{1}{|\Aut(G)||G|^t}\prod_{p \in T}\prod_{i = 1}^\infty (1- p^{-i-t}).
$$
\end{thm}

%----------------------
\section*{Acknowledgments}
The authors are grateful to the anonymous referees for their helpful comments and suggestions.
We are very grateful to Gilyoung Cheong for many helpful discussions and valuable comments on the exposition of this paper.
We thank András Mészáros for pointing out a gap in the proof of Theorem 4.1 in the previous version and suggesting to use the tail distribution. 
We also thank Joonkyung Lee for helpful comments. 

% Funding information
Dong Yeap Kang was supported by Institute for Basic Science (IBS-R029-Y6).
Jungin Lee was supported by the National Research Foundation of Korea (NRF) grant funded by the Korea government (MSIT) (No. RS-2024-00334558 and No. RS-2025-02262988).
Myungjun Yu was supported by the National Research Foundation of Korea (NRF) grant funded by the Korea government (MSIT) (No. 2020R1C1C1A01007604) and by Yonsei University Research Fund (2024-22-0146).

%----------------------

\end{document}